\documentclass[reqno, 11pt]{amsart} 
\usepackage[utf8]{inputenc}
\usepackage{amssymb}
\usepackage{amsthm}
\usepackage{amsfonts}
\usepackage{amsmath}
\usepackage{dsfont}
\usepackage{hyperref}
\hypersetup{
   colorlinks=true,
   linkcolor=blue,
}
\usepackage{IEEEtrantools}
\usepackage{enumerate}
\usepackage{mathtools}
\usepackage{physics}
\usepackage{tikz}
\usepackage{mathdots}
\usepackage{yhmath}
\usepackage{cancel}
\usepackage{color}
\usepackage{siunitx}
\usepackage{array}
\usepackage{multirow}
\usepackage{amssymb}
\usepackage{tabularx}
\usepackage{booktabs}
\usepackage{graphicx}
\graphicspath{ {./simulacoes/} }
\usepackage[normalem]{ulem} 
\usepackage{soul}
\usepackage{extarrows}

\usetikzlibrary{patterns}
\usetikzlibrary{shapes}
\usepackage{pdfpages}
\usepackage{refcount}
\usepackage{comment}

\usetikzlibrary{automata}
\usetikzlibrary{shapes,calc}
\usetikzlibrary{arrows.meta,arrows}
\usetikzlibrary{positioning}
\usepackage{caption}
\usepackage{subcaption}

\makeatletter
\def\l@subsection{\@tocline{2}{0pt}{2.5pc}{5pc}{}}
\renewcommand\tocchapter[3]{%
  \indentlabel{\@ifnotempty{#2}{\ignorespaces#2.\quad}}#3%
}
\newcommand\@dotsep{4.5}
\def\@tocline#1#2#3#4#5#6#7{\relax
  \ifnum #1>\c@tocdepth % then omit
  \else
    \par \addpenalty\@secpenalty\addvspace{#2}%
    \begingroup \hyphenpenalty\@M
    \@ifempty{#4}{%
      \@tempdima\csname r@tocindent\number#1\endcsname\relax
    }{%
      \@tempdima#4\relax
    }%
    \parindent\z@ \leftskip#3\relax \advance\leftskip\@tempdima\relax
    \rightskip\@pnumwidth plus1em \parfillskip-\@pnumwidth
    #5\leavevmode\hskip-\@tempdima{#6}\nobreak
    \leaders\hbox{$\m@th\mkern \@dotsep mu\hbox{.}\mkern \@dotsep mu$}\hfill
    \nobreak
    \hbox to\@pnumwidth{\@tocpagenum{#7}}\par
    \nobreak
    \endgroup
  \fi}
\makeatother
\AtBeginDocument{%
\makeatletter
\expandafter\renewcommand\csname r@tocindent0\endcsname{0pt}
\makeatother
}
\def\l@subsection{\@tocline{2}{0pt}{2.5pc}{5pc}{}}

\newtheorem{hypothesis}{Hypothesis}

\newtheorem{lemma}{Lemma}[section]
\newtheorem{proposition}{Proposition}[section]

\newtheorem{condition}{Condition}[section]
\newtheorem{conditionbis}{Condition}[section]
\newtheorem{corollary}{Corollary}[section]

\renewcommand\theconditionbis{\Roman{conditionbis}}

\newtheorem{theorem}{Theorem}[section]

\newtheorem{remark}{Remark}[section]

\newtheorem{conjecture}{Conjecture}[section]
\newcommand{\PP}{\mathbb{P}}
\newcommand{\EE}{\mathbb{E}}
\newcommand{\ZZ}{\mathbb{Z}}
\newcommand{\FF}{\mathcal{F}}
\newcommand{\Rr}{\mathcal{R}}

\newcommand{\CC}{\mathcal{C}}
\newcommand{\Rs}{\mathbb{R}}
\newcommand{\um}{\mathds{1}}

\newcommand{\Name}{{\rm GERW}}

\newcommand{\Nametwo}{{\rm ERW}}

\newcommand{\com}[1]{\textcolor{red}{\texttt{giulio:} #1}}
\newcommand{\comu}[1]{\textcolor{cyan}{\texttt{Glauco:} #1}}

\usepackage{chngcntr}
\counterwithin{figure}{section}

\makeatletter
\@addtoreset{equation}{section}
\makeatother

\renewcommand\thetable{\thesection.\@arabic\c@table}

\title[excited random walk under Bernoulli excitations]{
Generalized Excited Random Walks under Bernoulli excitations}

\author[R. Alves, G. Iacobelli, G. Valle, L. Zuaznabár]{Rodrigo B. Alves$^1$, Giulio Iacobelli$^2$,\\ Glauco Valle$^3$, Leonel Zuazn\'abar$^4$ }
\thanks{1. Supported by CAPES}
\thanks{2. Supported by FAPERJ  grant E-26/210.516/2024.}
\thanks{3. Supported by CNPq grant 307938/2022-0 and FAPERJ grant E-26/202.636/2019.}

\address{
\newline
\newline
Instituto de Matemática e Estatística, Universidade Estadual do Rio de Janeiro.
\newline Caixa postal 20550900, Rio de Janeiro, Brasil
\newline
$^1$ e-mail: {\rm \texttt{rodrigo.alves@ime.uerj.br}}
}

\address{
\newline
\newline
Instituto de Matem\'atica, Universidade Federal do Rio de Janeiro.
\newline  Caixa Postal 68530, Rio de Janeiro, RJ, Brasil
\newline
$^2$ e-mail: {\rm \texttt{giulio@im.ufrj.br}}
\newline
$^3$ e-mail: {\rm \texttt{glauco.valle@im.ufrj.br}}
}

\address{
\newline
\newline
Centro de Matem\'atica, Computa\c{c}\~ao e Cogni\c{c}\~ao, Universidade Federal do ABC.
\newline Santo André, SP, Brasil.
\newline
$^4$ e-mail: {\rm \texttt{l.zuaznabar@ufabc.edu.br}}
}


\begin{document}

\maketitle

\begin{abstract}

We study a variant of the Generalized Excited Random Walk (\Name) on $\mathbb{Z}^d$ introduced by Menshikov et al. in \cite{menshikov2012general},  
% It consists in a particular version of the model studied in \cite{alves2022note} 
where excitation may or may not occur according to a time-dependent probability.  Specifically, given a sequence $\{p_n\}_{n \ge 1}$, %$p_n \in (0, 1]$
% for all $n \ge 1$
%,
if at time $n-1$ the process visits a site for the first time, then with probability $p_{n}$ it gains a drift in a fixed direction. Otherwise, it behaves as a $d$-dimensional martingale with zero-mean vector. We refer to the model as
% a \Name{} in time-inhomogeneous Bernoulli environment, in short, 
$p_n$-\Name{}. 

%Otherwise, with probability $1-p_n$, it  behaves as a $d$-martingale with zero-mean vector.  Whenever the process visits an already-visited site, the process acts again as a $d$-martingale with zero-mean vector. We refer to the model as a \Name{} in Bernoulli environment, in short $p_n$-\Name{}.
Assuming bounded jumps  and %a sequence $\{p_n\}_{n \ge 1}$ which decays polynomially, namelly $p_n = \mathcal{C}n^{-\beta} \wedge 1$ with $\beta > 0$ and $\mathcal{C}$ is a positive constant, 
$p_n \approx n^{-\beta}$,  
we show a series of results for the $p_n$-\Name{}  depending on the value of $\beta$ and on the dimension $d$. Specifically, for every $\beta\in(0,1]$ and $d=2$ or $d \ge h(\beta)$, with $h$ a decreasing function of $\beta$, we prove a SLLN for the range, while for $\beta<1/2$ we prove a sub-ballistic SLLN for the process whenever the SLLN for the range holds. We also study the $p_n$-\Name{} under diffusive scaling and we obtain a Functional Central Limit Theorem for $\beta > 1/2$ and $d\geq 2$, or  $\beta=1/2$ and $d=2$. Finally,  for  $\beta=1/2$  and $d \ge 11$ we show that the diffusively rescaled $p_n$-\Name{} converges in distribution to a Brownian Motion plus a multiple of the square root of time.
\end{abstract}

% If you want the subsection to appear in the table of content use 
%\setcounter{tocdepth}{3}

\tableofcontents

%\pagenumbering{arabic}

\setcounter{tocdepth}{2}
\section{Introduction}

The multi-dimensional Generalized Excited Random Walk (\Name) was introduced by Menshikov et al. in~\cite{menshikov2012general} following a series of works on multi-dimensional excited random walk \cite{benjamini2003excited,kozma2003excited,berard2007central}. The model considered in ~\cite{menshikov2012general} is a  uniformly elliptic, self-interacting,  random walk with bounded jumps in dimension $d\ge 2$ which behaves as follows:   on already visited sites it behaves as a $d$-dimensional martingale with bounded jumps and zero-mean vector, whereas whenever a site is visited for the first time its increment has a drift in a fixed direction $\ell$ of the unit sphere in $\mathbb{R}^d$. They show that the GERW with a drift condition in direction $\ell$ is ballistic in that direction. Besides that, they proved a Law of Large Numbers and a Central Limit Theorem (both for dimensions $d \geq 2$) under stronger hypotheses on the definition of GERW; these particular models were called \textit{excited random walk in random environment}. 

 The GERW is an important class of toy models of non-Markovian random walks used to understand how weakly, or rarely, the random walk needs to be pushed in a fixed direction to still exhibit a ballistic behavior. In \cite{alves2022note} we have studied a variation of the GERW  considering that the strength of the drift on first visits decreases with time. Specifically, we assume that the drift in a fixed direction $\ell$ at time $n$, say $\lambda_n$, is of order $n^{-\beta}$ (if at this time a site is visited for the first time). We called this variation  $\lambda_n$-\Name{}. If $\{X_n\}_{n\geq 0}$ denotes this $\lambda_n$-\Name{}, we show in \cite{alves2022note} that $\lim_{n \rightarrow \infty} X_n \cdot \ell = \infty$ (directional transience) with positive probability if $\beta$ is sufficiently small. This shows that the $\lambda_n$-\Name{}  has an intermediary behavior between a zero-mean random walk (non-excited) and the \Name{} considered in \cite{menshikov2012general}. Indeed, for  $\lambda_n$ of order  $n^{-\beta}$, the  $\lambda_n$-\Name{} is not ballistic, since the total mean drift accumulated by time $n$ is bounded by  $n^{1-\beta}$.
 The $\lambda_n$-\Name{} is well motivated since for the multi-dimensional excited random walk a relevant question is if it is still possible to guarantee properties such as directional transience and ballisticity by reducing the number of times it has a drift in a given direction.
 
In this paper, we discuss limit theorems 
% for $\beta \ge 1/2$ 
for a particular class of $\lambda_n$-\Name{}, in the same spirit as the {\em excited random walk in random environment} discussed in 
\cite{menshikov2012general} (and mentioned above). Let us point out that the model considered in this paper is neither a generalization nor a particular case of the excited random walk under Bernoulli excitations. Specifically, given a sequence of parameters $\{p_n\}_{n \geq 1}$ with $p_n \in (0,1]$ for all $n \geq 1$, if at time $n-1$ the process visits a site for the first time, it becomes excited with probability $p_n$ (thus gaining a drift). Otherwise, with probability $1-p_n$, it does not gain drift and behaves as a $d$-dimensional martingale with zero-mean vector. If, instead, the process has already visited the site, there is no excitation and the process acts again as a $d$-dimensional martingale with zero-mean vector. We call this model $p_n$-\Name{}. %We will focus on two specific cases $\{p_n\}_{n \geq 1}$: 
We will consider $p_n = \mathcal{C}n^{-\beta} \wedge 1$ with $\beta >0$ and $\mathcal{C}$ a positive constant.  Depending on the value of $\beta$ and on the dimension $d$, we obtain the following results:   
\begin{itemize}
    \item For $\beta \in (0,1]$ and $d=2$ or $d \ge \max\{9, \lfloor 2 + 4/\beta \rfloor + 1\}$, we prove a Strong Law of Large Numbers for the range of a $p_n$-\Nametwo{} (see, Theorem~\ref{thm:SLLN}). For $\beta>1$, the SLLN for the range  holds for $d\geq 2$. The $p_n$-\Nametwo{} is a special type of $p_n$-\Name{} (see,  Section~\ref{sec:p_n-ERW} for its definition).
    \item For $\beta <1/2$ and $d\geq2$, if the SLLN for the range holds,  we prove a sub-ballistic Strong Law of Large Numbers for the $p_n$-\Nametwo{} (see, Theorem~\ref{thm:conv-Xbeta}). 
    \item For $\beta$ sufficiently small and  $d\geq 2$,   the $p_n$-\Name{} has a positive probability of never returning to the origin in a fixed  direction $\ell$. This is proved in \cite{alves2022note}.
    \item For  $\beta > 1/2$ and $d\geq 2$,  the $p_n$-\Name{}, under diffusive  rescaling,  converges in distribution to a Gaussian Process (see, Proposition~\ref{pn-WGERW-Gauss}). 
    % Note that, for $\beta>1$ the $p_n$-\Name{} only steps with a drift finitely many times almost surely and therefore the process eventually behaves as a $d$-dimensional martingale; in this case, the convergence in distribution follows directly from \cite[Theorem 7.1.4]{ethier2009markov}.
    \item For $\beta = 1/2$ and $d \ge 3$, if the SLLN for the range holds,     the $p_n$-\Nametwo{},  under diffusive scaling,  converges in distribution  to a Brownian Motion plus a multiple of square root of time  (see, Theorem~\ref{thm:main-conv}).
    % \begin{itemize}
    %     \item[$\ast$] If $d=2$ the $p_n$-\Nametwo{}, which is a special type of $p_n$-\Name{} (see,  Section~\ref{sec:p_n-ERW} for its definition), converges in distribution,  under diffusive scaling,  to a Brownian Motion (see, Theorem~\ref{pn-ERW-d=2}).
    %     %
    %     \item[$\ast$] If $d \ge 3$,  the $p_n$-\Nametwo{} under  diffusive scaling  is  tight and  every limit point is stochastically dominated, in direction $\ell$, by a Brownian Motion plus a  multiple of square root of time  (see, $(a)$ in Theorem~\ref{pn-ERW-d=>4}).
    %     \item[$\ast$] If $d \ge 22$, the $p_n$-\Nametwo{}  under diffusive scaling   is  tight and  every limit point stochastically dominates, in direction $\ell$, a Brownian Motion plus a  multiple of square root of time (see, $(b)$ in Theorem~\ref{pn-ERW-d=>4}). 
    % \end{itemize}
\end{itemize}

For $\beta\in (0,1]$ and $d < \max\{9, \lfloor 2 + 4/\beta \rfloor + 1\}$, we conjecture that the SLLN for the range of the $p_n$-\Nametwo{} should hold, but in this case we only provide an upper bound.  As a consequence, we can show that any limit point of the $p_n$-\Nametwo{} under diffusive scaling is stochastic bounded from above in the direction of the drift by a Brownian Motion plus a multiple of the square root of time.

% \com{Maybe it is worthwhile to include a sentence commenting on the role of the two parameters $d$ and $\beta$. For example, for $d$ fixed, intuitively for large $\beta$ the model behaves like a SSRW, while for $\beta$ small the model behaves like the ERW. For $\beta$ fixed, when $d=2$ the model should behave like a SSRW, while for $d\geq 3$ the model does not really compare neither to the SSRW nor the ERW.....\texttt{we should also provide some heuristics addressing this question of one reviewer "...the results should be easier to prove in higher dimensions, right?". }  Do we have any clue about the result concerning the lower bound on the range?}
%\sout{Moreover, for $\beta<1/2$ and $d\geq 2$, it is also conjectured that the $p_n$-\Name{} has a positive probability of never returning to the origin, see discussion in \cite{alves2022note}. }

\medskip 
The paper is organized as follows: In Section~\ref{sec:model} we formally define the $p_n$-\Name{},  as well as the $p_n$-\Nametwo{},  and state the main theorems concerning the different asymptotic behaviors  depending on the model ($p_n$-\Name{} or $p_n$-\Nametwo{}), the value of $\beta$ and the dimension $d$. Section~\ref{sec:rangeERW}  is devoted to the proof of the SLLN for the range. 
In Section~\ref{sec:SLLN_X} we provide the proof of the sub-ballistic SLLN for the $p_n$-\Nametwo{}. 
Finally, Section~\ref{resultados_pn} contains the proofs for the  convergence in distribution under diffusive rescaling. 

\section{Formal definitions and main results}\label{sec:model}
Let $d \ge 2$ be the fixed dimension, $\{\xi_i, \FF_i\}_{i \geq 1}$ be the increments of a $d$-martingale  
(i.e., a $d$ dimensional process where each component is a martingale) 
with zero-mean vector and $\{\gamma_i, \FF_i\}_{i\geq 1}$ be a $\ZZ^d$ random vector,  both adapted processes on a stochastic basis $(\Omega,\mathcal{F},\PP,\{ \mathcal{F}_n \}_{n \geq 0})$ where $\FF_0$ contains all the $\PP$-null sets of $\FF$. %satisfying the usual conditions \cm{i.e., the filtration is complete and right-continuous, ou devemos definir a filtração?} \com{list the conditions and give a reference}.
We denote the expectation with respect to $\PP$ by $\mathbb{E}$ and the Euclidean norm in $\mathbb{R}^d$ by $||\cdot||$. Let $\ell \in \mathbb{S}^{d-1}$ be a direction, where $\mathbb{S}^{d-1}$ is the unit sphere of $\mathbb{R}^d$. We assume the following conditions:

\begin{condition}\label{condition_A}
There exists a positive constant $K$ such that
\[
\sup_{n \geq 1} || \xi_n || \le K \quad \text{  and } \quad  \sup_{n \ge 1}||\gamma_n|| \le K\;,
\]
on every realization.    
\end{condition}

\begin{condition}\label{condition_B}
For every $n \ge 1$,  we have that 
\[
\EE[\xi_n| \FF_{n-1}] = 0 \quad  \text{ and } \quad \EE[\gamma_n \cdot \ell | \FF_{n-1}] \ge \lambda\;,
\]
 where $\lambda$ is a positive constant.    
\end{condition}
 
Let $\{U_i\}_{i \geq 1}$ be a sequence of i.i.d. random variables with uniform distribution in $[0,1]$ independent of the $\{\xi_i\}_{i\geq 1}$ and $\{\gamma_i\}_{i\geq 1}$, and $\{p_n\}_{n \ge 1}$ be a sequence  such that $p_n \in (0, 1]$ for all $n \ge 1$. 
The $p_n$-\Name{} is a $\ZZ^d$-valued process $X=\{X_n\}_{n \ge 0}$.  
Setting   $E_0:= \emptyset$ and  $X_0 = 0$,  we define $\{X_n\}_{n \geq 1}$ recursively as follows: given  $\{X_{k}\}_{k=0}^{n-1}$, denote $E_{n-1}:=\{ \exists\;  k < n-1 \; \text{ such that }\;  X_k = X_{n-1} \}$, i.e.,  $E_{n-1}$ corresponds to the event that the process $\{X_n\}_{n \geq 0}$ is, at time $n-1$,  in an already visited site. Then, we define $X_n$ as   
\begin{align*}%\label{xn-incremnto1}
X_n := X_{n-1} +  \um_{E_{n-1}} \xi_n + \um_{E_{n-1}^c \cap \{U_{n} > p_{n}\}}  \xi_n + \um_{E_{n-1}^c \cap \{ U_{n} \leq p_{n}\}} \gamma_n  \,, 
\end{align*}
where $\um_E$ denotes the indicator function of set $E$. Note that the process can be rewritten as 
\begin{align}\label{xn-incremnto1}
X_n := \sum_{i=1}^n \big( \um_{E_{i-1}} \xi_i + \um_{E_{i-1}^c \cap \{U_{i} > p_{i}\}}  \xi_i + \um_{E_{i-1}^c \cap \{ U_{i} \leq p_{i}\}} \gamma_i \big) , \ n\ge 1 \,.
\end{align}

% Setting   $E_0:= \emptyset$ and, 
% for $i\ge 1$,  $E_i:= \{ \exists\;  k < i \; \text{ such that }\;  X_k = X_i \}$ (i.e.,  $E_i$ corresponds to the event that the process $\{X_n\}_{n \geq 0}$ is, at time $i$,  in an already visited site),   we define  $\{X_n\}_{n \geq 0}$ recursively as: $X_0 = 0$ and 
% %
% \begin{align}\label{xn-incremnto1}
% X_n := \sum_{i=1}^n \big( \um_{E_{i-1}} \xi_i + \um_{E_{i-1}^c \cap \{U_{i} > p_{i} \}} \xi_i + \um_{E_{i-1}^c \cap \{ U_{i} \leq p_{i}\}} \gamma_i \big) , \ n\ge 1 \,, 
% \end{align}
% where $\um_E$ denotes the indicator function of set $E$.
% %
The $p_n$-\Name{} can be interpreted as follows: if at time $n-1$ the process visits a site for the first time, it becomes excited with probability $p_{n}$ (thus gaining a drift). Otherwise, with probability $1-p_{n}$, it does not become excited (no drift) and behaves as a $d$-martingale with zero-mean vector. If the process has already visited the site, it is not excited and the process acts again as a $d$-martingale with zero-mean vector.

\subsubsection*{\bf A special case of $p_n$-\Name{} ($p_n$-\Nametwo{})}\label{sec:p_n-ERW}  
A particular type of $p_n$-\Name{}, which we call $p_n$-\Nametwo{} in the direction $\ell$, is obtained if we further assume that  $\{\xi_i\}_{i \ge 1}$ is an i.i.d. sequence on $\ZZ^d$ with zero-mean vector and finite covariance matrix. Additionally, we assume $\PP[ \xi_i \cdot \ell' = 0 ] < 1$ for all $i \ge 0$ and for all $\ell'\in \mathbb{S}^{d-1}$ (i.e., the increments $\xi_i$ are genuinely  $d$-dimensional). Moreover, we assume that the sequence $\{\gamma_i\}_{i \ge 1}$ is also i.i.d. on $\ZZ^d$  (recall that $\gamma_i$ satisfies  $\EE[\gamma_i \cdot \ell |\FF_{i-1}] \ge \lambda$), 
and that the sequences $\{\xi_i\}_{i \ge 1}$ and $\{\gamma_i\}_{i \ge 1}$ are independent.
 
\medskip
\noindent
{\it Example of $p_n$-\Nametwo{}:}  
Here we provide a concrete example of a $p_n$-\Nametwo{} on $\ZZ^d$ which may be thought of as a generalization of the classical ERW~\cite{benjamini2003excited}. Specifically, it evolves as the classical ERW but when a site is visited for the first time, it  gains a drift with probability $p_n$.   
Let $\ell = e_1$ (drift direction) and let $\{\xi_i\}_{i \ge 1}$ and $\{\gamma_i\}_{i \ge 1}$ be i.i.d. sequences of discrete random variables taking values on $\{\pm e_1, \ldots, \pm e_d\}$ (canonical directions).  We denote by $q_{\xi}$ and $q_{\gamma}$, the probability distributions associated   with $\{\xi_i\}_{i \ge 1}$ and $\{\gamma_i\}_{i \ge 1}$, respectively. 
We assume that $q_{\xi}(e_j)=q_{\xi}(-e_j)=\frac{1}{2d}$, for all $j \in \{1, \ldots, d\}$. As regards to $q_{\gamma}$, fixing $\delta \in (\frac{1}{2},1]$, we assume that  $q_{\gamma}(e_1)= \frac{\delta}{d}$, $q_{\gamma}(-e_1)=\frac{1-\delta}{d}$ and  $q_{\gamma}(e_j)=q_{\gamma}(-e_j)=\frac{1}{2d}$, for all $j \in \{2, \ldots, d\}$. Note that, with the above choice, we obtain that $\EE[\gamma_i \cdot e_1 |\FF_{i-1}]= \EE[\gamma_1 \cdot e_1] = 
% \frac{\delta}{d} - \frac{1-\delta}{d}=
\frac{2\delta -1}{d}=:\lambda >0$ (the first equality follows from the i.i.d. hypothesis and the fact that the sequences $\{\gamma_i\}_{i \ge 1}$ and $\{\xi_i\}_{i \ge 1}$ are independent). 
%for all $n \geq 1$.   
% Next we will give an example of a process which respect the conditions to be a $p_n$-\Name{}. 
%Here we suppose that $\pi$ have iid marginals. 
% Fix $\delta \in (1/2,1]$ and let $q^{(0)}(x, e_i)$, $x \in \ZZ^d$, $i=1,...,d$, be defined as
% \begin{equation*}\label{prob-ERW}
% \begin{split}
% q^{(0)}(x, e_1) & = \delta/d , \ \,
% q^{(0)}(x, -e_1)  = (1-\delta)/d\;, 
% \\
% q^{(0)}(x, \pm e_i) & = 1/2d \quad \text{for all } i = 2,\dots, d\;,
% \end{split}
% \end{equation*}
% and $q^{(1)}(x, e_i)$, $x \in \ZZ^d$, $i=1,...,d$, be the transition probabilities of a SRW in $\ZZ^d$.
The corresponding  $p_n$-\Nametwo{} (see, \eqref{xn-incremnto1}) is a process $\{X_n\}_{n \geq 0}$ on $\ZZ^d$ with transition probabilities 
\begin{equation*}
\begin{split}
\PP & \Big[X_{n+1} = x + e_i  \Big\vert X_n = x, \sum_{j=0}^{n-1} \um_{\{X_
j = x\}} = 0 \Big] = 
\\
& = \um_{\{U_n \leq p_n\}} q_{\gamma}(e_i) +   \um_{\{U_n > p_n\}} q_{\xi}(e_i) \,, 
\end{split}
\end{equation*}
and for every $m \in \{1,2,\dots, n-1\}$ we have 
\begin{equation*}
\PP \Big[X_{n+1} = x + e_i  \Big\vert X_n = x, \sum_{j=0}^{n-1}\um_{\{X_
j = x\}} = m \Big] = q_{\xi}(e_i) \,, 
\end{equation*}
Note that, when $p_n =1$ for all $n \ge 1$, we recover the classical ERW.
\hfill $\diamond$

\medskip
\noindent
{\it Relation between  $p_n$-\Name{} and the $\lambda_n$-\Name{} introduced in~ \cite{alves2022note}.}  
We show here that a $p_n$-\Name{} can be thought of as a special type of a $\lambda_n$-\Name{} if we further impose a uniform elliptic condition.
As can be seen in \cite{alves2022note}, the $\lambda_n$-\Name{} is also defined in terms of three conditions (along the lines of \cite{menshikov2012general}), which are summarised below: 
\begin{enumerate}[$i)$]
    \item $\exists K>0$ such that $\sup_{n \geq 0} || X_{n+1} - X_{n} || < K$ on every realization; 
    \item almost surely, on $\{ X_k \neq X_n \, \forall \; k < n \}$, $\mathbb{E} [ X_{n+1} - X_n | \mathcal{F}_n] \cdot \ell  \geq \lambda_n$, and 
    \\
    on $\{ \exists\,  k < n  \text{ such that }  X_k = X_n \}$,
 $\mathbb{E} [ X_{n+1} - X_n | \mathcal{F}_n] = 0$;
 \item  There exist $h, r > 0$ such that: 
 for all $n$, 
 \\ $\PP \left[ \left( X_{n+1} - X_n \right) \cdot \ell > r | \mathcal{F}_n \right] \geq h$, a.s., and 
 \\
 on the event $ \{ \mathbb{E} [ X_{n+1} - X_n | \mathcal{F}_n] = 0 \}$,  for all $\ell' \in \mathbb{S}^{d-1}$,  with $|| \ell '|| = 1$,
$\PP \left[ \left( X_{n+1} - X_n \right) \cdot \ell ' > r | \mathcal{F}_n \right] \geq h$, a.s..
\end{enumerate}
Now if $\{X_n\}_{n \ge 0}$ is a $p_n$-\Name{} (see, \eqref{xn-incremnto1}), Condition~\ref{condition_A}  implies $i)$. As far as $ii)$ is concerned, if we denote by $\widetilde{\FF}_n=\sigma(X_1, \ldots, X_n)$, by integrating 
$\mathbb{E} [ X_{n+1} - X_n | \mathcal{F}_n] \cdot \ell$ with respect to $U_1, \ldots, U_{n+1}$ and using Condition~\ref{condition_B}, we obtain $\mathbb{E} [ X_{n+1} - X_n | \widetilde{\FF}_n] \cdot \ell\geq p_{n+1}\lambda$.  Thus $ii)$  is satisfied with $\lambda_n = p_{n+1}\lambda$.
Condition $iii)$ is a uniform elliptic condition which we do not assume for $p_n$-\Name{}; thus in order to compare a $p_n$-\Name{} with a $\lambda_n$-\Name{}, this condition must be further imposed. 
\begin{remark}
Note that if $\{X_n\}_{n \ge 0}$ is a $p_n$-\Name{} with $p_{n+1} = p$ where $p \in (0, 1]$ for all $n \ge 1$ (homogeneous) which also satisfies the uniform elliptic condition $iii)$, we obtain the~\Name 
~defined in~\cite{menshikov2012general}. Thus, from the result in~\cite{menshikov2012general}, we obtain that the homogeneous $p_n$-\Name{} satisfying $iii)$ is ballistic  in the direction of the drift. 
\end{remark} 

\medskip

In this paper we study the $p_n$-\Name{} when the sequence $\{p_n\}_{n \ge 1}$ is of the form $p_n =\mathcal{C} n^{-\beta} \wedge 1$, with $\beta>0$  and $\mathcal{C}$ is a positive constant.
 When $\beta>1/2$, we may relax the bounded jump assumption of  Condition~\ref{condition_A};   namely, we may just assume the following condition: 
 \renewcommand{\theconditionbis}{\Roman{conditionbis}*}
\setcounter{conditionbis}{0}
\begin{conditionbis}\label{condiçaoI*}
For all $k \geq 1$ and $\theta < \beta - 1/2$, where $\beta> 1/2$, we have
    \[\sup_{k \geq 1} \frac{\EE[\| \xi_k \|]}{k^{\theta}} < \infty  \quad \text{ and } \quad
    \sup_{k \geq 1} \frac{\EE[\| \gamma_k \|]}{k^{\theta}} < \infty \;. \]
    % \item [ii)]  It holds that
    % \begin{equation}\label{condGaussiano}
    % \frac{1}{n}\sum_{i=1}^{\lfloor nt \rfloor} \xi_i \xi_i^T \to C(t) \quad \text{as } n \to \infty \;,
    % \end{equation}
    % in probability and
    % \begin{equation*}
    % \lim_{k \to \infty} k^{-1/2} \EE\left[ \sup_{1 \leq i \leq k} \| \xi_i \| \right]  = 0\,. 
%\end{equation*}     %\item [ii)] For a large $n$ we have,
    %\[ \sum_{k = 1}^{\infty} \PP[\| X_k - X_{k-1} \| > k^{\frac{\theta}{2}}] < \infty \;, \]
    %where $\theta$ is a positive constant such that $\theta/2 < \beta - 1/2$.
%\end{itemize}
\end{conditionbis}
If a process $X$ (defined as in~\eqref{xn-incremnto1}) has associated sequences $\{\gamma_i\}_{ i\geq 1}$ and $\{\xi_i\}_{i \geq 1}$ which  satisfy  Condition~\ref{condiçaoI*}, and~\ref{condition_B},  we call $X$ a $p_n$-\Name*.

\subsection{Main results}\label{main_pn}

% Our first result is a SLLN for the range of  a $p_n$-\Nametwo{}.

For $X$ a $p_n$-\Nametwo{} we defined its range  as 
\[
\Rr^X_{n}:=\{X_0, X_1, \ldots, X_n\} \subset \mathbb{Z}^d\,,
\]
 i.e.,  the set of sites visited by  $X$ up to time $n$. 

We study the range of a $p_n$-\Nametwo{} when  $p_n= \mathcal{C}n^{-\beta} \wedge 1$ and $d\geq 2$.  
Observe that for $\beta >1$,  the $p_n$-\Nametwo{} will eventually behave as a random walk with i.i.d. increments $\{\xi_i\}_{i\geq 1}$. In this case, it is well-known that $\lim_{n \to \infty}|\Rr^X_{n}|/n =\pi_d$  almost surely, where  $\pi_d$ denotes the probability that the random walk with increments $\{\xi_i\}_{i\geq 1}$ never returns to the origin (see, e.g., \cite{spitzer2001principles}).  The interesting question is how the range of a $p_n$-\Nametwo{} behaves for $\beta \leq 1$, i.e., when the probability of the random walk getting excited on first visits is non-summable.  
% (in our proof we will need a large deviation principle proved in \cite{hamana2001large} that we state in Section \ref{sec:rangeERW} for the sake of completeness).}

Below we present two results that together provide a Strong Law of Large Numbers (SLLN) for $\Rr^X_{n}$ when the dimension is sufficiently large, depending on the value of $\beta \in (0,1]$. The first result is a tight upper bound that works in any dimension $d\geq2$. The second is a tight lower bound that only works when $d \ge h(\beta)$, where $h$ is a decreasing function of $\beta$, that is, as $\beta$ approaches zero, the required dimension increases. We believe this is a limitation of the proof technique employed rather than the actual behaviour of the process. 
% Considering the representation in \eqref{xn-incremnto1} for $X$ a $p_n$-\Nametwo{} (see, Section~\ref{sec:p_n-ERW}), 
% Let us denote by $\pi_d$ the probability of a random walk with i.i.d. increments (with zero mean and finite variance) given by the corresponding $\{\xi_i\}_{i\geq 0}$ never returning to the origin. 
%
\begin{proposition}\label{prop:RangeERW} 
Let  $X$ be a $p_n$-\Nametwo{} in direction $\ell$ on $\ZZ^d$ with $p_n= \mathcal{C} n^{-\beta} \wedge 1$. Then, for every $\beta >0$ and $d\geq 2$ it holds that
\begin{align*}
    \limsup_{n \to \infty} \frac{|\Rr_n^X|}{n} \le \pi_d \,, \quad  
    \text{ a.s.}\,.
\end{align*}
\end{proposition}

%
%Below, we propose a conjecture about the range of the $p_n$-\Nametwo{} on $\ZZ^d$, in direction $\ell \in \mathbb{S}^{d-1}$,  with $p_n= \mathcal{C}n^{-\beta} \wedge 1$, with $\beta\geq 1/2$ and  $d \ge 2$.
%

\begin{hypothesis}[Lower Bound for the range of the $p_n$-\Nametwo{}]\label{hyp} Let  $X$ be a $p_n$-\Nametwo{} in direction $\ell \in \mathbb{S}^{d-1}$ on $\ZZ^d$ with $d\geq2$ and $p_n= \mathcal{C}n^{-\beta} \wedge 1$ and $\beta >0$. We say that $X$ satisfies Hypothesis~\ref{hyp} if 
\begin{align*}
\liminf_{n \to \infty} \frac{|\Rr_n^X|}{n} \ge \pi_d \,, \quad  \text{ a.s.}\,.
\end{align*}
% \begin{equation*}
%     \frac{|\Rr_n^X|}{n} \xrightarrow[n \to \infty]{} \pi_d \ \text{ a.s.}\,.
% \end{equation*}    
% where $\pi_d$ is the probability that the random walk with increments $\{\xi_i\}_{i \ge 1}$ never returns to the origin.
\end{hypothesis}

Note that for $d=2$, we have  $\pi_d=0$, whereas for $d\geq 3 $, $\pi_d\in (0,1]$. Thus, Hypothesis~\ref{hyp} clearly holds for $d=2$. 

\begin{proposition}\label{prop:RangeERW_lower} 
% Let  $X$ be a $p_n$-\Nametwo{} in direction $\ell$ on $\ZZ^d$ with $d\geq 22$, $p_n= \mathcal{C} n^{-\beta} \wedge 1$ and $\beta \ge 1/2$. Then, it holds that
% \begin{align*}
% \liminf_{n \to \infty} \frac{|\Rr_n^X|}{n} \ge \pi_d \ \text{ a.s.}\,.
% \end{align*}
Let $X$ a $p_n$-\Nametwo{} in direction $\ell \in \mathbb{S}^{d-1}$ on $\ZZ^d$ with $p_n= \mathcal{C}n^{-\beta} \wedge 1$ and $d\geq 3$. Then, for every $\beta \in (0,1]$, $X$ satisfies Hypothesis~\ref{hyp} if $d \ge \max\{9, \lfloor 2 + 4/\beta \rfloor + 1\}$. 
\end{proposition}

Combining  Proposition~\ref{prop:RangeERW} and~\ref{prop:RangeERW_lower}, we   obtain the following result.
\begin{theorem}[SLLN for the range]\label{thm:SLLN}
Let $X$ be a $p_n$-\Nametwo{} in direction $\ell \in \mathbb{S}^{d-1}$ on $\ZZ^d$ with $p_n= \mathcal{C}n^{-\beta} \wedge 1$. Then, for every $\beta \in (0,1]$ and $d=2$ or $d\ge h(\beta):=\max\{9, \lfloor 2 + 4/\beta \rfloor + 1\}$, it holds that 
\begin{align*}
\lim_{n \to \infty} \frac{|\Rr_n^X|}{n} = \pi_d \,, \quad  \text{ a.s.}\,.
\end{align*}
\end{theorem}

For example, applying the above theorem, we have that for $\beta=1$ the SLLN for the range holds for  $d \ge 9$, whereas for $\beta=1/2$ it holds for $d \ge 11$. 
We conjecture that the SLLN for the range when $\beta\in (0,1]$ holds for any $d\geq 3$.  Despite the dimensional gap, we believe that Theorem~\ref{thm:SLLN} is a strong  result since the absence of the Markov property in our model makes the control of the range a hard problem to tackle.

Our next result is a sub-ballistic SLLN for the $p_n$-\Nametwo{}.

\begin{theorem}\label{thm:conv-Xbeta} Let $X$ be a $p_n$-\Nametwo{} in direction $\ell \in \mathbb{S}^{d-1}$ on $\ZZ^d$, $d\geq 2$, $p_n= \mathcal{C}n^{-\beta} \wedge 1$ and $\beta \in (0,1/2)$. 
Assume that $X$ satisfies Hypothesis~\ref{hyp}, then 
$$
\lim_{n\to \infty} \frac{X_n}{n^{1 - \beta}} = \mathcal{C}\frac{\pi_d \EE[\gamma_1]}{1-\beta} \,, \quad \text{ a.s.}\,.
$$    
\end{theorem} 
% \com{Check whether and how the constant $\mathcal{C}$ appears in the limiting expression}
% \com{Do we have an educated  guess why Hypothesis~\ref{hyp} should hold for $\beta<1/2$! I guess it would be nice to comment upon this here...}\comu{Me parece que teremos a validade da hipótese com a dimensão podendo depender de $\beta$}

We now present two  results related to the behavior of $p_n$-\Name{} under diffusive rescaling.
Before we state them let us introduce some notation. Define 
\begin{equation}\label{eq:B}
\Hat{B}_t^n := \frac{X_{\lfloor nt \rfloor}}{n^{1/2}} + (nt - \lfloor nt \rfloor)\frac{(X_{\lfloor nt \rfloor + 1} - X_{\lfloor nt \rfloor})}{n^{1/2}} \,, \ t\ge 0 \,,
\end{equation}
where the process $X$ will make reference to a specified process ($p_n$-\Name, $p_n$-\Name*{} or $p_n$-\Nametwo{}) in each result; we shall refer to  $\Hat{B}_t^n$ as the corresponding rescaled process.

Let $C_{\Rs^d}[0,T]$ be the space of continuous functions from $[0,T]$ to $\mathbb{R}^d$ for every $T > 0$. We consider $C_{\Rs^d}[0,T]$ endowed with the uniform topology. Denote by $C_{\Rs^d}[0, \infty)$  the space of continuous functions from $[0,\infty)$ to $\mathbb{R}^d$  endowed  with the metric
\begin{equation}\label{def:rho}
 \rho(f, g) := \sum_{k=1}^{\infty}\frac{1}{2^k} \sup_{0 \le t \le k}(||f(t) - g(t)|| \wedge 1) \,,  \ f, \, g \in  C_{\Rs^d}[0, \infty) \,.
\end{equation}
% It is well-known that if a sequence of random functions in $C_{\Rs^d}[0, \infty)$ converges in probability under the uniform metric in $C_{\Rs^d}[0, T]$ for all $T >0$, then it also converges in probability under the metric $\rho$ \com{this last paragraph  could be removed once we fix the remark 3.1}.

Our first result is a Central Limit Theorem for the $p_n$-\Name* when  $\beta>1/2$ and $d\geq 2$. Henceforth,  for a column vector $a \in \mathbb{R}^d$, we denote by $a^T$ its transpose.
%Here we will set $p_n = K'n^{-1/2}$ \com{$p_n = K'n^{-\beta}$?} where $\beta > 1/2$ and $K'$ is a positive constant such that $K' \in (0,1]$.

\begin{proposition}\label{pn-WGERW-Gauss} 
Let $X$ be a $p_n$-\Name* in direction $\ell\in \mathbb{S}^{d-1}$ on $\ZZ^d$, $d \ge 2$,  $p_n= \mathcal{C}n^{-\beta} \wedge 1$ and $\beta > 1/2$. 
%We define
%\begin{equation*}
    %B_t^{n} = \frac{X_{\lfloor nt \rfloor}}{n^{1/2}} \quad \text{for } t \in \mathbb{R}^{+}\;.
%\end{equation*}
Suppose that 
    \begin{equation*}
    \lim_{k \to \infty} k^{-1/2} \EE\Big[ \sup_{1 \leq i \leq k} \| \xi_i \| \Big]  = 0 \,, \quad %\text{ and}
    \end{equation*}
    and that there exists $C=((c_{i,j}))$ a continuous $d \times d$ matrix-valued function on $[0,\infty)$ satisfying $C(0) = 0$  and 
\begin{equation*}
\sum_{i,j = 1}^d (c_{i,j}(t) - c_{i,j}(s))\alpha_i \alpha_j \geq 0 \quad \text{for any } \alpha \in \mathbb{R}^d, \quad t > s \geq 0\,,
\end{equation*}
such that
\begin{equation}\label{condGaussiano}
    \frac{1}{n}\sum_{i=1}^{\lfloor nt \rfloor} \xi_i \xi_i^T \xrightarrow[n \to \infty]{} C(t)  \ \textrm{ in probability}\,.
    \end{equation}
% where $C=((c_{i,j}))$ is a continuous $d \times d$ matrix-valued function on $[0,\infty)$ satisfying $C(0) = 0$  and 
% \begin{equation*}
% \sum_{i,j = 1}^d (c_{i,j}(t) - c_{i,j}(s))\alpha_i \alpha_j \geq 0 \quad \text{for any } \alpha \in \mathbb{R}^d, \quad t > s \geq 0\,.
% \end{equation*}
Then $\{\Hat{B}_{\cdot}^n\}_{n\geq 1}$ converges in distribution to a process with independent Gaussian increments with sample paths in $C_{\mathbb{R}^d}[0, \infty)$.
\end{proposition}

From Proposition~\ref{pn-WGERW-Gauss} we obtain the following corollary.

\begin{corollary}\label{pnESRW->BM}
Let $X$ be a $p_n$-\Nametwo{} in direction $\ell\in \mathbb{S}^{d-1}$ on $\ZZ^d$,  $d \ge 2$, $p_n= Cn^{-\beta} \wedge 1$ and  $\beta > 1/2$. %We define
%\begin{equation*}
 %   B_t^{n} = \frac{X_{\lfloor nt \rfloor}}{n^{1/2}} \quad \text{for } t \in \mathbb{R}^{+}\;.
%\end{equation*}
Then $\{\Hat{B}_{\cdot}^n\}_{n\geq 1}$ converges in distribution to a $d$-dimensional Brownian Motion in $C_{\Rs^d}[0,\infty)$ with covariance matrix $\Sigma=\mathbb{E}[\xi_1\xi_1^T]$. 
% \textcolor{brown}{Por qu\'e no caso i.i.d. se teria a condi\c{c}\~ao acima de \eqref{condGaussiano}?Neste caso $c_{i,i}(t) = t\mathbb{E}[\xi_1(i)\xi_1(j)]$. Aqu\'i \'e precisso que as componentes sejam n\~ao correlacionadas para obter a condi\c{c}\~ao acima de \eqref{condGaussiano}.}
\end{corollary}

\begin{remark}
Note that  the $p_n$-\Name* when $\beta>1$ eventually behaves as a $d$-dimensional martingale and, in this case, the above proposition (essentially) reduces to  \cite[Theorem 7.1.4]{ethier2009markov}. In light of that, Proposition~\ref{pn-WGERW-Gauss} is a simple and not too surprising result which we state and prove for completeness. 
The main and more technical result in the context of diffusive rescaling is for the case $\beta=1/2$.
\end{remark}

% \begin{remark}\label{rem_cond_antes}
% Let us point out that Theorem~\ref{pn-WGERW-Gauss} holds true under slightly weaker conditions. Specifically the Condition~\ref{condiçao I*} and the sequence $\{p_n\}_{n \ge 1}$ can be more general. As it emerges from the proof,  for the statement of Theorem~\ref{pn-WGERW-Gauss} be true we only need that $\sum_{i=1}^{\lfloor nt \rfloor} p_i \EE[||\gamma_i||] = o(\sqrt{n})$ (for more details, see Remark~\ref{rem_cond_frac} ). 
% %
% % \cm{then it will be possible to prove that the second sum portion in~\eqref{p_n-WGERW_incrementos} goes to zero in probability. Hence we finish the poof with Slutsky's Theorem (Theorem 11.4 from~\cite{gut2005probability})}.
% \end{remark}

%\comu{Podemos colocar um remark. Bastaria ter $p_n \sim o(\sqrt{n})$ e $\sup \sqrt{n}p_n E[\|\gamma_n\|] < \infty$} \cm{Coloquei um remark após a prova, naõ exatamente isso.}

%\com{Eu sugiro colocar o remark aqui!}

We now state our main result for $\beta=1/2$.  
\begin{theorem}\label{thm:main-conv}
Let  $X$ be a $p_n$-\Nametwo{} in direction $\ell \in \mathbb{S}^{d-1}$ on $\ZZ^d$,  $d \ge 2$ and   $p_n= \mathcal{C} n^{-1/2} \wedge 1$. %We define,
%\begin{equation*}
 %   B_t^{n} = \frac{X_{\lfloor nt \rfloor}}{n^{1/2}} \quad \text{for } t \in \mathbb{R}^{+}\;.
%\end{equation*}
Assume that  $X$ satisfies Hypothesis~\ref{hyp}, then  $\{\Hat{B}_{\cdot}^n\}_{n\geq 1}$ converges in distribution to 
$\big\{W_{t} + 2 \mathcal{C} \pi_d \EE[\gamma_1]\sqrt{t}\big\}_{t\geq 0}$,
where $W_{\cdot}$ is a Brownian Motion.
\end{theorem}

% \com{Check whether and how the constant $\mathcal{C}$ appears in the limiting expression}

 Since $\pi_d=0$ and Hypothesis~\ref{hyp} holds for $d=2$ and $\beta=1/2$, we obtain the following corollary. 
 \begin{corollary}\label{pn-ERW-d=2} 
Let $X$ be a two dimensional $p_n$-\Nametwo{} in direction $\ell$ on $\ZZ^2$ with $p_n= \mathcal{C} n^{-1/2} \wedge 1$. %We define,
%\begin{equation*}
 %   B_t^{n} = \frac{X_{\lfloor nt \rfloor}}{n^{1/2}} \quad \text{for } t \in \mathbb{R}^{+}\;.
%\end{equation*}
Then $\{\Hat{B}_{\cdot}^n\}_{n\geq 1}$ converges in distribution to a $2$-dimensional Brownian Motion in $C_{\Rs^2}[0, \infty)$.
\end{corollary}

\begin{remark}
Note that Corollary~\ref{pnESRW->BM} and~\ref{pn-ERW-d=2} tell us that for $d= 2$ the $p_n$-\Nametwo{}, with $p_n= \mathcal{C} n^{-\beta} \wedge 1$ and $\beta\geq 1/2$  under  diffusive scaling 
% \sout{a suitable rescaling} 
converges in distribution to a Brownian Motion and has no ballisticity. Together with 
% Proposition \ref{pn-WGERW-Gauss} and 
Theorem \ref{thm:conv-Xbeta}, absence of ballisticity holds in dimension $d=2$ for any $\beta > 0$. This is a distinct behavior from the usual \Nametwo{} which is ballistic (see~\cite{benjamini2003excited}, \cite{kozma2003excited} and \cite{kozma2005excited}).
Note also that the behavior of the $p_n$-\Nametwo{}, with $p_n= \mathcal{C} n^{-1/2} \wedge 1$ and $d>2$ is different from the case of a random walk with i.i.d. increments, as indicated by the presence of the additive term  $2 \pi_d \mathcal{C}\EE[\gamma_1]\sqrt{t}$ in Theorem~\ref{thm:main-conv}.  
% \cm{For simplicity, when we refer to the \textit{rescaled $p_n$-\Nametwo{}}, we mean the $p_n$-\Nametwo{} under a suitable rescaling.}
% \com{As far as I understand, the reviewer is suggesting to replace (everywhere in the paper),  "$p_n$-\Nametwo{} under a suitable rescaling" with  "rescaled $p_n$-\Nametwo{}" after we state once that we only work under diffusive scaling; I do agree!}
\end{remark}

Observe that Theorem~\ref{thm:main-conv} and  Proposition~\ref{prop:RangeERW_lower} imply that  for $d=2$ or $d \ge \max\{9, \lfloor 2 + 4/\beta \rfloor + 1\}$ the 
$p_n$-\Nametwo{} under diffusive scaling converges in distribution to a Brownian Motion plus a correction function given by $2 \mathcal{C} \pi_d \EE[\gamma_1] \sqrt{t}$.
This dimensional gap is due to the fact that Hypothesis~\ref{hyp} was proved only for $d=2$ or $d \ge \max\{9, \lfloor 2 + 4/\beta \rfloor + 1\}$(Proposition~\ref{prop:RangeERW_lower}).
If we drop Hypothesis~\ref{hyp} we obtain the following weaker result but with no dimensional gap: Let $X$ be a $p_n$-\Nametwo{} in direction $\ell \in \mathbb{S}^{d-1}$ on $\ZZ^d$, $d \ge 2$ and $p_n= \mathcal{C} n^{-1/2} \wedge 1$, then $\{\Hat{B}_{\cdot}^n\}_{n\ge 1}$ is tight in $C_{\Rs^d}[0, \infty)$ and for any of its limit points $\mathcal{Y}_{\cdot}$ it holds that
\begin{align}\label{prop:d2-23}
\{\mathcal{Y}_t \cdot \ell\}_{t\ge 0} \preceq \left\{W_t \cdot \ell + 2 \mathcal{C}\pi_d \EE[\gamma_1 \cdot \ell] \sqrt{t}\right\}_{t\ge 0} \, ,
\end{align}
where $W_{\cdot}$ is a Brownian Motion and   $\preceq$ denotes stochastic domination.
The proof follows from Theorem \ref{thm:main-conv} by a straightforward argument, using only the upper bound in Proposition \ref{prop:RangeERW} (this is rather easy to verify; see also Remarks \ref{main-ub1} and \ref{main-ub2}). 

\medskip

Table~\ref{table:1} provides a summary of the main results concerning the $p_n$-\Name{}, with $p_n=\mathcal{C}n^{-\beta} \wedge 1$ for different values of $\beta$ and dimension $d$. 

\renewcommand\thetable{\thesection.\arabic{table}}

\begin{table}[!h]    
\caption{Overview of the main results. 
% The function $h$ appearing in the second line is defined as  $h(\beta):=\max\{ 2 \frac{\beta^2 + 3 \beta +1}{\beta^2}, 4 \frac{2\beta +1}{\beta}\}$ 
}
\label{table:1}
\begin{tabular}{
|p{0.155\textwidth}|p{0.26\textwidth}
|p{0.48\textwidth}|}
% \hline 
%  $\displaystyle p$-GERW \hfill ($\displaystyle d  \ge 2$, $p\in(0,1]$) &  \ ballisticity in the drift direction for every $\displaystyle p\  >\ 0$. \\
% \hline 
%  $\displaystyle p$-SGERW \hfill ($\displaystyle d  \ge 2$, $p\in(0,1]$) &  \ besides ballisticity, LLN and CLT. \\
% \hline 
%  $\displaystyle p_{n}$-GERW \ \hfill ($\displaystyle \beta <  1/6$, $\displaystyle d \geq 2$) &  \ positive probability of never returning to the origin (in the direction $\ell$) \\
%\hline 
 %$\displaystyle p_{n}$-\Nameone{} \ \ \hfill ($\displaystyle \beta  >1/2$, $\displaystyle d\geq 2$) &  \ convergence in distribution to standard Brownian Motion. \\
\hline 
  % $\displaystyle p_{n}$-\Nametwo{} \hfill ($\displaystyle \beta \geq 1/2$, $\displaystyle d=2$ or \phantom{bla bla bla bla bla bl} $\displaystyle d\geq 22$) &  {\it \small Strong Law of Large Numbers for the range.} 
    $\displaystyle p_{n}$-\Nametwo{} & $\displaystyle \beta >0$, $\displaystyle d\geq 2$  &  {\it \small $    \limsup_{n \to \infty} |\Rr_n^X|/n \le \pi_d \ \text{ a.s.}$.} 
 \\
 \hline 
    $\displaystyle p_{n}$-\Nametwo{} & 
$\displaystyle \beta \in (0,1]$, $\displaystyle d \ge h(\beta) $  
      &  {\it \small $    \liminf_{n \to \infty} |\Rr_n^X|/n \ge \pi_d \ \text{ a.s.}$ (Hyp~\ref{hyp}).} 
 \\
 \hline 
  $\displaystyle p_{n}$-\Nametwo{} & $\displaystyle \beta <1/2$, Hyp.~\ref{hyp}&  {\it \small Sub-ballistic Strong Law of Large Numbers.} 
 \\
 \hline 
 $\displaystyle p_{n}$-GERW* & $\displaystyle \beta   > 1/2$, $\displaystyle d\geq 2$ &   {\it \small Convergence in distribution to a Gaussian Process (under diffusive scaling).} \\
\hline 
 $\displaystyle p_{n}$-\Nametwo{} & $\displaystyle \beta \geq 1/2$, $\displaystyle d=2$ &  {\it \small Convergence in distribution to a Brownian Motion (diffusive scaling).} 
 \\
 \hline 
$\displaystyle p_{n}$-ERW & $\displaystyle \beta =1/2$, Hyp.~\ref{hyp} &  {\it \small  Convergence in distribution to a Brownian Motion plus a multiple of square root of time (diffusive scaling).}
 \\
 \hline 
$\displaystyle p_{n}$-ERW & $\displaystyle \beta <1/2$, Hyp.~\ref{hyp} & {\it \small  Convergence in distribution to a deterministic function multiple of square root of time ($n^{1-\beta}$ scaling).
}
%  \\
%  \hline 
% {\color{red} $\displaystyle p_{n}$-ERW \hfill ($d=2$ or $d\ge 22$) }& {\color{red} {\it \small Hypothesis 1 holds. \texttt{\sc Already in the first line... }}
% }
%  \\
% \hline 
%  \sout{$\displaystyle p_{n}$-ERW \hfill ($\displaystyle \beta =1/2$, $\displaystyle  d\geq 3$)} &   {\it \small  \sout{Tightness  and any limit point is} stochastically dominated in the drift direction from above by a Brownian Motion plus a multiple of square root of time.}
%  \\
% \hline 
% \sout{$\displaystyle p_{n}$-ERW \hfill ($\displaystyle \beta =1/2$, $\displaystyle d\geq 22$)} &   {\it \small  \sout{Any limit point is also stochastically} dominated in the drift direction from below by a Brownian Motion plus a multiple of square root of time.}
 \\
 \hline
  $\displaystyle p_{n}$-GERW & $\displaystyle \beta$ small, $\displaystyle d\geq 2$ &   {\it \small Directional transience (see  \cite{alves2022note}).}\\
 \hline
\end{tabular}
\end{table}

\section{On the range of $p_n$-\Nametwo{}}\label{sec:rangeERW}

We begin by introducing some notation and recalling a well-known large deviation principle for the range of a random walk with i.i.d. increments.

Here $\{\xi_i\}_{i \ge 1}$ is a sequence of i.i.d. $\ZZ^d$-valued random variables with zero-mean vector and finite covariance matrix. Let $\{Y_n\}_{n \ge 0}$ be the random walk on $\ZZ^d$ with increments $\{\xi_i\}_{i \ge 1}$ starting at $Y_0 = 0$, thus $Y_n = \sum_{i=1}^{n} \xi_i$, $n\ge 1$.  For $m \leq n$ define 
\[ \Rr_{[m,n]} ^Y := \{Y_m, Y_{m+1}, \ldots, Y_n\}\;,\]
and denote by $\Rr^Y_{n} = \Rr^Y_{[0,n]}$,  the range of the random walk $\{Y_n\}_{n \geq 0}$. Recall 
% from the statement of Proposition~\ref{prop:RangeERW} 
that $\pi_d$ denotes the probability that $\{Y_n\}_{n \geq 0}$ never returns to the origin. 

\begin{theorem}[{\cite[Theorem 1]{hamana2001large}}]\label{teo: RnZ>}
Let $\{Y_n\}_{n \geq 0}$ be a genuinely  $d$-dimensional random walk on $\ZZ^d$, with $d\geq 2$. It holds that 
\begin{align*}
\tag{L}&\lim_{n \to \infty} \PP[|\Rr_{n}^Y| \geq \theta n]= 1\,, &&  \text{for every $\theta < \pi_d$}\,, 
\\
\tag{U}&\PP[|\Rr_{n}^Y| \geq \theta' n] \leq e^{-c_{\theta'}n}\,,    && \text{for every  $\theta' > \pi_d$ and $n$ sufficiently large}\,, 
\end{align*}
where $c_{\theta'}$ is a positive constant that depends of $\theta'$.
% (note that for $d=2$, we have that $\pi_d=0$, whereas for $d\geq 3 $, $\pi_d\in (0,1]$. 
\end{theorem}

\begin{remark}
 Note  that~\cite[Theorem 1]{hamana2001large} is stated for ``aperiodic''  random walks (see,  (\cite[Condition (1.9)]{hamana2001large}). % with a definition of aperiodicity which is distinct than the usual one related to the periodicity of markov chains (see \cite{yuval}). Following \cite{spitzer2001principles}, a random walk is aperiodic if (\comu{completar}).
However,  as explained in~\cite[page 188-Section 2]{hamana2001large} the  result also holds for any genuinely $d$-dimensional random walk. 
% can be extended to periodic random walks and the aperiodicity does not entail a loss of generality. \comu{é isso mesmo?}
\end{remark}

\subsection{Upper bound for the range (proof of  Proposition~\ref{prop:RangeERW})}

\hfill \\

The heuristic of the proof of Proposition~~\ref{prop:RangeERW} is the following: 
suppose that the $p_n$-\Nametwo{} at time $j$ visits a new site and gets excited. Then, after $k$ further steps it visits another new site and gets excited again and  no excitation occurs between time $j+1$ and $j+k-1$. Thus,  we know that between time $j+2$ and $j+k$ the process evolves as a random walk with i.i.d. increments. Specifically, in each of these time windows between two consecutive excitations, we can use the range of the random walk with i.i.d. increments to  upper bound the range of the $p_n$-\Nametwo{}.

%The main idea of this proof is we know that the $p_n$-\Nametwo{} behaves like random walk biased in direction $\ell$ \comu{"drift random walk" não é um bom termo} when it eats a cookie. Then between two cookies we have that the process behaves like an i.i.d. random walk. Hence in this lengths, we think as independent i.i.d. random walks \comu{"i.i.d. random walk" não é um bom termo, talvez "random walk with i.i.d. increments. Recorrente.} and use the ranges of those process to upper bound the range of the $p_n$-\Nametwo{}. By Lemma~\ref{RnY_upperb} we can control the range of each independent random walk, thus we obtain the desired result.  \comu{Só neste parágrafos há vários problemas de inglês}

\begin{proof}[Proof of Proposition~\ref{prop:RangeERW}]
%\textcolor{red}{Let us denote $(N_i, i \geq 0)$  as the sequence of times that the process $X$ is allowed to eat a cookie, this happens if the position is being visited for the first time and the Bernoulli trial in the site goes in favor of the process with a drift

%\texttt{The above paragraph is confusing: $N_i$ only looks at the Bernoulli variables regardless if the sites visited at the corresponding time has already been visited before or not!}}. %, that is, it behaves like process with a drift to the right. 
%\cm{Let us denote $\{N_i\}_{i \geq 0}$ as the sequence of times that the Bernoulli trial goes in favor of the $p_n$-\Nametwo{} behaves like a random walk biased in direction $\ell$.}

For simplicity, we set $\mathcal{C} = 1$, and the general case follows analogously. The proof for $\beta>1$ is straightforward, since from a simple Borel-Cantelli argument it follows that the number of excitations is finite almost surely. So let us proceed assuming $\beta \in (0,1]$.

Denote by $\{N_i\}_{i \geq 0}$ the sequence of stopping times 
\[ N_0 \equiv 0 \ \ \mathrm{and} \ \ N_i := \inf\{ k > N_{i-1}: Z_k^{\beta} = 1 \} \, , \ i\ge 1 \,, \]
where $Z_k^{\beta} = \um_{\{U_k \le k^{{-\beta}} \}}$, $k \geq 1$, are independent random variables with Bernoulli distribution of parameter $k^{{-\beta}}$, respectively.
Set $\Delta N_i = N_i - N_{i-1} $ and define 
\[ M_n := \inf \Big\{ i \geq 1 : \sum_{j=1}^i \Delta N_j \geq n \Big\} \,. \]

Note that $|\Rr_n ^X| - 1  = \sum_{t=1}^n \um_{\{X_t \neq X_l, \forall l < t \}}$ is bounded from above by 
\begin{align*}
 &  \sum_{t=1}^{N_1} \um_{\{X_t \neq X_l, \forall l < t \}} + \sum_{t=N_1 +1}^{N_2} \um_{\{X_t \neq X_l, \forall l < t \}} + \dots + \sum_{t=N_{M_n -1} +1}^{N_{M_n}} \um_{\{X_t \neq X_l, \forall l < t \}}
 \\
 & \leq M_n + \sum_{j = 1}^{M_n} \sum_{t = N_{j-1} + 2}^{N_j} \um_{\{X_t \neq X_l, \forall l < t \}} \\
 &\leq M_n + \sum_{j = 1}^{M_n} \sum_{t = N_{j-1} + 2}^{N_j} \um_{\{X_t \neq X_l, \forall l \in [N_{j-1} + 1, t) \}} \,.
\end{align*}
In each time interval $[N_{j-1} + 2, N_j]$ the process $X$ behaves like a random walk with i.i.d. increments. Note that $N_{j-1}+ 1$ is not accounted in the interval, since $Z^\beta_{N_{j-1}}=1$, and thus if the process at that time is visiting a position for the first time, it gets excited.

To have some control on the length of these intervals, or equivalently on $\{\Delta N_j\}_{j \geq 1}$, we proceed as follows: Let $\varepsilon \in (0, 1)$ and note that 
\begin{equation}\label{eq: Rn<}
|\Rr_n ^X| \leq |\Rr_{[0,n^{\varepsilon}]} ^X| + |\Rr_{[n^{\varepsilon}, n]} ^X| \leq n^{\varepsilon} + |\Rr_{[n^{\varepsilon}, n]} ^X|\; , 
\end{equation}
this last step will be necessary to guarantee that after time $n^\varepsilon$ the time intervals $\Delta N_j$ are (in distribution) sufficiently large. Thus, we may just redefine $N_0 \equiv n^{\varepsilon}$ and apply the very same decomposition as before to obtain %\comu{Qual o efeito da escolha do $N_0$ do lado direito?} {\color{blue} (note que não podemos usar $\Rr^Y_{\Delta N_j}$ na soma abaixo, porque teríamos ranges sobre intervalos que não são disjuntos.)}
\begin{align} \label{pR_DNj}
|\Rr_{[n^{\varepsilon}, n]} ^X| -1 &\leq
 M_n + \sum_{j = 1}^{M_n} \sum_{t = N_{j-1} + 2}^{N_j} \um_{\{X_t \neq X_l, \forall l \in [N_{j-1}+1, t) \}} \nonumber
 \\ 
 & \leq  M_n + \sum_{j = 1}^{M_n} |\Rr_{[N_{j-1} + 2, N_j]} ^Y| \,. 
\end{align}
%where $\{Y_n\}_{n \geq 0}$ denotes a random walk whose i.i.d.~increments in $\ZZ^2$ are $\{\xi_i\}_{i \ge 1}$. 
%\cm{ Important to notice that $Y$ is independent of $N$, since the increments of $Y$ are determined by the sequence of $\{ \xi_i \}_{i \ge 1}$ which occur in the length $[N_{j-1}+ 2, N_j]$. ??} \comu{importante mencionar que Y é independente dos N's pois seus incrementos são deteminados pelos $\xi$'s.}
%
We pointed out that to deal properly with the rightmost sum in~\eqref{pR_DNj}, we need to keep in mind that $Y$ is independent of  $\{N_j\}_{j\ge 1}$.
Now, for any $k \in \{1, 2, \dots, n \}$ fixed, we define  the  random set 
\[
A_{n,k} \coloneqq \{ j \in \{1, 2, \dots, M_n \} : \Delta N_j \leq k \}\,,
\] 
and we write 
\begin{align*}
 \sum_{j = 1}^{M_n} |\Rr_{[N_{j-1} + 2, N_j]}^Y| & = \sum_{j \in A_{n,k}} |\Rr_{[N_{j-1} + 2, N_j]}^Y| + \sum_{j \in A_{n,k}^{c}} |\Rr_{[N_{j-1} + 2, N_j]}^Y| 
\\
& \leq (k-1)|A_{n,k}| + \sum_{j \in A_{n,k}^c} |\Rr_{[N_{j-1} + 2, N_j]}^Y| \;. 
\end{align*}
Using the simple inequality
$$
M_n = |A_{n,k}| + |A_{n,k}^c| \le |A_{n,k}| + \frac{M_n}{k} \le |A_{n,k}| + \frac{n}{k}\,,
$$
and \eqref{pR_DNj}, we obtain that 
\begin{equation}\label{R_DNj}
|\Rr_{[n^{\varepsilon}, n]} ^X| \le \Big(1 +  \frac{n}{k}\Big) + k |A_{n,k}| + \sum_{j \in A_{n,k}^c} |\Rr_{[N_{j-1} + 2, N_j]}^Y|\,.
\end{equation}
To control the right-hand side of \eqref{R_DNj}, we begin with estimates on $|A_{n,k}|$. Let $\{G_j\}_{j =1}^{n}$ be a sequence of i.i.d. geometric random variables with parameter $n^{-\varepsilon{\beta}}$. By a coupling argument, redefining $N_0 = n^{\epsilon}$, we obtain  $\sum_{j=1}^nG_j \preceq \sum_{j=1}^n\Delta N_j$, 
%$j \in \{1, 2, \dots, M_n \}$, 
where $\preceq$ denotes stochastic dominance.  
% Thus, it holds that
% \begin{equation}\label{eq: DNj<k}
% \PP[\Delta N_j \leq k] \leq \mathbb{P}[G_j \le k]= 1 - \left( 1 - \frac{1}{n^{\varepsilon{\beta}}} \right)^k  \,.    
% \end{equation}

Since $|A_{n,k}| = \sum_{j=1}^{M_n} \um_{\{\Delta N_j \leq k \}}$, for every $a>0$ and $n$ sufficiently large, it holds that
\begin{align*}
\PP\Big[  |A_{n,k}|   > a \Big]  &\leq  \PP\Big[  \sum_{j=1}^{n} \um_{\{\Delta N_j \leq k \}}  > a \Big]\\
%& \leq \PP\left[  \sum_{j=1}^{n} 1_{\{\Delta N_j \leq k \}}  > a \right]
%\\
 &\leq \binom{n}{\lceil a \rceil} \Big( 1 - \Big( 1 - \frac{1}{n^{\varepsilon{\beta}}} \Big)^k \Big)^{\lceil a \rceil}\,  
\\
& \leq \Big( \frac{ne}{\lceil a \rceil} \Big)^a \Big( 1 - \Big( 1 - \frac{1}{n^{\varepsilon{ \beta}}} \Big)^k \Big)^{\lceil a \rceil}\\
&
 \leq \Big( \frac{ne}{a} \Big)^{\lceil a \rceil} \Big( 1 - \exp{-\frac{3}{2}\frac{k} {n^{\varepsilon{\beta}}}} \Big)^{\lceil a \rceil} 
 \\&\leq \Big( \frac{ne}{a} \times \frac{3k}{2n^{\varepsilon{\beta}}} \Big)^{\lceil a \rceil} \,,
\end{align*}
where in the last inequalities we used that 
$\left(1-\frac{1}{x}\right)^x\geq e^{-3/2}$, $\forall x\geq 2$ (with $n$ sufficiently large) and that $1-e^{-x}\leq x$. 
Setting   $a= n^{1-\varepsilon \beta/2}$ we obtain that 
\begin{equation}\label{eq:1_DNj<k}
\PP\Big[  |A_{n,k}|  > n^{1-\varepsilon {\beta/2}} \Big] 
% \leq \Big( \frac{3ek}{2n^{\frac{\varepsilon}{4}}} \Big)^{n^{1-\varepsilon}} 
{ \leq \Big( \frac{3ek}{2n^{\frac{\varepsilon \beta}{2}}} \Big)^{n^{1-\varepsilon \beta/2}}} \,.  
\end{equation}

We now set $k = \lceil \log^2 (n) \rceil$. With this choice, the deterministic first term in \eqref{R_DNj} divided by $n$ converges to zero.  Moreover, the sum in $n$ of the probabilities of the events $\{ |A_{n, \lceil \log^2 (n) \rceil}|  > n^{1-\varepsilon {\beta/2}} \}$ is finite by \eqref{eq:1_DNj<k}, { for every $\beta \in (0,1]$}. Thus, by Borel-Cantelli,  the second term in \eqref{R_DNj} divided by $n$ converges to zero almost surely.

Now we are left with the  analysis of the third term in~\eqref{R_DNj}. 
First,  we observe that for all $n\geq 2 $
\begin{equation}\label{eq: R_DNj>y3}
|A_{n, \lceil \log^2 (n) \rceil}^c| \leq  \frac{n}{\lceil \log^2 (n) \rceil} \,.    
\end{equation}

By Theorem~\ref{teo: RnZ>},  for all $i> \lceil \log^2(n) \rceil$ (with   $n$ sufficiently large) and for all $\gamma \in (\pi_d,1]$, it holds that 
\begin{align}\label{eq: RYi}
\begin{split}
\PP[|\Rr_{i-2}^Y| > \gamma i] &\leq  \PP[|\Rr_{i-2}^Y| > \gamma (i-2)]
 \leq \exp(-c_{\gamma}(i-2)) 
\\
& \leq \exp(-c_{\gamma}(\lceil \log^2(n) \rceil-2)) \,.
\end{split}
\end{align}
Recalling that for all $j \in A_{n, \lceil \log^2 (n) \rceil}^c$ it holds that $\Delta N_j > \lceil \log^2 (n) \rceil$,   by~\eqref{eq: R_DNj>y3} and~\eqref{eq: RYi} we obtain that
\begin{align}\label{eq: gDNj}
\begin{split}
\PP & \left[\exists j \in A_{n, \lceil \log^2 (n) \rceil}^c : |\Rr_{[N_{j-1} + 2, N_j]}^Y| > \gamma \Delta N_j\right] 
 \leq \frac{n \exp\left(-c_{\gamma} \lceil \log^2 (n) -2\rceil \right)}{\lceil \log^2 (n) \rceil}  
\\
& \leq \exp\left( \log\left( \frac{n}{\lceil \log^2 (n) \rceil} \right) - c_{\gamma} \lceil \log^2 (n)-2 \rceil \right) \,.
\end{split}
\end{align}
%\textcolor{purple}{Remember that, for all $j \in A_{n, \lceil \log^2 (n) \rceil}^c$, then we have $\Delta N_j \in ( \lceil \log^2 (n) \rceil$, n]. Hence we obtain
%\begin{align*}
%\PP[\Rr_{[N_{j-1} + 2, N_j]}^Y| > \gamma \Delta N_j] & \leq \sum_{i=\lceil \log^2 (n) \rceil +1}^{n} \PP[\Rr_{i-2}^Y| > \gamma i| \Delta N_j = i]
%\\
%& \leq \sum_{i=\lceil \log^2 (n) \rceil +1}^{n} \PP[\Rr_{i-2}^Y| > \gamma (i-2)| \Delta N_j = i]
%\\
%& \leq \sum_{i=\lceil \log^2 (n) \rceil +1}^{n} \exp(-c_{\gamma}(i-2))
%\\
%& \leq (n-\lceil \log^2 (n) \rceil)\exp\left(-c_{\gamma} \lceil \log^2 (n) -2\rceil \right)\;. 
%end{align*}}
Since it holds that
% $\big\{ \exists j \in A_{n, \lceil \log^2 (n) \rceil}^c : |\Rr_{[N_{j-1} + 2, N_j]}^Y| > \gamma \Delta N_j\big\}$ contains 
\begin{align*}
&\left\{ \exists j \in A_{n, \lceil \log^2 (n) \rceil}^c : |\Rr_{[N_{j-1} + 2, N_j]}^Y| > \gamma \Delta N_j\right\} \\
&\supseteq
\Big\{\sum_{j \in A_{n,\lceil \log^2 (n) \rceil}^c} |\Rr_{[N_{j-1} + 2, N_j]}^Y| > \gamma n\Big\}\,,
\end{align*}
then this last event has probability bounded above by the rightmost term in \eqref{eq: gDNj}
%\begin{equation}\label{eq: sumRy}
%\begin{split}    
%    \PP\Big[ \sum_{j \in A_{n,k}^c} |\Rr_{[N_{j-1} + 2, N_j]}^Y| > \gamma n \Big] & \leq \exp\Big( \log\Big( \frac{n}{\lceil \log^2 (n) \rceil} \Big) - c_{\gamma} \lceil \log^2 (n)-2 \rceil \Big) \,.
    %\\
    %& \to 0 \quad \text{as } n \to \infty \,.
%\end{split}    
%\end{equation}
which is summable.
%$$\Big\{ \sum_{j \in A_{n,k}^c} |\Rr_{[N_{j-1} + 2, N_j]}^Y| > \gamma n \Big\}$$ 
Thus, by Borell-Cantelli Lemma, the third term in~\eqref{R_DNj} divided by $n$ is bigger than $\gamma$ only finitely many times almost surely.
Hence,  letting $\gamma\downarrow \pi_d$ completes the proof. 
\end{proof}

\subsection{Lower bound for the range (proof of Proposition~\ref{prop:RangeERW_lower})}
\hfill \\

Throughout this section, we set $\mathcal{C} = 1$ for simplicity, as the general case follows analogously. As for Proposition~\ref{prop:RangeERW}, the proof for $\beta>1$ is straightforward. Let $\beta \in (0,1]$ and for $\ell \geq 1$ set $Z_\ell^\beta=\um_{\{U_\ell \leq  \ell^{-\beta}\}}$. 
Given $k\geq 1$ and $\delta \in (0,1)$, let us define the following set:
\begin{align}\label{eq:set_A}
A_{k,\delta, \beta} := \left\{\sum_{\ell = k}^{k + \lfloor k^\delta \rfloor}Z_\ell^\beta=0\right\}\,,
\end{align}
i.e.,  the event that in the first $\lfloor k^\delta \rfloor$ steps after time $k$ none of the corresponding Bernoulli are successful. In particular, the occurrence of the  event $A_{k,\delta, \beta}$ implies that the random walk does not get excited in the time window from $k$ to $k + \lfloor k^\delta \rfloor$. 

Before proving Proposition~\ref{prop:RangeERW_lower} we state a couple of auxiliary results. 
\begin{lemma}\label{lemma_1}  
Given { $\beta \in (0,1]$}  consider the sequence of events $\{A_{k, \delta,\beta}\}_{k\geq 1}$ defined in \eqref{eq:set_A}. Then, { for every $\delta \in (0, \beta)$}  it holds that 
\begin{equation*}			
\lim_{n \to \infty} \frac{1}{n} \sum_{k = 1}^n \um_{A_{k,\delta, \beta}^c} = 0 \,, \quad  \text{ a.s.}\,.
\end{equation*}

\end{lemma}

%%%%%% OLD LEMA %%%%%%%%%
% \begin{lemma}\label{lemma_1}  \com{maybe we should state this lemma for $\beta \geq 1/2$ to be more coherent with the rest....} 
% Let $\delta$ be a positive real number such that $\delta \in (0, 1/2)$ and we consider the event $A_{k, \delta}$ for $k \ge 1$. Then we have
% \begin{equation*}			
% \lim_{n \to \infty} \frac{1}{n} \sum_{k = 1}^n \um_{A_{k,\delta}^c} = 0\,, \, \text{ a.s.}.
% \end{equation*}
% \end{lemma}

The proof of Lemma~\ref{lemma_1} is given in  Appendix~\ref{sec:appendixC}.

\medskip

Given a $\mathbb{Z}^d$-valued process $\{S_n\}_{n\ge 1}$, we set
$$
e_k^S := \um_{\{ S_m \neq S_k \text{ for all } m > k \}} \,.
$$
Let $\mathcal{D}_n^S$ denote  the set of sites visited by the process $S$ up to time $n$  which are never revisited later. Then,  
\begin{equation*}
  |\mathcal{D}_n^S|= \sum_{k = 0}^{n} e_k^S \,.
\end{equation*}

\begin{lemma}\label{compXY}
Let  $X$ be a $p_n$-\Nametwo{} in direction $\ell$, on $\ZZ^d$ with  $p_n= \mathcal{C} n^{-\beta} \wedge 1$ and $\beta \in (0,1]$. For $k\ge 1$, let $Y^k =\{Y^k_i\}_{i\ge 0}$ denote a random walk on $\ZZ^d$ defined by $Y_0^k = X_k$ and for $n\geq 1$
\begin{equation*}
     Y_n^k = X_k + \sum_{i = 1}^{n}\xi_{k + i}\,.
\end{equation*}
{ If $d \ge \max\{9, \lfloor 2 + 4/\beta \rfloor + 1\}$,  there exists $\delta \in (0,\beta)$ such that }
\begin{equation*}
\sum_{k = 1}^{\infty} \frac{ \mathbb{P} \left( e_k^{X} \um_{A_{k,\delta,\beta}} < e_0^{Y^k} \um_{A_{k, \delta,\beta}} \right)}{k} < \infty \,.   
\end{equation*}
\end{lemma}

Before proving Lemma~\ref{compXY}, we show how the proof of Proposition~\ref{prop:RangeERW_lower} follows from it.
\medskip

\begin{proof}[Proof of Proposition~\ref{prop:RangeERW_lower}]
The following deterministic inequality holds $e_k^X \ge e_0^{Y^k} - \um_{\{e_k^X = 0, e_0^{Y^k} = 1\}}$ then we have
\begin{equation}\label{eq1}
\frac{1}{n} \sum_{k=1}^n e_k^X \um_{A_{k, \delta, \beta}} \ge \frac{1}{n} \sum_{k=1}^n e_0^{Y^k} \um_{A_{k, \delta, \beta}} - \frac{1}{n} \sum_{k=1}^n \um_{\{e_k^X = 0, e_0^{Y^k} = 1\}} \um_{A_{k, \delta, \beta}}\,.    
\end{equation}

Note that $|\Rr_n^X| \ge |\mathcal{D}_n^X|$, since if $x \in \mathcal{D}_n^X$ then $x \in \Rr_n^X$. Therefore, it holds that 
\begin{align}\label{eq_4.10}
\liminf_{n \to \infty} \frac{|\Rr_n^X|}{n} &\ge \liminf_{n \to \infty} \frac{|\mathcal{D}_n^X|}{n}\nonumber \\
&\geq \liminf_{n \to \infty} \frac{1}{n}\sum_{k=1}^n e_k^X \um_{A_{k,\delta,\beta}} + \liminf_{n \to \infty} \frac{1}{n}\sum_{k=1}^n e_k^X \um_{A_{k,\delta,\beta}^c}\,. 
\end{align}
For {$\delta \in (0, \beta)$}, Lemma~\ref{lemma_1} implies that 
\begin{equation}\label{eq_4.11}
    \liminf_{n \to \infty} \frac{1}{n}\sum_{k=1}^n e_k^X \um_{A_{k,\delta,\beta}^c}=0\,, \quad  \text{ a.s.}\,.
\end{equation}

% \rv{We denote $B_k = A_{k, \delta, \beta} \cap \{e_k^X = 0, e_0^Y = 1\}$.  By~\eqref{eq1} we have 
% \begin{equation*}
% \liminf_{n \to \infty} \frac{1}{n}\sum_{k=1}^n e_k^X \um_{A_{k,\delta,\beta}} \ge \liminf_{n \to \infty} \frac{1}{n}\sum_{k=1}^n e_0^Y \um_{A_{k,\delta,\beta}} - \liminf_{n \to \infty} \frac{1}{n}\sum_{k=1}^n \um_{B_k}\,.  
% \end{equation*}}

We observe that
\begin{align*}
    e_0^{Y^k} = \um_{\{\sum_{i = k+1}^m\xi_{i} \neq 0, \, \forall \, m\ge k + 1\}} = \um_{\{\sum_{i = 1}^m\xi_{i} + Y_0\neq \sum_{i = 1}^k\xi_{i} + Y_0, \, \forall \, m\ge k + 1\}} =e^Y_k\,,
\end{align*}
where $Y=Y^0$ on the RHS above denotes a random walk with i.i.d. increments $\{\xi_i\}_{i\geq 1}$. 
Hence
\begin{equation}\label{eq_4.13}
    \liminf_{n \to \infty} \frac{1}{n} \sum_{k = 1}^n e_0^{Y^k} \um_{A_{k, \delta,\beta}} = \liminf_{n \to \infty} \frac{1}{n} \sum_{k = 1}^n e_k^{Y} \um_{A_{k, \delta,\beta}}\,. 
\end{equation}
Using again Lemma~\ref{lemma_1} and the known fact that $\lim_{n \to \infty} \frac{|\mathcal{D}_n^Y|}{n} = \pi_d$ (see e.g.~\cite{spitzer2001principles} page 39) we obtain that 
\begin{equation}\label{eq_4.14}
\lim_{n \to \infty} \frac{1}{n} \sum_{k = 1}^n e_k^Y \um_{A_{k, \delta,\beta}} = \pi_d\,, \quad  \text{ a.s.}\,.
\end{equation}

We denote $B_k = A_{k, \delta, \beta} \cap \{e_k^X = 0, e_0^Y = 1\}$. By Lemma~\ref{compXY}, \(\sum_{k\ge 1}\mathbb P(B_k)/k<\infty\), which implies that \(\sum_{k\ge 1} \um_{B_k}/k<\infty\) almost surely, thus, by Kronecker's lemma
\begin{equation}\label{eq_4.12}
\lim_{n \to \infty} \frac{1}{n}\sum_{k=1}^n \um_{B_k} = 0\,. 
\end{equation}
% Note that
% \begin{align*}
%     e_0^{Y^k} = \um_{\{\sum_{i = k+1}^m\xi_{i} \neq 0, \, \forall \, m\ge k + 1\}} = \um_{\{\sum_{i = 1}^m\xi_{i} + Y_0\neq \sum_{i = 1}^k\xi_{i} + Y_0, \, \forall \, m\ge k + 1\}} =e^Y_k\,,
% \end{align*}
% where $Y=Y^0$ on the RHS above denotes a random walk with i.i.d. increments $\{\xi_i\}_{i\geq 1}$. 
% Hence
% \begin{equation}\label{eq_4.13}
%     \liminf_{n \to \infty} \frac{1}{n} \sum_{k = 1}^n e_0^{Y^k} \um_{A_{k, \delta,\beta}} = \liminf_{n \to \infty} \frac{1}{n} \sum_{k = 1}^n e_k^{Y} \um_{A_{k, \delta,\beta}}\,. 
% \end{equation}
% Using again Lemma~\ref{lemma_1} and the known fact that $\lim_{n \to \infty} \frac{|\mathcal{D}_n^Y|}{n} = \pi_d$ (see e.g.~\cite{spitzer2001principles} page 39) we obtain that 
% \begin{equation}\label{eq_4.14}
% \lim_{n \to \infty} \frac{1}{n} \sum_{k = 1}^n e_k^Y \um_{A_{k, \delta,\beta}} = \pi_d\,, \quad  \text{ a.s.}\,.
% \end{equation}
Then,  by equations \eqref{eq1},\eqref{eq_4.10},\eqref{eq_4.11},\eqref{eq_4.12},\eqref{eq_4.13} and \eqref{eq_4.14}, we obtain that
$$
    \liminf_{n \to \infty} \frac{|\Rr_n^X|}{n} \ge \pi_d \,, \quad  
    \text{ a.s.}\,.
$$
\end{proof}

\begin{proof}[Proof of Lemma~\ref{compXY}] 
To avoid clutter  we will denote $A_{k,\delta, \beta}$ just by $A_{k,\delta}$.
Let $\nu_k^X$ be the first time the process $X$ returns to the site it visited at time $k$.  We also introduce the sequence $\{ \tau_i^k \}_{i \ge 0}$ where $\tau_i^k$ for $i \ge 1$ represents the $i$-th time  the  random walk gets excited  after time $k$. Specifically, 
\begin{align*}
& \nu^{X}_k:= \inf \{n \geq 1: X_{k+n}= X_{k}\}\,,
\\
& \tau^k_0\equiv 0, \text{ and }
\\
& \tau_i^k:=\inf\{n> \tau^k_{i-1}: \um_{\{X_l \neq X_{k+n}, \forall l < k+n\}} \cdot \um_{\{Z_{k+n}^{\beta} = 1\}} = 1\}, \text{ for $i\geq 1$}\,.
\end{align*}

For $k \geq 1$, let us also define a sequence of independent random variables $\{H_j^k\}_{j \ge 1}$, such that each $H_j^k$ has a geometric distribution with parameter {$(k + j)^{-\beta}$ for all $j \ge 1$ and $\beta \in (0,1]$}. Additionally, we introduce the following definitions   
\begin{align*}
& \mathcal{G}_m^k = \sum_{j = 1}^m H_j^k \,, \, 
\\
& G^k_0\equiv 0 \text{ and }  G^k_i:= \inf\{\ell > G_{i-1}^k: Z_{\ell + k}^{\beta} =1\},  \text{ for $i\geq 1$}\,,
\end{align*}
where, $Z_\ell^{\beta}=\um_{\{U_\ell \leq  \ell^{-\beta}\}}$.
%\textcolor{red}{For each fixed $k \ge 1$ we define a coupling $\hat{\mathbb{P}}^k$ between the excited random walk $X$ and the random walk $Y$ with i.i.d. increments, defined such that $Y_0 = X_k$ and $Y_n := X_k + \sum_{i = k + 1}^{k + n} \xi_i$ for $n \ge 1$.} Suppose we have $\delta \in (0, 1/2)$. The first step in our proof is to establish the following
% \begin{equation*}
% \sum_{k = 1}^{\infty} \hat{\mathbb{P}} \left( e_k^{X} \um_{A_{k, \delta}} < e_k^{Y} \um_{A_{k, \delta}} \right) < \infty \,.   
% \end{equation*}
%
What we are after is to provide an upper bound for $\PP ( e_k^{X} \um_{A_{k, \delta}} < e_0^{Y^k} \um_{A_{k, \delta}} )$ for all $k \ge 1$. The process $X$ can be rewritten as (see, ~\eqref{xn-incremento2})
$$
X_0 = 0, \text{ and }\; X_n  = \sum_{i=1}^n \xi_i + \sum_{i=1}^n\um_{B_i}\big( \gamma_i - \xi_i\big) \, , \ n\ge 1\,,
$$
where {$B_i:= E_{i-1}^c \cap \{U_i \le i^{-\beta}\}$} for any $i \ge 1$ and we recall that $E_i$ corresponds to the event that the process $\{X_n\}_{n \geq 0}$ is, at time $i$,  in an already visited site. Then, we have that
\begin{equation*}
\begin{split}
 \mathbb{P}( e^{X}_k \um_{A_{k,\delta}} < e^{Y^k}_0 \um_{A_{k, \delta}}) &= \mathbb{P}(e^{X}_k=0, e^{Y^k}_0 =1, A_{k, \delta}) 
\\
& = \mathbb{P}(\nu^{X}_k < +\infty, e^{Y^k}_0 = 1, A_{k,\delta}) 
\\
& \le \sum_{m=1}^{\infty} \mathbb{P}\left( \tau_{m}^k < \nu^{X}_k \leq \tau_{m+1}^k, A_{k, \delta} \right)\,.
\end{split}
\end{equation*}
On the event $\{\tau_m^k < \nu_k^X \le \tau_{m+1}^k$\} we have that
$$
    0 = X_{k + \nu_k} - X_k = \sum_{j=k + 1}^{k + \nu_k^X} \xi_i + \sum_{j= k + 1}^{k + \nu_k^X}\um_{B_i}\big( \gamma_i - \xi_i\big) = \sum_{j=k + 1}^{k + \nu_k^X} \xi_i + \sum_{j= 1}^{m}\big( \gamma_{\tau_j^k} - \xi_{\tau_j^k}\big)\,. 
$$
Since $\tau^k_i\geq G^k_i$ for all $i\geq 0$ and on the event $A_{k, \delta}$ it holds that $G_i^k >  i + \lfloor k^{\delta}\rfloor$, we obtain that
\begin{equation*}
\begin{split}
& \mathbb{P}  ( e^{X}_k \um_{A_{k,\delta}} < e^{Y^k}_0 \um_{A_{k, \delta}})
\\
& \le \sum_{m=1}^{\infty} \sum_{\ell = m}^{\infty} \mathbb{P}\Big( G_m^k < \nu^{X}_k, \sum_{j=k+1}^{k+\nu^{X}_k} \xi_j = \sum_{j=1}^{m} \big( \xi_{\tau_j^k} - \gamma_{\tau_j^k} \big), G_m^k = \ell + \lfloor k^{\delta} \rfloor, A_{k, \delta} \Big)
\\
& \leq \sum_{m=1}^{\infty} \sum_{\ell=m}^\infty \sum_{n=\ell+\lfloor k^\delta\rfloor + 1}^{\infty} \mathbb{P}\Big( \nu^{X}_k=n, \sum_{j=k+1}^{k+n} \xi_j = \sum_{j=1}^{m} \big( \xi_{\tau_j^k} - \gamma_{\tau_j^k} \big), G_m^k = \ell + \lfloor k^\delta\rfloor \Big) 
\\
& \leq \sum_{m=1}^{\infty} \sum_{\ell=m}^\infty \sum_{n=\ell+\lfloor k^\delta\rfloor + 1}^{\infty} \mathbb{P}\Big(\sum_{j=k+1}^{k+n}\xi_j = \sum_{j=1}^{m}\big(\xi_{\tau_j^k} - \gamma_{\tau_j^k}\big), G_m^k = \ell + \lfloor k^\delta\rfloor \Big) \,.
\end{split}    
\end{equation*}

Note that 
$$
  \sum_{j=1}^{m}\big(\xi_{\tau_j^k} - \gamma_{\tau_j^k}\big) =  \sum_{j=1}^{m}\xi_{\tau_j^k}  -  \Big[m\mu + \sum_{j=1}^{m}\big(\gamma_{\tau_j^k} - \mu\big)\Big] \,, 
$$
where $\mu$ denote the mean vector of $\gamma$ and  by the Central Limit Theorem, both terms $\sum_{j=1}^{m}\xi_{\tau_j^k} $ and $\sum_{j=1}^{m}\big(\gamma_{\tau_j^k} -\mu) $ are of order $\sqrt{m}$. In view of this, let $\alpha \in (1/2,1)$ be a parameter to be determined later. For $m\geq 1$,  let us define 
\[
D_m:= \left\{x+y: x \in B(0, 2Km^\alpha) \text{ and } y \in B(m\mu, 2Km^\alpha)\right\}\,,
\]
where $B(a, R):=\{x\in \mathbb{Z}^d: \Vert x-a\Vert \leq R\}$ with  $a \in \mathbb{Z}^d$ and $R \in [0,+\infty)$. 

Then we have that
\begin{equation}\label{eq:decomposition}
\begin{split}
& \mathbb{P}\Big(\sum_{j=k+1}^{k+n}\xi_j = \sum_{j=1}^{m} \big( \xi_{\tau_j^k} - \gamma_{\tau_j^k}\big), G_m^k = \ell + \lfloor k^\delta\rfloor \Big) 
=  
\\
& \sum_{z \in D_m} \mathbb{P}\Big(
%\nu^{\rm ERW}_k=n,
\sum_{j=k+1}^{k+n}\xi_j =z, \sum_{j=1}^{m} \big( \xi_{\tau_j^k} - \gamma_{\tau_j^k} \big) = z, G_m^k = \ell + \lfloor k^\delta\rfloor  \Big)
\\
&+ \sum_{z \in D^\complement_m} \mathbb{P} \Big(
%\nu^{\rm ERW}_k=n, 
\sum_{j=k+1}^{k+n}\xi_j = z, \sum_{j=1}^{m} \big( \xi_{\tau_j^k} - \gamma_{\tau_j^k} \big) = z, G_m^k = \ell + \lfloor k^\delta\rfloor \Big)\,.    
\end{split}
\end{equation}

We will analyze the two terms on the RHS of ~\eqref{eq:decomposition} separately. For the  first sum portion it holds that
\begin{equation}\label{eq_sum1<1}
\begin{split}
\sum_{z \in D_m} & \mathbb{P} \Big(
\sum_{j=k+1}^{k+n} \xi_j = z, \sum_{j=1}^{m} \big( \xi_{\tau_j^k} - \gamma_{\tau_j^k} \big) = z, G_m^k = \ell + \lfloor k^\delta\rfloor \Big)
\\
& \leq \sum_{z \in D_m} \mathbb{P}\Big(
\sum_{j=k+1}^{k+n}\xi_j = z \Big) \mathbb{P} ( G_m^k = \ell + \lfloor k^\delta\rfloor) 
% \\
% & \leq |D_m| \frac{C_d}{n^{d/2}} \hat{\mathbb{P}}(A_k, G_m^k = \ell + \lfloor k^\delta\rfloor) 
\\
& \leq (4Km^\alpha)^d \frac{C_d}{n^{d/2}} \mathbb{P}( G_m^k = \ell + \lfloor k^\delta\rfloor)\,,    
\end{split}    
\end{equation}
where, in the second inequality in~\eqref{eq_sum1<1} we use the fact that $|D_m| \le (4Km^{\alpha})^d$ and the Local Central Limit Theorem 
% (see inequality (2.4) in \cite[Theorem 2.1.1]{lawler2010random} and commentary in page 24)  
see, e.g., Inequality (2.8) in \cite[Theorem 2.1.3]{lawler2010random}, with $C_d$ a positive constant. For the second term in~\eqref{eq:decomposition} we have
\begin{equation}\label{eq_sum<2}
\begin{split}
&\sum_{z \in D^c_m} \mathbb{P}\Big(
%\nu^{\rm ERW}_k=n, 
\sum_{j=k+1}^{k+n} \xi_j = z, \sum_{j=1}^{m} \big( \xi_{\tau_j^k} - \gamma_{\tau_j^k} \big) = z, G_m^k = \ell + \lfloor k^\delta\rfloor \Big)
\\
& \le \sum_{z \in D^c_m \cap B(0,2Km)} \mathbb{P}\Big(
%\nu^{\rm ERW}_k=n, 
\sum_{j=k+1}^{k+n} \xi_j = z,\sum_{j=1}^{m}\big( \xi_{\tau_j^k} - \gamma_{\tau_j^k} \big) = z \Big) %, A_k, G_m^k = \ell + \lfloor k^\delta\rfloor \right)
\\
&\leq \mathbb{P}\Big(
%\nu^{\rm ERW}_k=n, 
\sum_{j=k+1}^{k+n} \xi_j \in B(0,2Km) \setminus D_m, \sum_{j=1}^{m} \big( \xi_{\tau_j^k} - \gamma_{\tau_j^k} \big) \in B(0,2Km) \setminus D_m \Big) %, A_k, G^k_m = \ell + \lfloor k^\delta\rfloor \right) 
\\
&\leq \mathbb{P}\Big(
%\nu^{\rm ERW}_k=n, 
\sum_{j=k+1}^{k+n} \xi_j \in B(0,2Km) \setminus D_m \Big)^{1/2} \mathbb{P}\Big( \sum_{j=1}^{m} \big( \xi_{\tau_j^k} - \gamma_{\tau_j^k} \big) \in B(0,2Km) \setminus D_m \Big)^{1/2} \,, \end{split}    
\end{equation}
where, in the first inequality we are using that $\sum_{j=1}^{m} \big( \xi_{\tau_j^k} - \gamma_{\tau_j^k} \big) \in B(0,2Km)$, while   the last inequality follows from Cauchy-Schwarz inequality. Regarding the first term in~\eqref{eq_sum<2}, by using the Local Central Limit Theorem again, it holds that 
\begin{equation}\label{eq_sum<2.1}
\begin{split}
\mathbb{P}\Big(
%\nu^{\rm ERW}_k=n, 
\sum_{j=k+1}^{k+n}\xi_j \in B(0,2Km) \setminus D_m \Big) & = \sum_{z \in D^c_m \cap B(0,2Km)} \mathbb{P}\Big(
%\nu^{\rm ERW}_k=n, 
\sum_{j=k+1}^{k+n} \xi_j = z\Big)
\\
& \leq \frac{C_d}{n^{d/2}} |B(0,2Km)|\,.
\end{split}
\end{equation}
For the second term in~\eqref{eq_sum<2}, let us define $F := B(0,2Km) \setminus D_m$, $G := B(0,2Km) \setminus B(0,2Km^{\alpha})$ and $A := B(0, 2Km) \setminus B(m\mu, 2Km^{\alpha})$. Note that, by the definition of $D_m$, if $x+y \in F$ then $x \in G$ or $y \in A$. We then  obtain
\begin{equation}\label{eq_sum<2.2}
\begin{split}
& \mathbb{P}\Big(\sum_{j=1}^{m} \big(\xi_{\tau_j^k} - \gamma_{\tau_j^k} \big) \in F \Big) \leq \mathbb{P}\Big( \sum_{j=1}^{m} \xi_{\tau^k_j} \in G \Big)  + \mathbb{P}\Big(\sum_{j=1}^{m} \gamma_{\tau^k_j} \in A \Big) %\Big| A_k, G^k_m = \ell + \lfloor k^\delta\rfloor \right)
\\
& = \mathbb{P}\Big(\sum_{j=1}^{m}\xi_{j} \in G \Big) + \hat{\mathbb{P}}\Big(\sum_{j=1}^{m}\gamma_{j} \in A \Big) 
\\
& = \mathbb{P}\Big(\sum_{j=1}^{m}\xi_{j} \in G \Big) + \mathbb{P}\Big( \sum_{j=1}^{m}(\gamma_{j}-\mu) \in B(-m\mu,2Km) \setminus  B(0,2Km^\alpha)\Big) 
% \\
% & \leq 2|B(0,2Km)| \Hat{C}_{d,K} \Big[ e^{-2dK^2m^{2\alpha-1}}\Big(\frac{1}{m^{\frac{d}{2}}} + m^{\frac{7d-2}{2}} + \frac{1}{m^{\frac{d+2}{2}}} \Big)  + \frac{1}{m^{\frac{9d-1}{2}}} \Big]
% \\
% & \leq 2|B(0,2Km)| \Hat{C}_{d,K} \Big(m^{4d} e^{-2dK^2m^{2\alpha-1}} + \frac{1}{m^{\frac{9d-1}{2}}} \Big)\,. 
\\
& 
\leq 2|B(0,2Km)| \Hat{C}_{d,K} \Big[ e^{-2d\zeta K^2m^{2\alpha-1}}\Big(\frac{1}{m^{\frac{d}{2}}} + \frac{(2Km)^{8d}}{m^{4d}m^{\frac{d+2}{2}}} + \frac{1}{m^{\frac{d+2}{2}}} \Big)  + \frac{1}{m^{\frac{9d-1}{2}}} \Big]
\\
& 
\leq 2|B(0,2Km)| \widetilde{C}_{d,K} \Big(m^{4d} e^{-2d \zeta K^2m^{2\alpha-1}} + \frac{1}{m^{\frac{9d-1}{2}}} \Big)\,. 
\end{split}    
\end{equation}
In~\eqref{eq_sum<2.2}, we applied Lemma~\ref{lem: iid} to achieve the first equality. For the first inequality, we utilized the Local Central Limit Theorem, see, e.g., Inequality (2.8) in \cite[Theorem 2.1.3]{lawler2010random} with $k=8d$,
%and the fact that $\Bar{p}_n(x) \le C_d n^{-d/2}\exp(-||x||^2 \zeta d/2n)$), with 
$\zeta >0$ 
and the constant $C_d$ depending on the covariance matrix of $\xi$ (see, e.g., Equations~2.2 and 1.1 and Proposition~1.1.1 (c) in \cite{lawler2010random}).  
 %along with the observation that $(2Km^{\alpha})^2 > m$, which follows since, by definition, $\alpha > 1/2$. The final inequality is derived from the fact that $\alpha < 1$. 
%
Thus, using~\eqref{eq_sum<2.1} and~\eqref{eq_sum<2.2} we obtain that the rightmost term in~\eqref{eq_sum<2} is bounded above by 
\begin{equation}\label{eq:sum<2.3}
\begin{split}
%\sum_{z \in D^\complement_m} &\hat{\mathbb{P}}\left(
%\nu^{\rm ERW}_k=n, 
%\sum_{j=k+1}^{k+n}\xi_j =z, \sum_{j=1}^{m}\left(\xi_{\tau_j^k} - \gamma_{\tau_j^k}\right)=z \right) %, A_k, G_m^k = \ell + \lfloor k^\delta\rfloor  \right)
%\\
\Big[ 2 & |B(0,2Km)| \widetilde{C}_{d,K} \Big(m^{4d} e^{-2d\zeta K^2m^{2\alpha-1}} + m^{-\frac{9d-1}{2}} \Big) \Big]^{1/2} \Big[ \frac{C_d}{n^{d/2}} |B(0,2Km)|\Big]^{1/2}
%\Big(4 & |B(0,2Km)|  \frac{C_d^2}{n^{d/2}}\Big)^{1/2} \Big( 2m^{3d+9} e^{-\frac{d(2Km^\alpha)^2}{2m}} + m^{-(d+8)}\Big)^{1/2} %\hat{\mathbb{P}}\left(A_k, G_m^k = \ell + \lfloor k^\delta\rfloor  \right)
\\
& \le \frac{m^d\Tilde{C}_{d,K}}{n^{d/4}} \Big(m^{4d} e^{-2d\zeta K^2m^{2\alpha-1}} + m^{-\frac{9d-1}{2}} \Big)^{1/2} 
\\
& \le \frac{m^d\Tilde{C}_{d,K}}{n^{d/4}} \Big(m^{2d} e^{-d\zeta K^2 m^{2\alpha-1}} + m^{-\frac{9d-1}{4}} \Big)\,.
%&\leq 2(2Km)^{d/2} \frac{C_d}{n^{d/4}} \Big( 2^{\frac{1}{2}}m^{\frac{3}{2}(d+3)} e^{-\frac{d(2Km^\alpha)^2}{4m}} + m^{-\frac{1}{2}(d+8)}\Big) \,.
% \hat{\mathbb{P}}\left(A_k, G_m^k = \ell + \lfloor k^\delta\rfloor  \right) 
% \\
% &=2^{1/2} (2Km)^d \left(\frac{C_d}{m^{d/2}}\right)^{1/2}  e^{-\frac{d(2Km^\alpha)^2}{4m}} \left(\frac{C_d}{n^{d/2}}\right)^{1/2} 
% %\hat{\mathbb{P}}\left(A_k, G_m^k = \ell + \lfloor k^\delta\rfloor  \right) 
% \,.
\end{split}
\end{equation}
Overall, by~\eqref{eq_sum1<1} and~\eqref{eq:sum<2.3} we obtain that

\begin{equation*}
\begin{split}
& \mathbb{P}  ( e^{X}_k \um_{A_{k,\delta}} < e^{Y}_k \um_{A_{k, \delta}})  \leq\underbrace{\sum_{m=1}^{\infty} \sum_{\ell=m}^\infty  \sum_{n=\ell+ \lfloor k^\delta\rfloor + 1}^{\infty} (4Km^\alpha)^d \frac{C_d}{n^{d/2}} \mathbb{P}(G_m^k = \ell + \lfloor k^\delta\rfloor) }_{SUM_1}
\\
&+\underbrace{\sum_{m=1}^{\infty} \sum_{\ell = m}^\infty \sum_{n = \ell + \lfloor k^\delta\rfloor + 1}^{\infty} \frac{m^d\Tilde{C}_{d,K}}{n^{d/4}} \Big(m^{2d} e^{-d\zeta K^2 m^{2\alpha-1}} + m^{-\frac{9d-1}{4}} \Big) }_{SUM_2}\,. %\hat{\mathbb{P}}\left(A_k, G_m^k = \ell + \lfloor k^\delta\rfloor \right)}_{SUM_2} \,. 
\end{split}
\end{equation*}

To achieve the desired result, it suffices to show that $\mathbb{P} ( e^{X}_k \um_{A_{k,\delta}} < e^{Y^k}_0 \um_{A_{k, \delta}}) < Ck^{-\varepsilon}$ for some $C, \varepsilon > 0$ and all sufficiently large $k$. Therefore, it is enough to bound $SUM_1 + SUM_2$ by $C k^{-\varepsilon}$ for some $\varepsilon > 0$.
%We will prove that $ \sum_{k \ge 1}\mathbb{P} ( e^{X}_k \um_{A_{k,\delta}} < e^{Y^k}_0 \um_{A_{k, \delta}}) < \infty$ by demonstrating that both $SUM_1$ and $SUM_2$ are summable over $k$. 

We bound $SUM_1$ and $SUM_2$ separately by $C k^{-\varepsilon}$, determine for which values of $d$ one obtains $\varepsilon > 0$ in each case, and then choose $d$ so that both conditions are satisfied. We begin with $SUM_1$  for which we obtain that 
\begin{equation*}
\begin{split}
& SUM_1 = \sum_{m=1}^{\infty} (4Km^\alpha)^d \sum_{\ell=m}^\infty \mathbb{P}(G_m^k = \ell + \lfloor k^\delta\rfloor) \sum_{n=\ell+ \lfloor k^\delta\rfloor + 1}^{\infty}  \frac{C_d}{n^{d/2}} 
\\
& \le C_d \sum_{m=1}^{\infty} (4Km^\alpha)^d \sum_{\ell = m}^\infty \mathbb{P}( G_m^k = \ell + \lfloor k^\delta\rfloor) \int_{\ell + \lfloor k^\delta\rfloor}^\infty  \frac{1}{x^{d/2}}dx
\\
& \overset{d>2}{=} \frac{C_d (4K)^d}{d/2-1}\sum_{m=1}^{\infty}    m^{\alpha d}   \sum_{\ell=m}^\infty \mathbb{P} (G_m^k = \ell + \lfloor k^\delta\rfloor) \frac{1}{(\ell + \lfloor k^\delta\rfloor)^{d/2-1}}
\\
& \leq C_{d,K} \sum_{m=1}^{\infty}    m^{\alpha d}   \sum_{\ell=m}^\infty \mathbb{P}(G_m^k = \ell) \frac{1}{(\ell + \lfloor k^\delta\rfloor)^{d/2-1}} 
\\
& \leq C_{d,K} \sum_{m=1}^{\infty}    m^{\alpha d} \,\mathbb{E}\Big[ \frac{1}{( G_m^k + \lfloor k^\delta\rfloor)^{d/2-1}} \Big] \,.
\end{split}    
\end{equation*}
Note that in the second to last inequality we use that $\mathbb{P} (G_m^k = \ell + \lfloor k^\delta\rfloor)\leq \mathbb{P} (G_m^k = \ell)$.
We first note that $G^k_m \ge m$ and also that stochastically dominates $\mathcal{G}_m^k$, where $ \mathcal{G}_m^k = \sum_{j=1}^m H^k_j$, with $\{H^k_j\}_{j\geq 1}$  independent random variables with geometric distribution with parameter {$(k+j)^{-\beta}$} for each $j \ge 1$. Now, we  compute an upper bound for $\mathbb{E}[(G_m^k + \lfloor k^\delta \rfloor)^{1-d/2}]$. 
{ Given $\lambda<1$, we decompose the latter expectation as follows
\begin{equation*}
\begin{split}
&\mathbb{E}[(G_m^k + \lfloor k^\delta \rfloor)^{1-d/2}] = 
\\
& \mathbb{E}\Big[ \frac{1}{( G_m^k + \lfloor k^\delta\rfloor)^{d/2-1}} ; G_m^k < \frac{\lambda m^{1+\beta}}{1+\beta} \Big] +  \mathbb{E}\Big[ \frac{1}{( G_m^k + \lfloor k^\delta\rfloor)^{d/2-1}} ; G_m^k \ge \frac{\lambda m^{1+\beta}}{1+\beta} \Big]
\\
& \leq \frac{1}{( m + \lfloor k^\delta\rfloor)^{d/2-1}} \mathbb{P} \Big[ G_m^k < \frac{\lambda m^{1+\beta}}{1+\beta} \Big] + \frac{(1+\beta)^{d/2-1}}{(\lambda m^{1+\beta} + (1+\beta) \lfloor k^\delta\rfloor)^{d/2-1}}
\\
& \leq \frac{1}{( m + \lfloor k^\delta\rfloor)^{d/2-1}} \mathbb{P} \Big[ \mathcal{G}_m^k < \frac{\lambda m^{1+\beta}}{1+\beta} \Big] + \frac{(1+\beta)^{d/2-1}}{(\lambda m^{1+\beta} + (1+\beta) \lfloor k^\delta\rfloor)^{d/2-1}}\,.
\end{split}
\end{equation*}
Then, using Lemma~\ref{lem:geo_sum_bound} we obtain that  $\PP[\mathcal{G}_m^k < \frac{\lambda m^{1+\beta}}{1+\beta}]\leq e^{-\frac{C_{\lambda}}{1+\beta}m }$. 
Choosing $\alpha=1/2 + 1/(cd)>1/2$ in $SUM_1$, with $c$ is a sufficiently large positive constant, and setting $\rho_{\beta, \lambda}:= \frac{1+\beta}{\lambda}$ and $\eta_{\beta, \lambda}:= \frac{C_\lambda}{1+\beta}$ (both positive),  we obtain  
\begin{equation*}
\begin{split}
& SUM_1  \leq C_{d,K} \Big(\sum_{m=1}^{\infty}    \frac{m^{d/2 + 1/c} e^{-\eta_{\beta, \lambda}m} }{( m + \lfloor k^\delta\rfloor)^{d/2-1}} +  \sum_{m = 1}^{\infty} \frac{\rho_{\beta, \lambda}^{d/2-1} m^{d/2 + 1/c}}{\big(m^{1+\beta} + \rho_{\beta, \lambda} \lfloor k^\delta\rfloor\big)^{d/2-1}} \Big)\,.
\end{split}
\end{equation*}
% The first summation above can be bounded by 
% \begin{align*}
% \sum_{m=1}^{\infty} &   \frac{m^{d/2 + 1/c} e^{-\eta_{\beta, \lambda}m} }{( m + \lfloor k^\delta\rfloor)^{d/2-1}} \leq \int_{\lfloor k^\delta \rfloor}^{\infty} x^{1+ 1/c}e^{-\eta_{\beta, \lambda}x}dx 
% \\
% &\leq  \frac{1}{\eta_{\beta, \lambda}} e^{-\eta_{\beta, \lambda} \lfloor k^\delta \rfloor} \left( \lfloor k^\delta \rfloor^{1+1/c}  + \frac{1+1/c}{\eta_{\beta, \lambda}} \big(\lfloor k^\delta \rfloor^{1/c} + \frac{1}{\eta_{\beta, \lambda}}\big) \right)\,,
% \end{align*}
% which is summable in $k$.

The first summation above can be bounded by
\begin{align*}
\sum_{m=1}^{\infty} &   \frac{m^{d/2 + 1/c} e^{-\eta_{\beta, \lambda}m} }{( m + \lfloor k^\delta\rfloor)^{d/2-1}} \le \frac{1}{\lfloor k^\delta\rfloor^{d/2-1}} \sum_{m=1}^{\infty}  m^{d/2 + 1/c} e^{-\eta_{\beta, \lambda}m} = \frac{C_{\beta, \lambda, c, d}}{\lfloor k^\delta\rfloor^{d/2-1}}\,, 
\end{align*}
which yields a positive exponent whenever 
$d \ge 3$.

The second summation can be bounded as follows: 
\begin{align*}
&\sum_{m = 1}^{\infty} \frac{\rho_{\beta, \lambda}^{d/2-1} m^{d/2 + 1/c}}{\big(m^{1+\beta} + \rho_{\beta, \lambda} \lfloor k^\delta\rfloor\big)^{d/2-1}}  \leq\int_{0}^\infty \frac{\rho_{\beta, \lambda}^{d/2-1} x^{d/2 + 1/c}}{\big(x^{1+\beta} + \rho_{\beta, \lambda} \lfloor k^\delta\rfloor\big)^{d/2-1}} dx
\\
&= \frac{\rho_{\beta, \lambda}^{d/2-1}}{1+\beta}\int_{0}^\infty \frac{ y^{\frac{d/2 + 1/c - \beta}{1+\beta}}}{\big(y + \rho_{\beta, \lambda} \lfloor k^\delta\rfloor\big)^{d/2-1}}   dy \leq \frac{\rho_{\beta, \lambda}^{d/2-1}}{1+\beta} \int_{\rho_{\beta, \lambda}\lfloor k^\delta \rfloor}^{\infty} y^{-\frac{\beta (d/2+1) -1/c}{1+\beta} +1}dy
\\
&\overset{(*)}{=} \frac{\rho_{\beta, \lambda}^{d/2-1}}{ \beta(d/2+1) - 1/c -2(1+\beta)} \frac{1}{\big(\rho_{\beta, \lambda} \cdot \lfloor  k^\delta \rfloor\big)^{\frac{1}{1+\beta} \big( \beta(d/2+1) - 1/c\big) -2} }\,, 
\end{align*}
where in the second equality we used the change of variables $y = x^{1+\beta}$. In $(*)$, we require that 
$2-\frac{\beta}{1+\beta}(d/2 +1) + \frac{1/c}{1+\beta}<0$, which holds whenever $d>2\frac{\beta+2}{\beta}$ and $c$ is sufficiently large. 
Therefore, the second summation is bounded by the RHS of $(*)$. Moreover, the exponent of $k$ is positive provided that   
\begin{align}\label{eq:thres1}
\delta \left( \frac{\beta}{1+\beta}(d/2+1) - \frac{1}{c(1+\beta)} -2 \right)>0\,,  
\end{align}
that is, whenever 
\begin{equation}\label{eq:threshold}
d \ge \lfloor 2 + 4/\beta \rfloor +1 \,,  
\end{equation}
and $c$ sufficiently large.
%we can choose  $\delta$ sufficiently close to $\beta$  and $c$ sufficiency large such that the summability condition in \eqref{eq:thres1} \rv{ and and $d > 2(1 + 1/\delta)$ } are satisfied. 
%
% As far as $SUM_1$ is concerned, we thus obtain 
% \begin{equation}\label{eq:sum1_2}
% \begin{split}
% SUM_1& \leq   \frac{C_{d,K}}{\eta_{\beta, \lambda}} e^{-\eta_{\beta, \lambda} \lfloor k^\delta \rfloor} \left( \lfloor k^\delta \rfloor^{1+1/c}  + \frac{1+1/c}{\eta_{\beta, \lambda}} \big(\lfloor k^\delta \rfloor^{1/c} + \frac{1}{\eta_{\beta, \lambda}}\big) \right)
% \\
% + &\frac{C_{d,K} \,\rho_{\beta, \lambda}^{d/2-1}}{ \beta(d/2+1) - 1/c -2(1+\beta)} \frac{1}{\big(\rho_{\beta, \lambda} \cdot \lfloor  k^\delta \rfloor\big)^{\frac{1}{1+\beta} \big( \beta(d/2+1) - 1/c\big) -2} } \,.
% \end{split}    
% \end{equation}
%

We now compute an upper bound for $SUM_2$.
\begin{equation}\label{eq:sum2_1}
\begin{split}
& SUM_2 = \Tilde{C}_{d,K} \sum_{m=1}^{\infty} m^d \Big( m^{2d} e^{-d\zeta K^2 m^{2\alpha-1}} + \frac{1}{m^{\frac{9d-1}{4}}} \Big) \sum_{\ell = m}^\infty \sum_{n = \ell + \lfloor k^\delta\rfloor + 1}^{\infty} n^{-d/4}
\\
& \le \Tilde{C}_{d,K} \sum_{m=1}^{\infty} m^d \Big( m^{2d} e^{-d\zeta K^2 m^{2\alpha-1}} + \frac{1}{m^{\frac{9d-1}{4}}} \Big) 
\sum_{\ell = m}^\infty \frac{1}{(\lfloor k^\delta\rfloor + \ell)^{\frac{d}{4}-1}}
\\
& \le \Tilde{C}_{d,K} \sum_{m=1}^{\infty} m^d \Big( m^{2d} e^{-d\zeta K^2 m^{2\alpha-1}} + \frac{1}{m^{\frac{9d-1}{4}}} \Big) 
\sum_{\ell = \lfloor k^\delta\rfloor + m}^\infty \frac{1}{\ell^{\frac{d}{4}-1}}
\\
& \le \Tilde{C}_{d,K} \sum_{m=1}^{\infty} m^d \Big( m^{2d} e^{-d\zeta K^2 m^{2\alpha-1}} + \frac{1}{m^{\frac{9d-1}{4}}} \Big)  \frac{1}{(m+ \lfloor k^\delta\rfloor -1)^{\frac{d}{4}-2}}
\\
& \le \frac{C_{d,K}'}{\lfloor k^\delta\rfloor^{d/4-2}} \sum_{m=1}^{\infty} m^d \Big( m^{2d} e^{-d\zeta K^2 m^{2\alpha-1}} + \frac{1}{m^{\frac{9d-1}{4}}} \Big) \le \frac{C_{d,K}''}{\lfloor k^\delta\rfloor^{d/4-2}} \,,
\end{split}    
\end{equation}
where, in the last inequality of~\eqref{eq:sum2_1} we used that $\alpha > 1/2$. 

Observe that we have $SUM_2 \le C''k^{-\delta(d/4 -2)}$. In order for $\delta(d/4 - 2) > 0$, we require $d \ge 9$.

Thus, taking $d \ge \max\{9, \lfloor 2 + 4/\beta \rfloor + 1\}$, the desired result follows.

% $SUM_2$ is summable in $k$ whenever $\delta (d/4-2)>1$. As before, if  $d> 4 \frac{2\beta +1}{\beta}$,  choosing $\delta$ arbitrarily close to $\beta$ we obtain that the summability is assured. Using the latter condition together with \eqref{eq:threshold},  we finally obtain that if  $d> \max\{ 2 \frac{\beta^2 + 3 \beta +1}{\beta^2}, 4 \frac{2\beta +1}{\beta}\}$, then there exists $\delta \in (0,\beta)$ such that 
% \begin{equation*}
% \begin{split}
% & \sum_{k=1}^\infty \mathbb{P}  ( e^{X}_k \um_{A_{k,\delta}} < e^{Y^k}_0 \um_{A_{k, \delta}}) 
%\\
% & \le \sum_{k = 1}^\infty \left(C_{d,K,c}'' \Big( \frac{1}{(\lfloor k^\delta \rfloor)^{(3 -1/c)}} + \frac{1}{\lfloor k^{\delta} \rfloor^{d/6 - 3/2 -2/3c}} \Big) +  \frac{C_{d,K}''}{\lfloor k^\delta\rfloor^{d/4-2}} \right)
% < \infty \,.
% \end{split}
% \end{equation*}
% Note that $2 \frac{\beta^2 + 3 \beta +1}{\beta^2}< 4 \frac{2\beta +1}{\beta}$ for $\beta$ close to $1$, whereas the opposite occurs when $\beta$ is close to $0$.
}
\end{proof}

\section{Sub-ballistic strong law of large numbers for $p_n$-\Nametwo{}}\label{sec:SLLN_X}
In this section, we prove Theorem~\ref{thm:conv-Xbeta}.
Henceforth, let  $X$ be a $p_n$-\Nametwo{} in direction $\ell$ on $\ZZ^d$ with $d\geq 2$, $p_n= \mathcal{C} n^{-\beta} \wedge 1$ and $\beta\in (0,1)$; without loss of generality, we  assume $\CC = 1$, i.e., $p_n=n^{-\beta}$. 

Let us denote by $K^{X}_n$ the set of times from $0$ to $n-1$ in which $X$ visits a site for the first time and becomes excited.  We can write $K^{X}_n$ as
\begin{equation*}
K^{X}_n = \big\{ i \in \{1, 2, \dots, n\} : \um_{E_{i-1}^c \cap \{ U_i \leq i^{-\beta}\}} =1 \big\}\,.  
\end{equation*}

If $\{\psi_i\}_{i \ge 1}$ denotes  the sequence of $\FF$-stopping times  corresponding to the  times the $p_n$-\Nametwo{} visits a new site, then setting $\varphi_i := \psi_i + 1$, we have that
\begin{equation}\label{eq: def Kn}
|K^{X}_n| = \sum_{i=1}^{n} \um_{E_{i-1}^c \cap \{ U_i \leq i^{-\beta}\}} = \sum_{j=1}^{|\Rr_{n-1}^X|} \um_{\{U_{\varphi_j} \leq \varphi_j^{-\beta} \}}\,. 
\end{equation}
When clear from the context,  we shall omit the dependence on $X$ and simply write $|K_n|$ rather than $|K^{X}_n|$. 
The key ingredient in the proof of Theorem~\ref{thm:conv-Xbeta} is the following result.

% \begin{proposition}\label{conv-Kbeta} 
% Let  $X$ be a $p_n$-\Nametwo{} in direction $\ell$ on $\ZZ^d$ with $d\geq 2$, $p_n=n^{-\beta}$ and $\beta \in (0,1)$. 
% Assume that $X$ satisfies Hypothesis~\ref{hyp}. Then, it holds that 
% $$
% \lim_{n\to \infty} \frac{|K_{n}|}{n^{1 - \beta}} = \frac{\pi_d}{1-\beta} \quad \text{ a.s.}\,.
% $$
% \end{proposition}

\begin{proposition}\label{conv-Kbeta} 
Let  $X$ be a $p_n$-\Nametwo{} in direction $\ell$ on $\ZZ^d$ with $d\geq 2$, $p_n=n^{-\beta}$ and $\beta \in (0,1)$. Then, it holds that 
\begin{equation}\label{eq:UppK_n}
\limsup_{n\to \infty} \frac{|K^{X}_n|}{n^{1 - \beta}} \leq \frac{\pi_d}{1-\beta} \,, \quad  \text{ a.s.}\,.    
\end{equation}
Moreover, assuming that $X$ satisfies Hypothesis~\ref{hyp}, then 
\begin{equation}\label{eq:convK_n}
\lim_{n\to \infty} \frac{|K^{X}_n|}{n^{1 - \beta}} = \frac{\pi_d}{1-\beta} \,, \quad \text{ a.s.}\,.
\end{equation}
\end{proposition}

Before proving Proposition~\ref{conv-Kbeta}, we show how the proof of Theorem~\ref{thm:conv-Xbeta} follows from it.

\begin{proof}[Proof of Theorem~\ref{thm:conv-Xbeta}]
Define 
\begin{equation*}
J_{n} := \sum_{i=1}^{n} \um_{E_{i-1}^c \cap \{ U_i \leq   i^{-\beta}\}} (\gamma_i - \xi_i) \,.
\end{equation*}
Then, $X_n$ defined in ~\eqref{xn-incremnto1} can be rewritten as 
\begin{equation*}
X_n  = \sum_{i=1}^n \big( \um_{ E_{i-1} \cup \{ U_i > i^{-\beta}\}} \xi_i + \um_{E_{i-1}^c \cap \{ U_i \leq i^{-\beta}\}} \gamma_i \big) = \sum_{i=1}^n  \xi_i + J_n \,.
\end{equation*}
By Theorem 17.3.II.A in \cite{loeve} (page 250), since $\beta \in (0,1/2)$, it holds that
$$
    \lim_{n\to\infty}\frac{1}{n^{1-\beta}}\sum_{i=1}^{n}\xi_i = 0, \, \text{ a.s.}
$$
Then the convergence of $\big\{\frac{X_{n}}{n^{1 - \beta}}\big\}_{n\ge 1}$ depends on the convergence of 
$$
\frac{J_{n}}{n^{1 - \beta}}\ = \frac{1}{n^{1-\beta}} \sum_{i \in K_{n}}  (\gamma_i - \xi_i)= \frac{|K_{n}|}{n^{1-\beta}} \frac{1}{|K_{n}|}\sum_{i \in K_{n}}  (\gamma_i - \xi_i)\,,
$$
where we recall that $K_n$ denotes the set of times from $0$ to $n-1$ in which $X$ visits a site for the first time and becomes excited. By \eqref{eq:convK_n} in Proposition~\ref{conv-Kbeta}, we have that $\frac{|K_{n}|}{n^{1-\beta}}$ converges almost surely to $\frac{\pi_d}{1-\beta}$. 
Since  the sequence of random vectors $\{\gamma_{\varphi_i} -\xi_{\varphi_i}\}_{i \geq 1}$ is i.i.d. having  the same distribution as $\{\gamma_{i} -\xi_{i}\}_{i \geq 1}$, which is also i.i.d. (see Lemma~\ref{lem: iid}), we can use~\cite[Theorem 8.2 item (iii)]{gut2005probability} and obtain 
\begin{equation*} %\label{eq:prob02}
   \sum_{i \in K_{ n}} \frac{ (\gamma_i - \xi_i)}{|K_{ n }|} = \sum_{i=1}^{|K_ n|} \frac{ (\gamma_{\varphi_i} - \xi_{\varphi_i})}{|K_{ n}|} \xrightarrow[n \to \infty]{} \EE[\gamma_1 -\xi_1] = \EE[\gamma_1] \,, \quad  \text{ a.s.}\,.
\end{equation*}
\end{proof}

Before proving Proposition~\ref{conv-Kbeta} let us introduce some auxiliary results. 
%
% If $\{\psi_i\}_{i \ge 1}$ denotes  the sequence of $\FF$-stopping times  corresponding to the  times the $p_n$-\Nametwo{} visits a new site, then setting $\varphi_i := \psi_i + 1$, we have that
% \begin{equation}\label{eq: def Kn}
% |K_n| = \sum_{i=1}^{n} \um_{E_{i-1}^c \cap \{ U_i \leq i^{-\beta}\}} = \sum_{j=1}^{|\Rr_{n-1}^X|} \um_{\{U_{\varphi_j} \leq \varphi_j^{-\beta} \}}\,. 
% \end{equation}

\begin{lemma}\label{lem:varphi}{Let  $X$ be a $p_n$-\Nametwo{} in direction $\ell$ on $\ZZ^d$ with $d\geq 2$,  $p_n=n^{-\beta}$ and $\beta \in (0,1)$. Then, it holds that
\begin{equation}\label{conv1/phi1}
\limsup_n \frac{1}{n^{1-\beta}}\sum_{j=1}^n \frac{1}{\varphi_j^\beta} \le \frac{\pi_d^\beta}{1-\beta}\,, \quad  \text{ a.s.}\,.
\end{equation}
Moreover, assuming that $X$ satisfies  Hypothesis~\ref{hyp}, then 
\begin{equation}\label{conv1/phi2}
\frac{1}{n^{1-\beta}}\sum_{j=1}^n \frac{1}{\varphi_j^\beta} \xrightarrow[n\to \infty]{} \frac{\pi_d^\beta}{1-\beta}\,, \quad \text{ a.s.}\,.
\end{equation}}
\end{lemma}
\begin{proof} {We will only prove \eqref{conv1/phi2} under Hypothesis~\ref{hyp}, but it is straightforward to obtain \eqref{conv1/phi1} using the upper bound in Proposition \ref{prop:RangeERW} instead of Hypothesis~\ref{hyp}.}

Let us begin observing that by definition of the random times $\varphi_j$ and of the range $\Rr_n^X$,    we have that  $|\Rr^X_{\varphi_j}|-1 \le j \le |\Rr^X_{\varphi_j}|$, for every $j$. Moreover, since $\varphi_j \to + \infty$ a.s. as $j$ tends to infinity, by Hypothesis~\ref{hyp}, it  holds that  $\lim_{j \to \infty}\frac{j}{\varphi_j}= \lim_{j \to \infty}\frac{|\Rr_{\varphi_j}^X|}{\varphi_j} = \pi_d$  almost surely. Then, since
\begin{align*}
    \sum_{j=1}^n \frac{1}{\varphi_j^{\beta}} = \sum_{j=1}^n \Big(\frac{j}{\varphi_j}\Big)^{\beta}\frac{1}{j^\beta}\,, 
    % &=\sum_{j=1}^n \bigg(\Big(\frac{j}{\varphi_j}\Big)^{\beta} - \pi_d^\beta \bigg)\frac{1}{j^\beta} + \pi_d^\beta  \sum_{j=1}^n \frac{1}{j^\beta},
    % \\
    % & =\sum_{j=1}^n \bigg(\bigg(\frac{\Rr^X_{\varphi_j}}{\varphi_j}\bigg)^{\beta} - \pi_d^\beta \bigg)\frac{1}{j^\beta} + \pi_d^\beta  \sum_{j=1}^n \frac{1}{j^\beta}.  
\end{align*}
by 
the Stolz-Cesàro theorem applied to the sequence $a_n=\sum_{j=1}^n \Big(\frac{j}{\varphi_j}\Big)^{\beta}\frac{1}{j^\beta}$ and $b_n=\sum_{j=1}^n \frac{1}{j^\beta}$, the claim follows
using that $\Big(\frac{j}{\varphi_j}\Big)^\beta\xrightarrow[j \to \infty]{} \pi_d^\beta$  almost surely and that $\frac{1}{n^{1-\beta}}\sum_{j=1}^n \frac{1}{j^\beta} \to \frac{1}{1-\beta}$ as $n$ tends to infinity. 
\end{proof}

Let $\{\mathcal{F}_{\psi_k}\}_{k\geq 1}$ denote the filtration associated with the sequence $\{\psi_i\}_{i \ge 1}$ of $\FF$-stopping times. The process
\begin{align}\label{eq:M}
M_k^X := \sum_{j=1}^k \Big[ \um_{\{U_{\varphi_j} \leq \varphi_j^{-\beta} \}} - \frac{1}{\varphi_j^\beta} \Big]\,,
\end{align}
with $M_0^X\equiv 0$,  is a  martingale (with respect to $\{\mathcal{F}_{\psi_k}\}_{k\ge 1}$) with bounded increments (when clear from the context,  we omit to write the dependence on $X$ in the notation).  
As a matter of fact 
\begin{align*}
 \mathbb{E}[M_k^X\mid \mathcal{F}_{\psi_k}]&= \mathbb{E}\Big[ \um_{\{U_{\varphi_k} \leq \varphi_k^{-\beta} \}} - \frac{1}{\varphi_k^\beta}+M^X_{k-1}\mid \mathcal{F}_{\psi_k}\Big]=  M^X_{k-1}\,, 
\\
&\Big\vert \um_{\{U_{\varphi_k} \leq\varphi_k^{-\beta} \}} - \varphi_k^{-\beta}\Big\vert\leq 2\,.
\end{align*}
% %%%%%%%%%% Freedman's inequality%%%%%%%%%%
% {\color{red} \texttt{To be removed or placed in the appendix..}
% \begin{theorem}[Freedman's inequality]
% Let $(\xi_i, \mathcal{F}_{i-1})$ be a sequence of real-valued supermartingale differences (i.e., $\mathbb{E}[\xi_k\mid \mathcal{F}_{k-1}]\leq 0$). Let 
% \[
% S_k := \sum_{i=1}^k \xi_i\,,   \qquad \langle S\rangle_k := \sum_{i=1}^k \mathbb{E}[\xi_i^2 \mid \mathcal{F}_{i-1}]\,. 
% \]
% Assume that there exists $b>0$ such that  $\xi_i\leq b$. Then, for all $x>0$ and $v>0$, 
% \[
% \mathbb{P}\left( S_k \geq x, \langle S\rangle_k\leq v^2 \text{ for some $k$ } \right)\leq \exp\left\{-\frac{x^2}{2(v^2 + x b)}\right\}\,.
% \]
% \end{theorem}

% Note that for martingales Freedman's inequality can be used to obtain a two-sided inequality such as 
% \[
% \mathbb{P}\left( |S_k |\geq x, \langle S\rangle_k\leq v^2 \text{ for some $k$ } \right)\leq 2\exp\left\{-\frac{x^2}{2(v^2 + x b)}\right\}\,.
% \]

% }

The quadratic variation  $\langle M^X\rangle_k$ of the martingale defined in \eqref{eq:M} is given by  
\begin{align*}
\langle M^X\rangle_k = \sum_{j=1}^k \mathbb{E}\Big[ \Big(\um_{\{U_{\varphi_j} \leq \varphi_j^{-\beta} \}} - \frac{1}{\varphi_j^\beta}\Big)^2 \bigm|  \mathcal{F}_{\psi_j} \Big]= \sum_{j=1}^k \Big(1 - \frac{1}{\varphi_j^\beta}\Big)\frac{1}{\varphi_j^{\beta}}\,.  
\end{align*}

\begin{lemma}\label{lem:quadratic-variation}
Let  $X$ be a $p_n$-\Nametwo{} in direction $\ell$ on $\ZZ^d$ with $d\geq 2$,  $p_n=n^{-\beta}$ and $\beta \in (0,1)$. Then,  there exists a positive constant $c$  such that
\[
\mathbb{P}\big(\limsup_n \{\langle M^X\rangle_n \geq c \, n^{1-\beta}\}\big) = 0\,.
\]
\end{lemma}
\begin{proof}
Let us begin observing that $\langle M\rangle_n \leq \sum_{j=1}^n \frac{1}{\varphi_j^{\beta}}$. Thus, it is enough to show that 
\[
\mathbb{P}\Big(\limsup_n \Big\{ \sum_{j=1}^n \frac{1}{\varphi_j^{\beta}}\geq c \, n^{1-\beta}\Big\}\Big) = 0\,.
\]
% {\color{red}\sout{Since, by Lemma~\ref{lem:varphi}, we have that $\frac{1}{n^{1-\beta}}\sum_{j=1}^n \varphi_j^{-\beta}$ converges almost surely to $\pi_d^\beta$} 
This follows directly from \eqref{conv1/phi1} in Lemma~\ref{lem:varphi} if we choose $c > \pi_d^\beta/(1-\beta)$.
\end{proof}

In the next lemma we show that, for all $\beta \in (0,1)$,  $M_n/n^{1-\beta} \to 0$ almost surely. 
\begin{lemma}\label{lem:martingale}
Let  $X$ be a $p_n$-\Nametwo{} in direction $\ell$ on $\ZZ^d$ with $d\geq 2$,  $p_n=n^{-\beta}$ and $\beta \in (0,1)$. Then, it holds that  
\[
\frac{M^X_n}{n^{1-\beta}} \xlongrightarrow[n\to \infty]{} 0\,,  \quad \text{ a.s.}\,.
\] 
\end{lemma}

\begin{proof}
We show that for every $\varepsilon>0$ it holds that
\[
\mathbb{P}\Big(\limsup_n \Big\{ |M_n|\geq \varepsilon\, n^{1-\beta}\Big\}\Big) = 0\,.
\]
Note that, for every $n$, the following inclusion holds
\begin{align*}
\big\{ |M_n|\geq \varepsilon\, n^{1-\beta}\big\} \subseteq \big\{ |M_n|\geq \varepsilon\, n^{1-\beta}, \langle M\rangle_n \leq c \, n^{1-\beta} \big\}    \cup \big\{  \langle M\rangle_n \geq c \, n^{1-\beta} \big\}\,, 
\end{align*}
where, $c$ is a positive constant from Lemma~\ref{lem:quadratic-variation}. 
Thus,
\begin{align*}
 \mathbb{P}\Big(\limsup_n \Big\{ |M_n|\geq \varepsilon\, n^{1-\beta}\Big\}\Big) \leq & \mathbb{P}\Big(\limsup_n \big\{ |M_n|\geq \varepsilon\, n^{1-\beta}, \langle M\rangle_n \leq c \, n^{1-\beta} \big\}\Big)    
 \\
 & + \mathbb{P}\big(\limsup_n \{\langle M\rangle_n \geq c \, n^{1-\beta}\}\big) \,.
\end{align*}
By Lemma~\ref{lem:quadratic-variation}, the second term on the RHS above is equal to $0$. As far as the first term is concerned, we have that for every fix $n$ 

\begin{align*}
  \mathbb{P}&\Big(|M_n|\geq \varepsilon\, n^{1-\beta}, \langle M\rangle_n \leq c \, n^{1-\beta} \Big)\leq \mathbb{P}\Big(|M_k|\geq \varepsilon\, n^{1-\beta}, \langle M\rangle_k \leq c \, n^{1-\beta} \text{ for some $k$}\Big)
 \\
 & \leq 2\exp\left\{-\frac{\varepsilon^2 n^{2(1-\beta)}}{2(cn^{1-\beta} + \varepsilon n^{1-\beta})}\right\} 
  = 2\exp\left\{-\frac{\varepsilon^2 n^{(1-\beta)}}{2(c + \varepsilon)}\right\}\,, 
\end{align*}
where the last inequality follows from  Freedman's inequality \cite{Freedman}. 
\end{proof}

% \begin{corollary}\label{cor:range}
% Under Hypothesis A,  for all $\beta \geq 1/2$, it holds that 
% \[
% \frac{1}{n^{1-\beta}} \Big|
%   \sum_{j=1}^{|\Rr_{n-1}^X|} \um_{\{U_{\varphi_j} \leq \varphi_j^{-\beta} \}} - \sum_{j=1}^{\pi_d n} \um_{\{U_{\varphi_j} \leq \varphi_j^{-\beta} \}}
% \,\Big| \xlongrightarrow[n\to \infty]{} 0\,,  
% \]
% almost surely. 
% \end{corollary}

% \begin{proof}
% Note that
% \[
% \frac{1}{n^{1-\beta}} \Big|
%   \sum_{j=1}^{|\Rr_{n-1}^X|} \um_{\{U_{\varphi_j} \leq \varphi_j^{-\beta} \}} - \sum_{j=1}^{\pi_d n} \um_{\{U_{\varphi_j} \leq \varphi_j^{-\beta} \}} \Big|\,,  
% \]
% is bounded from above by
% \begin{equation}\label{cor:cota1}
% \frac{|M_{|\mathcal{R}^X_{n-1}|}|}{n^{1-\beta}} + \frac{|M_{\pi_d n}|}{n^{1-\beta}} + \frac{1}{n^{1-\beta}}
% \Big|
%   \sum_{j=1}^{|\Rr_{n-1}^X|} \frac{1}{\varphi_j^\beta} - \sum_{j=1}^{\pi_d n} \frac{1}{\varphi_j^\beta}
% \,\Big|\,.
% \end{equation}
% Using Hypothesis A and Lemma \ref{lem:martingale}, we obtain that the first two terms in \eqref{cor:cota1} converge to zero almost surely. The third term is bounded from above by
% \begin{align*}
% \frac{1}{n^{1-\beta}} \sum_{j=|\Rr_{n-1}^X| \wedge \pi_d n}^{|\Rr_{n-1}^X| \vee \pi_d n} \frac{1}{j^\beta} & \le \frac{1}{(1-\beta)n^{1-\beta}}
% \Big(\big||\Rr_{n-1}^X| - \pi_d n\big|+1\Big)^{1-\beta} \\
% & = \frac{1}{1-\beta} \Big( \Big| \frac{|\Rr_{n-1}^X|}{n} - \pi_d \Big| + \frac{1}{n} \Big)^{1-\beta}\,,
% \end{align*}
% which also converges to zero almost surely under Hypothesis A.
% \end{proof}

We are now ready to provide the proof of Proposition~\ref{conv-Kbeta}. 
\begin{proof}[Proof of Proposition~\ref{conv-Kbeta}]
We begin proving \eqref{eq:convK_n}.  
For every $\varepsilon>0$, we have that 
\begin{align*}
&\limsup_n\Big\{\Big|\frac{|K_{n}|}{n^{1 - \beta}} - \frac{\pi_d}{1-\beta}\,\Big| > \varepsilon\Big\}\subseteq  \\ 
&\tag{a}
\limsup_n \Big\{\,\Big|
\sum_{j=1}^{|\Rr_{n-1}^X|} \um_{\{U_{\varphi_j} \leq \varphi_j^{-\beta} \}} - \sum_{j=1}^{\pi_d n} \um_{\{U_{\varphi_j} \leq \varphi_j^{-\beta} \}}
\,\Big| > \frac{\varepsilon}{2} n^{1-\beta}\Big\}  
\\
&\tag{b}
\cup \;\limsup_n \Big\{ \,\Big|
\sum_{j=1}^{\pi_d n} \big(\um_{\{U_{\varphi_j} \leq \varphi_j^{-\beta} \}} - \frac{1}{\varphi_j^\beta}
\big)\,\Big| > \frac{\varepsilon}{4} n^{1-\beta}\Big\}  
\\
&\tag{c}
\cup  \; \limsup_n\Big\{ \,\Big|
\sum_{j=1}^{\pi_d n} \frac{1}{\varphi_j^\beta} - \frac{\pi_d n^{1-\beta}}{1-\beta}
\,\Big| > \frac{\varepsilon}{4} n^{1-\beta}\Big\}\,.  
\end{align*}
The event in (b) has probability zero according to Lemma~\ref{lem:martingale}, while the event in (c) also has probability zero according to Lemma~\ref{lem:varphi}. We are left with showing that the probability of the event in (a) is  also zero. Note that
\[
\frac{1}{n^{1-\beta}} \Big|
  \sum_{j=1}^{|\Rr_{n-1}^X|} \um_{\{U_{\varphi_j} \leq \varphi_j^{-\beta} \}} - \sum_{j=1}^{\pi_d n} \um_{\{U_{\varphi_j} \leq \varphi_j^{-\beta} \}} \Big|\,,  
\]
is bounded from above by
\begin{equation}\label{cor:cota1}
\frac{|M_{|\mathcal{R}^X_{n-1}|}|}{n^{1-\beta}} + \frac{|M_{\pi_d n}|}{n^{1-\beta}} + \frac{1}{n^{1-\beta}}
\Big|
  \sum_{j=1}^{|\Rr_{n-1}^X|} \frac{1}{\varphi_j^\beta} - \sum_{j=1}^{\pi_d n} \frac{1}{\varphi_j^\beta}
\,\Big|\,.
\end{equation}
Using Lemma \ref{lem:martingale} and Hypothesis 1, we will obtain that the first two terms in \eqref{cor:cota1} converge to zero almost surely. For the second term, the result is a direct consequence of Lemma~\ref{lem:martingale}. The first term requires a short additional argument because the martingale is evaluated at the random index $|\mathcal{R}^X_{n-1}|$. Under Hypothesis~\ref{hyp}, we have $|\mathcal{R}^X_{n-1}|/n \to \pi_d$ almost surely, and in particular $|\mathcal{R}^X_{n-1}| \to \infty$ almost surely. Writing
\[
\frac{|M_{|\mathcal{R}^X_{n-1}|}|}{n^{1-\beta}}
\;=\;
\frac{|M_{|\mathcal{R}^X_{n-1}|}|}{|\mathcal{R}^X_{n-1}|^{\,1-\beta}}
\cdot
\left(\frac{|\mathcal{R}^X_{n-1}|}{n}\right)^{1-\beta},
\]
the second factor converges to $\pi_d^{\,1-\beta}$ a.s. and by Lemma~\ref{lem:martingale} the first factor converges to zero almost surely. The third term is bounded from above by
\begin{align*}
\frac{1}{n^{1-\beta}} \sum_{j=|\Rr_{n-1}^X| \wedge \pi_d n}^{|\Rr_{n-1}^X| \vee \pi_d n} \frac{1}{j^\beta} & \le \frac{1}{(1-\beta)n^{1-\beta}}
\Big(\big||\Rr_{n-1}^X| - \pi_d n\big|+1\Big)^{1-\beta} \\
& = \frac{1}{1-\beta} \Big( \Big| \frac{|\Rr_{n-1}^X|}{n} - \pi_d \Big| + \frac{1}{n} \Big)^{1-\beta}\,,
\end{align*}
which also converges to zero almost surely under Hypothesis~\ref{hyp}.

As far as \eqref{eq:UppK_n} is concerned, note that by Proposition~\ref{prop:RangeERW}  for all $\epsilon > 0$ we have that
$$
\limsup_{n\to \infty} \frac{|K_{n}|}{n^{1 - \beta}} \le \limsup_{n\to \infty} \frac{1}{n^{1 - \beta}} \sum_{j=1}^{{\lfloor (\pi_d + \epsilon )n\rfloor}} \um_{\{U_{\varphi_j} \leq \varphi_j^{-\beta} \}}\,,
$$
almost surely and by Lemma~\ref{lem:martingale} it  holds that 
$$
\limsup_{n\to \infty} \frac{1}{n^{1 - \beta}} \sum_{j=1}^{\lfloor (\pi_d +\epsilon)n\rfloor} \um_{\{U_{\varphi_j} \leq \varphi_j^{-\beta} \}} =  \limsup_{n\to \infty} \frac{1}{n^{1 - \beta}} \sum_{j=1}^{\lfloor (\pi_d +\epsilon)n\rfloor} \frac{1}{\varphi_j^\beta}{\le \frac{(\pi_d +\epsilon)^{1-\beta}\pi_d^{\beta}}{1 -\beta}}\,,
$$
almost surely. Last inequality follows from  \eqref{conv1/phi1} in Lemma~\ref{lem:varphi}. Letting $\epsilon\to 0$ we obtain \eqref{eq:UppK_n}. 
\end{proof}

\section{Diffusive rescaling of the $p_n$-\Name}\label{resultados_pn}

% We start this section remembering that if $\{X_n\}_{n \ge 0}$ is a $p_n$-\Name{} in direction $\ell$ we can write it as in~\eqref{xn-incremnto1}, hence we have
% $$
% X_n = \sum_{i=1}^n \big(1_{\{E_{i-1}\}} \xi_i + 1_{\{E_{i-1}^c\}} 1_{\{U_i > p_i \}} \xi_i + 1_{\{E_{i-1}^c\}}1_{\{ U_i \leq p_i\}} \gamma_i \big) \,,
% $$
% where $\{U_i\}_{i \geq 1}$ is a sequence  of i.i.d. random variables with uniform distribution in [0, 1], $E_i = \{ \exists\;  k < i \; \text{ such that }\;  X_k = X_i \}$ for all $i \ge 1$, $\{\xi_i, \FF_i\}_{i \geq 1}$ is an increment of a $d$-martingale with zero mean and $\{\gamma_i, \FF_i\}_{i\geq 1}$ is random vector such that $\EE[\gamma_i \cdot \ell |\FF_{i-1}] \ge \lambda$ .

Let $\{X_n\}_{n \ge 0}$ be a $p_n$-\Name{}. We can rewrite~\eqref{xn-incremnto1} as 
\begin{align}\label{xn-incremento2}
\begin{split}
X_n  & = \sum_{i=1}^n \big( \um_{ E_{i-1} \cup \{ U_i > p_i\}} \xi_i + \um_{E_{i-1}^c \cap \{ U_i \leq p_i\}} \gamma_i \big)
\\
& = \sum_{i=1}^n \big( \xi_i + \um_{E_{i-1}^c \cap \{ U_i \leq p_i\}} (\gamma_i - \xi_i) \big) \,.
\end{split}
\end{align}
%Before we provide the proofs of the main results for the $p_n$-\Name{}, let us point out that, 
Also for the sake of simplicity, we will henceforth work with the \[B_{\cdot}^n := \frac{X_{\lfloor n \cdot \rfloor}}{n^{1/2}}\;.\] instead of its interpolated version $\Hat{B}_{t}^n$ in \eqref{eq:B} which is continuous in $[0, \infty)$.  
More generally, 
in order to simplify writing and notation, if we have a sequence of c\`adl\`ag processes with values in $\mathbb{R}^m$, for $m \ge 2$, of the form $\Sigma^{n}_t = \Sigma_{\lfloor n t \rfloor}$, we denote by $\Hat{\Sigma}^{n}_t$ its linearly interpolated version, i.e., 
\begin{equation*}
\Hat{\Sigma}^{n}_t := \Sigma_{\lfloor n t \rfloor} + (nt - \lfloor nt \rfloor)(\Sigma_{\lfloor nt \rfloor + 1} - \Sigma_{\lfloor nt \rfloor}) \,, \ t\ge 0 \,,
\end{equation*}
which is a random element of $C_{\Rs^m}[0, \infty)$. Moreover by ~\cite[Proposition 10.4, Chapter 3]{ethier2009markov} if we have convergence in distribution in the Skorohod space of $\{\Sigma_{\cdot}^n\}_{n\ge 1}$ to a continuous process, then we also have convergence in distribution in $C_{\Rs^m}[0, \infty)$ of $\{\Hat{\Sigma}_{\cdot}^n\}_{n\geq 1}$ to the same limit. %Still about notation, we can also have $\Sigma^{n}_t = \Sigma_{\lfloor n t \rfloor}$ as a set value function and coherently we set
% \begin{equation*}
% \Hat{\Sigma}^{n}_t := |\Sigma_{\lfloor n t \rfloor}| + (nt - \lfloor nt \rfloor)(|\Sigma_{\lfloor nt \rfloor + 1}| - |\Sigma_{\lfloor nt \rfloor}|) \,, \ t\ge 0 \, .
% \end{equation*}

\begin{remark}\label{rem:conver}
 In the proofs, we will often make use of the following facts:
 \begin{itemize}
     \item[a)] If a sequence of processes converges in probability with respect to the uniform norm in $C_{\Rs^m}[0, T]$ for all $T >0$, then it converges in probability in $C_{\Rs^m}[0, \infty)$ under the metric $\rho$ defined in \eqref{def:rho} (this is a well-known result which can be easily proved).
     \item[b)] If a sequence of processes is tight in $C_{\Rs^m}[0, T]$ for all $T>0$ with the topology of uniform convergence in the compacts, then the sequence is tight in $C_{\Rs^m}[0, \infty)$ (see, e.g., \cite[Theorem 4.10,  Chapter 2]{karatzas2012brownian}). 
     \item[c)] \cite[Theorem 7.3]{billingsley1999probability}: A sequence of processes $\{\Hat{\Sigma}_{\cdot}^n\}_{n\geq 1}$ is tight in $C_{\Rs^m}[0, T]$  if and only if the following two conditions are satisfied:
     \begin{enumerate}[1.]
         \item For each positive $\eta$ there exist  $a$ and $n_0$ such that 
         \[ \PP\big[|\Hat{\Sigma}_{0}^n|\geq a\big]\leq \eta\,, \quad \text{ for all $n\geq n_0$}\,.
         \]
         \item For each positive $\varepsilon$ and $\eta$ there exist $\delta\in (0,1)$ and $n_0$ such that
         \[
         \PP\Big[\sup_{|s-t|\leq\delta}|\Hat{\Sigma}_{s}^n-\Hat{\Sigma}_{t}^n|\geq \varepsilon\Big]\leq \eta\,,  \quad \text{ for all $n\geq n_0$} \text{ and } \forall t \le T\,.
         \]
     \end{enumerate}
 \end{itemize}
\end{remark}

%However for the computations we will do in the next sections, its clear that by Markov's inequality the second sum portion of $B_{t}^n$ converges to zero in probability. Hence, for our purposes, we only need to analyze the asymptotic behave of the first sum portion of $B_{t}^n$. 

\subsection{Case $\beta>1/2$ (easy one); proof of Proposition~\ref{pn-WGERW-Gauss}}\label{prova-pn-WGERW}

\hfill \\
%\subsection{Proof of the convergence in distribution of the $p_n$-\Name*  with $\beta>1/2$}\label{prova-pn-WGERW}

%The $p_n$-\Name* can be written as in~\eqref{xn-incremnto1}. 
Recall that for the $p_n$-\Name*, the variables $\{U_i\}_{i \geq 1}$ are independent of  both sequences $\{\gamma_i\}_{ i\geq 1}$ and $\{\xi_i\}_{i \geq 1}$.

We start stating a result that will be used in the proof of Proposition~\ref{pn-WGERW-Gauss}. %\cm{Due to this following Lemma it will be possible to see that the drift push through time will not make a difference for the process.}

\begin{lemma}\label{lem: tgD}
Let $X$ be a $p_n$-\Name* 
in direction $\ell\in \mathbb{S}^{d-1}$, on $\ZZ^d$ with $d\geq 2$, $p_n= \mathcal{C}n^{-\beta} \wedge 1$ with $\beta > 1/2$. Define  
% Considering the representation~\eqref{xn-incremento2} set 
\begin{equation*}
D_{\lfloor nt \rfloor} := \frac{1}{n^{1/2}} \sum_{i=1}^{\lfloor nt \rfloor} \um_{E_{i-1}^c \cap \{ U_i \leq  \mathcal{C} i^{-\beta}\}} (\gamma_i - \xi_i), \ t\ge 0 \,.
\end{equation*}
Then $\{\Hat{D}^n_\cdot\}_{n\ge 1}$ as a sequence of random elements of $C_{\Rs^d}[0, \infty)$,  converges in probability to the zero function.
\end{lemma}

The proof of Lemma~\ref{lem: tgD} will be postponed at the end of this section.
% \comu{Acho que este parágrafo pode ser removido. A prova é curta e bem claro o que será feito} The main idea behind the proof of Theorem~\ref{pn-WGERW-Gauss} is to use the decomposition~\eqref{xn-incremento2} and to analyze separately the sum portions suitably rescaled. We will see that the sum portion corresponding to the part of  $d$-martingale will converge in distribution and the other will converge to zero in probability. Thus, using Slutsky's Theorem we obtain the desired result. 

\begin{proof}[Proof of Proposition~\ref{pn-WGERW-Gauss}]
Using~\eqref{xn-incremento2} we write 
\begin{equation}\label{p_n-WGERW_incrementos}
    B_t^n = \frac{1}{n^{1/2}}\sum_{i=1}^{\lfloor nt \rfloor} \xi_i + \frac{1}{n^{1/2}} \sum_{i=1}^{\lfloor nt \rfloor} \um_{E_{i-1}^c \cap \{ U_i \leq {\mathcal{C}} i^{-\beta}\}} (\gamma_i - \xi_i) \,.
\end{equation}
%Now we will analyze separately the two portion sums of~\eqref{p_n-WGERW_incrementos}.
By Lemma~\ref{lem: tgD} the linear interpolation of the second term in~\eqref{p_n-WGERW_incrementos}  
%\begin{equation}\label{gamma_i-xi_i_wgerw}\frac{1}{n^{1/2}}\sum_{i=1}^{\lfloor nt \rfloor} 1_{\{E_{i-1}^c\}}1_{\{ U_i \leq i^{-\beta}\}} \gamma_i - 1_{\{E_{i-1}^c\}}1_{\{ U_i \leq i^{-\beta}\}} \xi_i  \to  0 \quad \text{as } n \to \infty\,, %\PP-\text{probability}\;.\end{equation}
converges in probability to the zero function (as a sequence of random elements of $C_{\Rs^d}[0, \infty)$). For the first term in~\eqref{p_n-WGERW_incrementos} we  use~\cite[Theorem 7.1.4, 7.1.1]{ethier2009markov} to obtain that
\begin{equation}\label{xi_i->Z}
    \Big\{ \frac{1}{n^{1/2}}\sum_{i=1}^{\lfloor nt \rfloor} \xi_i \Big\}_{t\ge 0} \xrightarrow[n \to \infty]{\mathcal{D}} \{ Z_{t} \}_{t\ge 0} \,,
\end{equation}
where $Z_{\cdot}$ denotes a Gaussian process with independent increments. Using Slutsky's Theorem (see~\cite[Theorem 11.4]{gut2005probability}) we finish the proof. 
\end{proof}

% \cm{Acho queue este remark poderia ficar depois da prova do lema.}
% \begin{remark}\label{rem_cond_frac}
% As we already point out in Remark~\ref{rem_cond_antes}  the Condition~\ref{condiçao I*} and the sequence $\{p_n\}_{n \ge 1}$ can be more general. Specifically, if we had $\sum_{i=1}^{\lfloor nt \rfloor} p_i \EE[||\gamma_i||] = o(\sqrt{n})$, then it will be possible to prove that the second sum portion in~\eqref{p_n-WGERW_incrementos} goes to zero in probability. Hence we finish the proof with Slutsky's Theorem (Theorem 11.4 from~\cite{gut2005probability}).
% \end{remark}

The proof of Corollary~\ref{pnESRW->BM}  (stating that the rescaled $p_n$-\Nametwo{} converges in distribution to a standard Brownian Motion) follows the proof of Proposition~\ref{pn-WGERW-Gauss} line by line.  The main difference is  in~\eqref{xi_i->Z} where instead of using ~\cite[Theorem 7.1.4]{ethier2009markov}, we can apply Donsker's Theorem, since the $\{\xi_i\}_{i\geq 1}$ corresponding to  $p_n$-\Nametwo{} are i.i.d. with zero-mean vector and finite covariance matrix (see, e.g., \cite[Theorem 8.2]{billingsley1999probability} or~\cite[Theorem 5.1.2]{ethier2009markov}).

\medskip
\begin{proof}[Proof of Lemma~\ref{lem: tgD}.] Without loss of generality, we shall assume $\CC = 1$. In light of Remark~\ref{rem:conver}-$a)$ it suffices to show convergence in $C_{\Rs^d}[0, T]$ for all $T>0$.   Let us then define 
\begin{equation*}
D_{\lfloor n t \rfloor}^{\gamma}:=\frac{1}{n^{1/2}}\sum_{i=1}^{\lfloor n t \rfloor} \um_{E_{i-1}^c}\um_{\{ U_i \leq i^{-\beta}\}} \gamma_i \,, \ t \ge 0,
%\ \text{and} \
% D_{\lfloor n\cdot \rfloor}^{\xi} := \frac{1}{n^{1/2}}\sum_{i=1}^{\lfloor nt \rfloor} 1_{\{E_{i-1}^c\}}1_{\{ U_i \leq i^{-\beta}\}} \xi_i \,. 
\end{equation*}
and analogously $D_{\lfloor n\cdot \rfloor}^{\xi}$ replacing $\gamma_i$ by $\xi_i$ in each term of the sum.
We begin showing  that $\{\Hat{D}_\cdot ^{n,\gamma}\}_{n\ge 1}$ converges in probability to the  zero function in $C_{\Rs^d}[0, T]$ for all $T > 0$ (recall that $\Hat{D}_t ^{n,\gamma}$ is the linearly interpolated version of   $D_{\lfloor nt \rfloor}^{\gamma}$). Note that
\begin{eqnarray}\label{eq: Dgamma}
& & \PP \Big( \sup_{0 \le t \le T} \big\| \Hat{D}^{n,\gamma}_{t} \big\| > \varepsilon  \Big)  \le \PP \Big( \sup_{0 \le t \le T} \sum_{i=1}^{\lfloor nT \rfloor+1}  \big\| \um_{E_{i-1}^c \cap \{ U_i \leq i^{-\beta}\}} \gamma_i \big\| > \varepsilon n^{\frac{1}{2}} \Big) \nonumber
\\
& & \le \PP \Big( \sum_{i=1}^{\lfloor nT \rfloor +1} \big\| \um_{E_{i-1}^c \cap \{ U_i \leq i^{-\beta}\}} \gamma_i \big\| > \varepsilon n^{\frac{1}{2}} \Big) \le  \PP \Big( \sum_{i=1}^{\lfloor nT \rfloor +1} \big\| \um_{ \{ U_i \leq i^{-\beta}\}} \gamma_i \big\| > \varepsilon n^{\frac{1}{2}} \Big) \nonumber
\\
& & \le \frac{1}{n^{1/2}\varepsilon} \sum_{i=1}^{\lfloor nT \rfloor +1} \frac{1}{i^{\beta}} \EE\left[\left\| \gamma_i \right\| \right]  \leq \frac{1}{n^{1/2}\varepsilon}  \sum_{i=1}^{\lfloor nT \rfloor +1} \frac{\EE\left[\left\| \gamma_i \right\| \right]}{i^{\theta}} \times   \frac{1}{i^{\beta-\theta}}  \,.    
\end{eqnarray}  
By Condition~\ref{condiçaoI*}, we know that  for all $i \geq 1$ and all $\theta < \beta - 1/2$ there exists a positive constant $L$ such that 
$\EE\left[\left\| \gamma_i \right\| \right] \le i^{\theta}   L$. 
Going back to~\eqref{eq: Dgamma}, we get
\begin{align*}%\label{eq: Dgamma3}
\begin{split}
\PP & \Big( \sup_{0 \le t \le T} \big\| \Hat{D}^{n,\gamma}_{t} \big\| > \varepsilon  \Big) %\le \frac{1}{n^{1/2}\varepsilon}  \sum_{i=1}^{\lfloor nT \rfloor} \frac{\EE\left[\left\| \gamma_i \right\| \right]}{i^{\theta}} \times   \frac{1}{i^{\beta-\theta}}  \\
\leq \frac{L}{n^{1/2}\varepsilon} \sum_{i=1}^{\lfloor nT \rfloor + 1} \frac{1}{i^{\beta-\theta}} 
\\ 
& \leq \frac{L}{n^{1/2}\varepsilon} \times c' (\lfloor nT \rfloor +1)^{1-\beta+\theta}
 \leq \frac{Lc'}{\varepsilon} \times \frac{\lfloor nT \rfloor + 1}{n} \times \frac{n^{1/2}}{(\lfloor nT \rfloor+1)^{\beta-\theta}}
\end{split}
\end{align*}
for some $c' > 0$, since $\sum_{i=1}^{k} \frac{1}{i^{\beta}} = O(k^{1-\beta})$. Given that $\beta >1/2+\theta$, we obtain that $\{\Hat{D}^{n,\gamma}_{\cdot}\}_{n\ge 1}$ converges in probability (with respect to the uniform metric) to the  zero function for all $T >0$.
%\begin{equation*}
%\frac{\sum_{i=1}^{\lfloor nt \rfloor} 1_{\{E_{i-1}^c\}}1_{\{ U_i \leq i^{-\beta}\}} \gamma_i}{n^{1/2}} \to  0 \quad \text{as } n \to \infty\,, %\PP-\text{probability}\;.    
%\end{equation*}

Using again Condition~\ref{condiçaoI*} and applying the same 
computations as above, we show that $\{\Hat{D}^{n,\xi}_{\cdot}\}_{n\ge 1}$ also converges  in probability (with respect to the uniform metric) to the  zero function for all $T >0$.
%\begin{equation*}\frac{\sum_{i=1}^{\lfloor nt \rfloor} 1_{\{E_{i-1}^c\}}1_{\{ U_i \leq i^{-\beta}\}} \xi_i}{n^{1/2}}  \to  0 \quad \text{as } n \to \infty\,, %\PP-\text{probability}\;.    \end{equation*}
%in probability in the space $C_{\Rs^d}[0, T]$ for all $T >0$.
Therefore the same convergence also holds for $\{\Hat{D}^{n}_{\cdot} = \Hat{D}^{n,\gamma}_{\cdot} - \Hat{D}^{n,\xi}_{\cdot}\}_{n\ge 1}$. 
%
% Since convergence in $C_{\Rs^d}[0, T]$ for all $T >0$ implies convergence in $C_{\Rs^d}[0, \infty)$ under the metric $\rho$, then $D_{\lfloor n\cdot \rfloor}$ converges in probability to zero as a sequence of random elements of $C_{\Rs^d}[0, \infty)$ 
% \com{...HERE!!}
% .
\end{proof}

%\subsection{Proof of the convergence in distribution of the $p_n$-\Nametwo{} with $\beta = 1/2$ and $d=2$.}\label{prova-pn-ERW_d=2}

\subsection{Case $\beta = 1/2$; proof of Theorem~\ref{thm:main-conv}}\label{sec:main-theorem}

\hfill \\

In the following, we present an important auxiliary result that will be useful in the proof of Theorem~\ref{thm:main-conv}. We still state and proof the result for any $\beta \le 1/2$ and we are only going to fix $\beta = 1/2$ in the proof of Theorem~\ref{thm:main-conv}. 

\begin{lemma}\label{Kn-convergence}
Take $d\geq 2$ and $p_n=\mathcal{C}n^{-\beta}\wedge 1$, $n\ge 1$, for $\beta \le 1/2$. Let $|K_n|$  as in~\eqref{eq: def Kn} and consider the corresponding sequence of continuous processes $\{|\Hat{K}^n_{\cdot}|/n^{1-\beta}\}_{n\geq 1}$. Assume that Hypothesis~\ref{hyp} holds. Then, $\{|\Hat{K}^n_{\cdot}|/n^{1-\beta}\}_{n\geq 1}$ converges in probability as random elements of $C_{\Rs}[0, \infty)$ to the deterministic function $$t \mapsto \mathcal{C} \frac{\pi_d \, t^{1-\beta}}{1-\beta}, \ \, t\ge 0\,.$$ 
\end{lemma}

% {\color{red} \begin{lemma}\label{Jntight}
% Let $|K_n|$ be defined as in~\eqref{eq: def Kn} and consider the corresponding sequence of continuous  processes $\{\Hat{K}^n_{\cdot}/n^{1/2}\}_{n\geq 1}$. It holds that
% \begin{enumerate}[i)]
%     \item For $d\geq 2$,  $\{\Hat{K}^n_{\cdot }/n^{1/2}\}_{n\ge 1}$ is tight  in  $C_{\Rs}[0, \infty)$;
%     \item For $d= 2$, $\{\Hat{K}^n_{\cdot}/n^{1/2}\}_{n\geq 1}$ converges in probability as random elements of $C_{\Rs}[0, \infty)$ to the zero function.
% \end{enumerate}
% \end{lemma}
% %
% }

The proof of Lemma~\ref{Kn-convergence} will be postponed to the end of this section.

\begin{proof}[Proof of Theorem~\ref{thm:main-conv}]
% The idea of the proof of Theorem~\ref{pn-ERW-d=2} is similar to the one used in Theorem~\ref{pn-WGERW-Gauss}, however to obtain a version of Lemma~\ref{lem: tgD} in the case $d=2$ and $\beta=1/2$  we will use Proposition~\ref{prop:RangeERW} (see, proof of Lemma~\ref{Jntight}). 
Without loss of generality, we  assume $\CC = 1$, i.e., $p_n=n^{-\beta}$. 
By~\eqref{xn-incremento2} and the definition of $|K_n|$, we can rewrite  the process $B_t^n$ as
\begin{align}\label{p_n-ERW_incrementos_d}
\begin{split}
B_t^n & 
%= \frac{1}{n^{1/2}}\sum_{i=1}^{\lfloor nt \rfloor} \xi_i + \frac{1}{n^{1/2}} \sum_{i=1}^{\lfloor nt \rfloor} \um_{E_{i-1}^c \cap \{ U_i \leq i^{-1/2}\}} (\gamma_i - \xi_i)\\ & 
= \frac{1}{n^{1/2}}\sum_{i=1}^{\lfloor nt \rfloor} \xi_i + \frac{1}{n^{1/2}} \sum_{i \in K_{\lfloor nt \rfloor}}  (\gamma_i - \xi_i) \,.
\end{split}
\end{align}
For the first term in \eqref{p_n-ERW_incrementos_d} we apply  \cite[Theorem 7.1.4]{ethier2009markov}
% Donsker's Theorem (see, e.g., \cite[Theorem 5.1.2 $(c)$]{ethier2009markov}) 
% (\textcolor{brown}{este teorema assume que as coornedadas das $\xi_i$ s\~ao n\~ao correlacionadas. No Stroock tambem pede que a matrix de covariancia seja a identiade)} 
and  obtain that %the first sum portion of~\eqref{p_n-ERW_incrementos_d=2} converges in distribution in $C_{\Rs^2}[0, \infty)$ to a Brownian Motion, i.e.,  
\begin{equation}\label{xi_i->W}
    \Big\{ \frac{1}{n^{1/2}}\sum_{i=1}^{\lfloor nt \rfloor} \xi_i \Big\}_{t\ge 0} \xrightarrow[n \to \infty]{\mathcal{D}} \{ W_{t} \}_{t\ge 0} \,,
\end{equation}
where $W_{\cdot}$ is a Brownian Motion  with zero-mean vector and covariance matrix $\EE[\xi_1 \xi_1^T]$. 
Now let us prove that 
\begin{equation}\label{eq:prob}
     \frac{1}{n^{1/2}} \sum_{i \in K_{\lfloor nt \rfloor}}  (\gamma_i - \xi_i)  \xrightarrow[n \to \infty]{} 2\pi_d \sqrt{t}\EE[\gamma_1] \, \text{ in probability}.
\end{equation}
First, observe that, for $s<t$
\begin{align*}
    \Big|\Big|\frac{1}{n^{1/2}} \sum_{i \in K_{\lfloor ns \rfloor}}  (\gamma_i - \xi_i) - \frac{1}{n^{1/2}} \sum_{i \in K_{\lfloor nt \rfloor}}  (\gamma_i - \xi_i) \Big|\Big| &= \frac{1}{n^{1/2}}\Big|\Big|\sum_{i = |K_{\lfloor ns\rfloor}|+ 1}^{|K_{\lfloor nt\rfloor}|} (\gamma_{\varphi_i} - \xi_{\varphi_i})\Big|\Big|\\
    &\le \frac{2K}{n^{1/2}} ||K_{\lfloor nt\rfloor} - K_{\lfloor ns\rfloor}||\,.
\end{align*}
Then by Lemma \ref{Kn-convergence} and Remark~\ref{rem:conver} (item $c$), we have that $\big\{\frac{1}{n^{1/2}} \sum_{i \in K_{\lfloor nt \rfloor}}  (\gamma_i - \xi_i): t\ge 0 \big\}$ is tight. Note that
\begin{equation*}
    \frac{1}{n^{1/2}} \sum_{i \in K_{\lfloor nt \rfloor}}  (\gamma_i - \xi_i) = 
\frac{|K_{\lfloor nt \rfloor}|}{n^{1/2}} \sum_{i \in K_{\lfloor nt \rfloor}} \frac{ (\gamma_i - \xi_i)}{|K_{\lfloor nt \rfloor}|}\, .
\end{equation*}
Under the assumption that Hypothesis~\ref{hyp} holds, Lemma~\ref{Kn-convergence} implies that  $|K_{\lfloor nt \rfloor}| \to \infty$ as $n \to \infty$ almost surely. Moreover, the sequence of random vectors $\{\gamma_{\varphi_i} -\xi_{\varphi_i}\}_{i \geq 1}$ is i.i.d. and has the same distribution of $\{\gamma_{i} -\xi_{i}\}_{i \geq 1}$, which is i.i.d. too (see Lemma~\ref{lem: iid}). Therefore,
\begin{equation}\label{eq:prob0}
   \sum_{i \in K_{\lfloor nt \rfloor}} \frac{ (\gamma_i - \xi_i)}{|K_{\lfloor nt \rfloor}|} = \sum_{i=1}^{|K_{\lfloor nt \rfloor}|} \frac{ (\gamma_{\varphi_i} - \xi_{\varphi_i})}{|K_{\lfloor nt \rfloor}|} \xrightarrow[n \to \infty]{} \EE[\gamma_1] \quad \text{ a.s.}\,.
\end{equation}
Then by Lemma~\ref{Kn-convergence} and Slutsky's Theorem (see, e.g., \cite[Theorem 11.4]{gut2005probability}) we obtain that
\begin{align*}
    &\Big(\frac{1}{n^{1/2}} \sum_{i \in K_{\lfloor nt_1 \rfloor}}  (\gamma_i - \xi_i),\dots,\frac{1}{n^{1/2}} \sum_{i \in K_{\lfloor nt_m \rfloor}}  (\gamma_i - \xi_i)\Big)
    \\
    & \qquad \qquad \qquad \qquad \xrightarrow[n \to \infty]{} \big(2\pi_d\sqrt{t_1}\EE[\gamma_1],\dots,2\pi_d\sqrt{t_1}\EE[\gamma_1]\big)\,,
\end{align*}
in probability, for any $0\le t_1,\le \dots\le t_m $. Hence we have that the process
$\big\{\frac{1}{n^{1/2}} \sum_{i \in K_{\lfloor nt \rfloor}}  (\gamma_i - \xi_i): t\ge 0 \big\}$
converges in distribution to the constant process $\{2\pi_d\sqrt{t}\mathbb{E}[\gamma_1]: t\ge 0\}$, which is equivalent to the convergence in probability. Then by \eqref{xi_i->W} and \eqref{eq:prob} we obtain Theorem~\ref{thm:main-conv}.
\end{proof}

\begin{remark}\label{main-ub1}
Recalling the discussion at the end of  Section~\ref{sec:model}, and assuming $\mathcal{C}=1$,  we now show how Equation~\eqref{prop:d2-23} can be proved without assuming Hypothesis~\ref{hyp}.
From \eqref{p_n-ERW_incrementos_d} and \eqref{eq:prob0} and observing that $\EE[\gamma_1\cdot \ell] \ge \lambda > 0$, Equation~\eqref{prop:d2-23} follows if, instead of the complete convergence result of Lemma~\ref{Kn-convergence}, we have that $\{\Hat{K}^n_{\cdot}/n^{1/2}\}_{n\geq 1}$ is tight and for any limit point
$\mathcal{H}_{\cdot}$
\begin{align}\label{main-ubeq1}
& 
\PP \left[\forall t \in [0,\infty):  \mathcal{H}_t \le 2\sqrt{t} \pi_d
\right] = 1\,.
\end{align}
This can be checked directly from the proof of Lemma~\ref{Kn-convergence} using only the upper bound in Proposition \ref{prop:RangeERW}, see Remark \ref{main-ub2}.
\end{remark}

%\com{I think the proof below may be stated and adjusted  to hold for $\beta \in (0,1)$, i.e., $\frac{|K_{\lfloor n t\rfloor}|}{n^{1-\beta}} \xrightarrow[n \to \infty]{} \frac{\pi_d}{1-\beta}  t^{1-\beta}$ and adjust the proof of tightness for the process $\left\{\Hat{K}^n_{\cdot}/n^{1-\beta}\right\}_{n\geq 1}$. This could give us that for $\beta <1/2$, under Hypothesis~\ref{hyp}, we have that $B_t^n$ converges in distribution to the function $t\mapsto \frac{\pi_d}{1-\beta}  t^{1-\beta}$ ...} 

\begin{proof}[Proof of Lemma~\ref{Kn-convergence}] 
Without loss of generality, we  assume $\CC = 1$, i.e., $p_n=n^{-\beta}$. By \eqref{eq:convK_n} in Proposition~\ref{conv-Kbeta}, for all $t> 0$ we have that 
\begin{align*}
\frac{|K_{\lfloor n t\rfloor}|}{n^{1-\beta}} & \xrightarrow[n \to \infty]{} \frac{\pi_d \, t^{1-\beta}}{1-\beta} \,, 
\end{align*}
in probability, i.e., the one-dimensional distributions converge in probability to a constant. Therefore the joint distributions corresponding to times $t_1, \ldots, t_m$ also converge in probability and we obtain the convergence in the sense of the finite-dimensional distributions for both the cádlág version  $|K_{\lfloor n \cdot\rfloor}|/n^{1-\beta}$ as well as for the linearly interpolated version $\Hat{K}^n_{\cdot}/n^{1-\beta}$.
% By Corollary 3.1 we already have the convergence of the
% finite-dimensional distributions. 
% we have the convergence of the one-dimensional distributions in probability to a constant. Since each one-dimensional random variable converges in probability to a constant, the joint convergence in probability follows naturally, which in turn implies convergence in distribution.

By~\cite[Theorem 7.1]{billingsley1999probability},  to prove the claim it only remains to prove that the sequence of processes $\{\Hat{K}^n_{\cdot}/n^{1-\beta}\}_{n\geq 1}$ is tight in $C_{\Rs^d}[0, \infty)$. 
By  Remark~\ref{rem:conver}-$b)$  it suffices to show tightness in $C_{\Rs^d}[0, T]$ for all $T > 0$ and this is equivalent to show that the sequence of processes $\left\{\Hat{K}^n_{\cdot}/n^{1-\beta}\right\}_{n\geq 1}$  satisfies  the two conditions in  Remark~\ref{rem:conver}-$c)$.
Since $\Hat{K}^n_{0}/n^{1-\beta}\equiv 0$ for all $n \ge 1$ and therefore the first condition in Remark~\ref{rem:conver}-$c)$ is satisfied.
To prove that $\{\Hat{K}^n_\cdot /n^{1-\beta}\}_{n\ge 1}$ satisfies the second condition in Remark~\ref{rem:conver}-$c)$  we  use~\cite[Corollary on page 83]{billingsley1999probability} which states that the second condition of ~\cite[Theorem 7.3]{billingsley1999probability} holds if, for every positive $\varepsilon$ and $\eta$, there exists a $\phi \in (0,1)$, and an integer $n_0$ such that
\begin{equation}\label{eq:tight}
\frac{1}{\phi} \, \PP \Big[ \sup_{t \le s \le t + \phi} \big\|\Hat{K}^n_{s} - \Hat{K}^n_{t} \big\|  \ge \varepsilon n^{1-\beta} \Big]  \le \eta \quad \forall n \ge n_0 \text{ and } \forall t \le T\,.
\end{equation}
Note that 
\begin{align*}
\PP &\Big[ \sup_{t \le s \le t + \phi} \big\|\Hat{K}^n_{s} - \Hat{K}^n_{t} \big\|  \ge \varepsilon n^{1-\beta} \Big]  \le \PP \Big[ \sum_{i = \lfloor nt \rfloor }^{\lfloor n(t + \phi) \rfloor +1} \um_{\big\{U_i \le i^{-\beta} \big\}} \ge \varepsilon n^{1-\beta} \Big]
\\
&
\le \exp\Big(-\varepsilon n^{1-\beta}\Big)\EE\Big[\exp\Big( \sum_{i = \lfloor nt \rfloor}^{\lfloor n(t + \phi)\rfloor+1} \um_{ \{U_i \le i^{-\beta}\}} \Big) \Big]\,, 
%& & \le \exp\Big(\frac{-\varepsilon n^{\frac{1}{2}}}{2K}\Big) \prod_{i = \lfloor nt \rfloor + 1}^{\lfloor n(t + \phi)\rfloor} \EE\left[\exp\left(  \um_{ \{U_i \le i^{-\frac{1}{2}}\}} \right) \right] \,,    
\end{align*}
where in the last inequality we have used exponential Markov's inequality.
Continuing the computation  we obtain that
\begin{eqnarray}\label{eq: PnA}
\lefteqn{\frac{1}{\phi} \PP \Big[ \sup_{t \le s \le t + \phi} \big\|\Hat{K}^n_{s} - \Hat{K}^n_{t} \big\|  \ge \varepsilon \Big] \le \frac{1}{\phi} e^{-\varepsilon n^{1-\beta}} \prod_{i = \lfloor nt \rfloor}^{\lfloor n(t + \phi)\rfloor + 1} \EE\left[\exp\left(  \um_{ \{U_i \le i^{-\beta} \}} \right) \right]} \nonumber \\
& & = \frac{1}{\phi} e^{-\varepsilon n^{1-\beta}} \prod_{i = \lfloor nt \rfloor}^{\lfloor n(t + \phi)\rfloor + 1} \left(1+ \frac{e-1}{i^{\beta}} \right) \le \frac{1}{\phi} e^{-\varepsilon n^{1-\beta}} \prod_{i = \lfloor nt \rfloor}^{\lfloor n(t + \phi)\rfloor + 1} \exp\left(\frac{e-1}{i^{\beta}} \right) 
\nonumber \\
& & \le \frac{1}{\phi} \exp(-\varepsilon n^{1-\beta})\exp\left(\frac{(e-1) \big((n(t + \phi))^{(1-\beta)} - (nt)^{1-\beta} + 2\big)}{1-\beta} \right) \,,  \nonumber
\end{eqnarray}
where the last inequality above follows from noticing that 
\begin{align*}
\sum_{i = \lfloor nt \rfloor}^{\lfloor n(t + \phi)\rfloor + 1}\frac{1}{i^{\beta}} & \le \int_{nt - 1}^{n(t+\phi) + 1} x^{-\beta} dx \le \frac{(n(t + \phi))^{(1-\beta)} - (nt)^{1-\beta} + 2}{1-\beta} \,.
\end{align*}
Therefore, in order to show that \eqref{eq:tight}  holds, it remains to show that for every positive $\varepsilon$ and $\eta$,  there exists a $\phi \in (0,1)$, and an integer $n_0$ such that
\begin{equation}\label{exp<et}
\frac{1}{\phi} \exp(-\varepsilon n^{1-\beta})\exp\left(\frac{(e-1) \big(n^{1-\beta} ((t + \phi)^{(1-\beta)} - t^{1-\beta}) + 2\big)}{1-\beta} \right) \le \eta
\end{equation}
for all $n \ge n_0$.
We accomplish this by choosing $\phi \in (0,1)$ sufficiently small such that $((t + \phi)^{(1-\beta)} - t^{1-\beta}) < \epsilon(1-\beta)/2(e-1)$, for all $t \in [0,T]$. Then, for every $\eta>0$, choosing $n$ sufficiently large, we obtain that~\eqref{exp<et} is satisfied. 
% One can see that, since we have $\phi \in (0,1)$,  for all $\hat{\varepsilon} > 0$, there exists a $\phi' > \phi$ such that $|\sqrt{t+\phi'} - \sqrt{t}|< \hat{\varepsilon}$. Now we can choose $\hat{\varepsilon} = c/4(e-1)$ and we obtain, for a large enough $n$, that~\eqref{exp<eta} is fulfilled for all $\eta$.
Consequently, we have that~\eqref{eq:tight} is satisfied and thus $\{\Hat{K}^n_{\cdot}\}_{n \ge 1}$ is tight in $C_{\Rs^d}[0,T]$. 
\end{proof}

\begin{remark}\label{main-ub2}
Recall Remark \ref{main-ub1}. Using only the upper bound in Proposition \ref{prop:RangeERW}, we have that $\{\Hat{K}^n_{\cdot}/n^{1/2}\}_{n\geq 1}$ is tight exactly as in the proof of Lemma~\ref{Kn-convergence} above. To check that for every limit point $\mathcal{H}_{\cdot}$ of $\{\Hat{K}^n_{\cdot}/n^{1/2}\}_{n\geq 1}$
\eqref{main-ubeq1} holds, it is enough to use \eqref{eq:UppK_n} instead of \eqref{eq:convK_n} in Proposition~\ref{conv-Kbeta}. 
\end{remark}

\begin{remark} 
Consider $p_n=n^{-\beta}$, $n\ge 1$, for $\beta <1/2$ and suppose that Hypothesis~\ref{hyp} holds. Using Lemma~\ref{Kn-convergence} and proceeding as in the proof of Theorem~\ref{thm:main-conv}, we can show that $\big\{X_{\lfloor n t\rfloor}/n^{1-\beta}\big\}_{t\ge 0}$ in distribution to the deterministic function $t\mapsto \frac{\pi_d}{1-\beta}  t^{1-\beta}$. This reinforces the result of Theorem~\ref{thm:conv-Xbeta}. We leave the details to the reader.
\end{remark}

\medskip 
\paragraph{\bf Acknowledgements:}
We are grateful to Augusto Quadros Teixeira, Christophe Gallesco, Guilherme Ost, Luiz Renato Fontes, and Maria Eulalia Vares for their helpful feedback and suggestions.
% R.B.A.~was supported by CAPES, G.V.~was supported by CNPq (grant 307938/2022-0) and FAPERJ (grant E-26/202.636/2019) and G.I.~was supported by FAPERJ (grant E-26/210.516/2024).

\medskip

\paragraph{\bf Conflict of interest:} The authors have no competing interests to declare that are relevant to the content of this article.

\medskip

\paragraph{\bf Data Availability:} Data sharing is not applicable to this article as no datasets were generated or analyzed during the current study.

\medskip

\appendix

\setcounter{tocdepth}{\value{tocdepth}}%
  \addtocontents{toc}{\protect\setcounter{tocdepth}{1}}% 
  
\section{}\label{sec:appendixC}

\begin{lemma}\label{lem: iid}
Let $\{\phi_n\}_{n \ge 1}$ be a sequence of i.i.d. random vectors on $\ZZ^d$, with $d \ge 2$ and $\{\kappa_n\}_{n \ge 1}$ an increasing sequence of $\FF_n$-predictable times defined on 
a probability space $(\Omega, \FF, \PP)$, where $\{\FF_n\}_{n\ge 1}$ is the filtration $\FF_0 = \{\emptyset, \Omega\}$, $\FF_n=\sigma(\phi_1, \dots, \phi_n)$, $n\ge 1$ and $\kappa_n < \infty$ for all $n$. Then we have that $\{\phi_{\kappa_n}\}_{n \ge 1}$ is i.i.d. and moreover $\phi_{\kappa_n}$ has the same distribution of $\phi_1$ for every $n \ge 1$. 
\end{lemma}
\begin{proof}
We begin showing that for any $n \ge 1$, $\phi_{\kappa_n}$ has the same distribution of $\phi_1$.
Let $A$ be a subset of   $\ZZ^d$ and fix a $j \ge 1$. Then we have that
\begin{equation}\label{eq: iid1}
\begin{split}
& \PP[\phi_{\kappa_j} \in A | \FF_{\kappa_{j}-1}] = \sum_{n=1}^{\infty} \PP[\{ \phi_{\kappa_j} \in A \} \cap \{\kappa_j = n\}| \FF_{\kappa_{j}-1}]
\\
& = \sum_{n=1}^{\infty} \PP[\phi_{\kappa_j} \in A |\{\kappa_j = n\}, \FF_{\kappa_{j}-1}] \PP[\kappa_j = n| \FF_{\kappa_{j}-1}] \\
& = \sum_{n=1}^{\infty} \PP[\phi_n \in A ]\PP[\kappa_j = n| \FF_{\kappa_{j}-1}] = \PP[\phi_1 \in A ] \underbrace{\sum_{n=1}^{\infty} \PP[\kappa_j = n| \FF_{\kappa_{j}-1}]}_{=1}\,.
\end{split}    
\end{equation}
The third equality in~\eqref{eq: iid1} follows from  the fact that $\{\kappa_n\}_{n \ge 1}$ is an increasing sequence of $\FF_n$-predictable times, $\{\phi_n\}_{n \ge 1}$ is i.i.d. and so $\phi_n$ is independent of $\FF_{n-1}$ for all $n \ge 1$.
It remains to prove independence. We will only prove pairwise independence, but it is straightforward to generalize the proof by induction and we leave the details to the reader.  For $B$ and $D$ subsets of $\ZZ^d$ and $j, i \in \mathbb{N}$ such that $j > i$ we have that 
\begin{equation}
\label{eq: iid2}
\begin{split}
\PP[&\{\phi_{\kappa_i} \in B\} \cap \{\phi_{\kappa_j} \in D\}] = 
\\ 
&= \sum_{n=1}^{\infty} \sum_{m > n} \PP[\{\{\phi_{\kappa_i} \in B\} \cap \{\kappa_i = n\}\} \cap \{\{\phi_{\kappa_j} \in D\} \cap \{\kappa_j = m\}\}] 
\\
&= \sum_{n=1}^{\infty} \sum_{m > n} \PP[\phi_{\kappa_j} \in D | \kappa_i = n, \kappa_j = m, \phi_{n} \in B]\PP[\kappa_i = n, \kappa_j = m, \phi_{\kappa_i} \in B] 
\\
&= \sum_{n=1}^{\infty} \sum_{m > n} \PP[\phi_m \in D]\PP[\kappa_i = n, \kappa_j = m, \phi_{\kappa_i} \in B]\,.
\end{split}
\end{equation}
The last equality in~\eqref{eq: iid2} follows from the fact that the event $\{\kappa_i = n, \kappa_j = m, \phi_{n} \in B\}$, for $m>n$, is $\FF_{m-1}$ measurable, since $\{\kappa_n\}_{n \ge 1}$ is an increasing sequence of $\FF_n$-predictable times.

Since the sequence $\{ \phi_n\}_{n \ge 1}$ is i.i.d., we can continue the computation in~\eqref{eq: iid2} and obtain that
\begin{equation*}
\begin{split}
\PP[\{\phi_{\kappa_i} \in B\} \cap \{\phi_{\kappa_j} \in D\}] & = \PP[\phi_1 \in D] \sum_{n=1}^{\infty} \sum_{m \ge n} \PP[\kappa_j = m,\kappa_i = n, \phi_{\kappa_i} \in B]
\\
& = \PP[\phi_1 \in D] \sum_{n=1}^{\infty}  \PP[\kappa_i = n, \phi_{\kappa_i} \in B] 
\\
& = \PP[\phi_1 \in D] \sum_{n=1}^{\infty} \PP[\phi_{n} \in B]\PP[\kappa_i = n]
\\
& = \PP[\phi_1 \in D] \PP[\phi_1 \in B] = \PP[\phi_{\kappa_i} \in D]\PP[\phi_{\kappa_j} \in B]\,,
\end{split}    
\end{equation*}
concluding the proof.
\end{proof}

Now we present the proof of Lemma~\ref{lemma_1}.
\begin{proof}[Proof of Lemma~\ref{lemma_1}] 
For simplicity, we denote $b_k = \lfloor k^{\delta}\rfloor$. Recall that $\delta<\beta$. From the definition of the set $A_{k, \delta, \beta}$ in \eqref{eq:set_A} we have that  there exists a positive constant $c_1$, { which does not depend on $\beta$,} such that
\begin{equation}\label{eq:n1}
\mathbb{P} (A_{k, \delta, \beta}^c) \leq 1 - \Big(1 - \frac{1}{k^{\beta}}\Big)^{b_k} \le c_1 \frac{b_k}{k^{\beta}} \,.
\end{equation}
Hence,  $\lim_{n\to\infty}\frac{1}{n}\sum_{k = 1}^{n}\mathbb{P}(A_{k,\delta, \beta}^c) = 0$ for all $\delta \in (0,\beta)$.
Then, to prove the result it is enough to show that 
\begin{equation*}%\label{eq_1}
\lim_{n\to\infty} \frac{1}{n} \sum_{k = 1}^n \Big(\um_{A_{k, \delta, \beta}^c} - \mathbb{P}(A_{k, \delta, \beta}^c)\Big) = 0, \, \text{ a.s.}\,.
\end{equation*}
By Borel-Cantelli's Lemma the latter holds if can show that 
\begin{equation}\label{eq_2}
\sum_{n\ge 1}\mathbb{E}\left[\left(\frac{1}{n}\sum_{k = 1}^n\Big(\um_{A_{k, \delta,\beta}^c} - \mathbb{P}\left(A_{k, \delta, \beta}^c\right)\Big)\right)^2\right] < \infty \,,
\end{equation}
holds for any $\delta \in (0,\beta)$. 
To avoid clutter, henceforth we will denote $A_k := A_{k, \delta, \beta}$.  Note that
\begin{equation}\label{eq:divide}
\begin{split}
\mathbb{E}\Big[\Big(\frac{1}{n}\sum_{k = 1}^{n}[\um_{A_k^c} - \mathbb{P}(A_k^c)]\Big)^2\Big] &=
 \frac{1}{n^2}\sum_{k=1}^{n}\mathbb{E}\left[(\um_{A_k^c} - \mathbb{P}(A_k^c))^2\right]
 \\
+& \frac{2}{n^2}\sum_{k = 1}^{n - 1}\sum_{m = k + 1}^{n}\mathbb{E}\left[(\um_{A_k^c} - \mathbb{P}(A_k^c))(\um_{A_m^c} - \mathbb{P}(A_m^c))\right] \,.
\end{split}
\end{equation}
For the first on the RHS of \eqref{eq:divide}, it holds that 
\begin{equation}\label{eq:n2}
\begin{split}
& \frac{1}{n^2}\sum_{k=1}^{n}\mathbb{E}\left[\left(\um_{A_k^c} - \mathbb{P}(A_k^c)\right)^2\right]  \le \frac{1}{n^2}\sum_{k = 1}^{n}\mathbb{P}(A_k^c) 
% =\frac{1}{n^2}\sum_{k = 1}^n\left[1 - \left(1 - \frac{1}{\sqrt{k}}\right)^{b_k}\right] 
% \\
% & 
\le \frac{c_1}{n^2}\sum_{k = 1}^n\frac{k^{\delta}}{k^{\beta}} \le \frac{c_2}{n^{1 + {\beta} - \delta}}\,,
\end{split}
\end{equation}
where,  $c_2$ is a positive constant and in the second inequality we used~\eqref{eq:n1}. 
Since $\delta \in (0, \beta)$, by~\eqref{eq:n2} we obtain
\begin{equation*}
\sum_{n = 1}^{\infty}\frac{1}{n^2} \sum_{k=1}^{n}\mathbb{E}\left[\left(\um_{A_k^c} - \mathbb{P}(A_k^c)\right)^2\right] < \infty\,.
\end{equation*}

For the second sum on the RHS of \eqref{eq:divide}, since the $\{U_i\}_{i \ge 1}$ are independent,   we have that  $\mathbb{E}[(\um_{A_k^c} - \mathbb{P}(A_k^c))(\um_{A_m^c} - \mathbb{P}(A_m^c))] = 0$ for $m > k + b_k$. Therefore,  we obtain that 
\begin{align*}
& \frac{1}{n^2}\sum_{k = 1}^{n - 1}\sum_{m = k + 1}^{n}\mathbb{E}\left[\left(\um_{A_k^c} - \mathbb{P}(A_k^c)\right)\left(\um_{A_m^c} - \mathbb{P}(A_m^c)\right)\right] = 
\\
& \frac{1}{n^2}\sum_{k = 1}^{n - 1}\;\;\sum_{m = k + 1}^{n\wedge(k + b_k)}\Big(\mathbb{P}\left(A_k^c\cap A_m^c\right) - \mathbb{P}(A_k^c)\mathbb{P}(A_m^c)\Big) \,.
\end{align*}

Recalling the definition of $A_k$ in \eqref{eq:set_A}, for $m\leq k +b_k$ we have that 
% \begin{align*}
% A_k\cap A_m^c & = \left(\bigcap_{n = 1}^{b_k}\left\{U_{n + k} > \frac{1}{\sqrt{k}}\right\}\right)\bigcap\left(\bigcup_{l = 1}^{b_m}\left\{U_{l + m} \le \frac{1}{\sqrt{m}}\right\}\right)
% \\
% & = \left(\bigcap_{n = 1}^{b_k}\left\{U_{n + k} > \frac{1}{\sqrt{k}}\right\}\right)\bigcap\left(\bigcup_{l =  b_k + k - m + 1}^{b_m}\left\{U_{l + m} \le \frac{1}{\sqrt{m}}\right\}\right)\,.
% \end{align*}
%
\begin{align*}
A_k\cap A_m^c & = \bigcap_{n = 0}^{b_k}\left\{U_{n + k} > \frac{1}{(k+n)^{{\beta}}}\right\}\cap \;\bigcup_{l = 0}^{b_m}\left\{U_{l + m} \le \frac{1}{(m+l)^{{\beta}}}\right\}
\\
& = \bigcap_{n = 0}^{b_k}\left\{U_{n + k} > \frac{1}{(k+n)^{{\beta}}}\right\}\cap \;\bigcup_{l = b_k + k - m + 1}^{b_m}\left\{U_{l + m} \le \frac{1}{(m+l)^{{\beta}}}\right\}\,.
\end{align*}
Setting  $E_{k,m, \beta} := \displaystyle\bigcup_{l =  b_k + k - m + 1}^{b_m}\left\{U_{l + m} \le \frac{1}{(m+l)^{{\beta}}}\right\}$, and noticing that $$E_{k,m, \beta} \subseteq A_{m}^c \, ,$$ together with the fact that 
\begin{equation*}
\mathbb{P}\left(A_k^c\cap A_m^c\right) = \mathbb{P}\left(A_m^c\right) - \mathbb{P}\left(A_k\cap A_m^c\right) = \mathbb{P}\left(A_m^c\right) - \mathbb{P}\left(A_k\right)\mathbb{P}\left(E_{k,m, \beta}\right)\,,
\end{equation*}
we obtain that 
\begin{align*}
& \mathbb{P}\left(A_k^c\cap A_m^c\right) - \mathbb{P}\left(A_k^c\right)\mathbb{P}\left(A_m^c\right)\\
&= \mathbb{P}\left(A_m^c\right) - \mathbb{P}\left(A_k\right)\mathbb{P}\left(E_{k,m,\beta}\right) - \mathbb{P}\left(A_k^c\right)\mathbb{P}\left(A_m^c\right) 
\\
& 
% = \mathbb{P}(A_m^c)\mathbb{P}\left(A_k\right) - \mathbb{P}\left(A_k\right)\mathbb{P}\left(E_{k,m}\right)
= \mathbb{P}\left(A_k\right)\left[\mathbb{P}\left(A_m^c\right) - \mathbb{P}\left(E_{k,m,\beta}\right)\right] \\
&= \mathbb{P}\left(A_k\right)\mathbb{P}\left(A_m^c\setminus E_{k,m,\beta}\right)
\\
& = \mathbb{P}\left(A_k\right) \mathbb{P}\Big(\bigcup_{l = 0}^{b_k + k - m}\Big\{U_{l + m} \le \frac{1}{(m+l)^{{\beta}}}\Big\}\Big)
\\
& \leq \mathbb{P}\left(A_k\right)\Big[ 1 - \left(1 - \frac{1}{m^{\beta}}\right)^{b_k + k - m} \Big] 
% \le 1 - \left(1 - \frac{1}{m^{\beta}}\right)^{b_k + k - m} 
{\leq c_3 \frac{b_m + k - m}{m^{\beta}}} \,,
\end{align*}

% Since $x \ge 1 - e^{-x}$ for all $x \ge 0$, the RHS above can be bounded  by \begin{align*}
%  1 - \left(1 - \frac{1}{\sqrt{m}}\right)^{b_k + k - m} &= 1 - e^{(b_m + k - m)\log(1 - \frac{1}{\sqrt{m}})}
% \\
% & \le - (b_m + k - m)\log(1 - \frac{1}{\sqrt{m}}) \\
% &=-\frac{b_m + k - m}{\sqrt{m}}\log(1 - \frac{1}{\sqrt{m}})^{\sqrt{m}}
% \\
% & \le c_3 \frac{b_m + k - m}{\sqrt{m}}\,, 
% \end{align*}
for some constant $c_3 > 0$. Then
\begin{align*}
& \sum_{m = k + 1}^{n\wedge(k + b_k)}\Big(\mathbb{P}\left(A_k^c\cap A_m^c\right) - \mathbb{P}(A_k^c)\mathbb{P}(A_m^c)\Big) \le c_3 \sum_{m = k + 1}^{k + b_k}\frac{b_k + k - m}{m^{\beta}} 
\\
& =  c_3 \sum_{m =1}^{b_k}\frac{b_k - m}{(m + k)^{\beta}} 
=  \frac{c_3}{{b_k}^{\beta}}\sum_{m =1}^{b_k}\frac{1 - \frac{m}{b_k}}{\big(\frac{m}{b_k} + \frac{k}{b_k}\big)^{\beta}}\frac{1}{b_k} \le  \frac{c_3}{{b_k}^{\beta}}\sum_{m =1}^{b_k}\frac{1 - \frac{m}{b_k}}{\big(\frac{m}{b_k} + 1\big)^{\beta}}\frac{1}{b_k}\,.
\end{align*}
Using that 
\begin{equation*}
\lim_{k\to\infty}\sum_{m =1}^{b_k}\frac{1 - \frac{m}{b_k}}{\big(\frac{m}{b_k} + 1\big)^{\beta}}\frac{1}{b_k} = \int_{0}^{1}\frac{1 - x}{(1 + x)^{\beta}}dx\,, 
\end{equation*}
we conclude that  there exists some positive constant $c_4$ such that 
\begin{equation*}
\sum_{m = k + 1}^{n\wedge(k + b_k)}\Big(\mathbb{P}\left(A_k^c\cap A_m^c\right) - \mathbb{P}(A_k^c)\mathbb{P}(A_m^c)\Big) \le c_4\frac{1}{{b_k}^{\beta}}\,,
\end{equation*}
which, in turn,  implies that 
\begin{equation*}
\begin{split}
& \sum_{n = 1}^{\infty}\frac{1}{n^2}\sum_{k = 1}^{n - 1}\sum_{m = k + 1}^{n}\mathbb{E}\left[\left(\um_{A_k^c} - \mathbb{P}(A_k^c)\right)\left(\um_{A_m^c} - \mathbb{P}(A_m^c)\right)\right]
\\
&  \le c_4 \sum_{n = 1}^{\infty} \frac{1}{n^2}\sum_{k = 1}^{n}\frac{1}{{b_k}^{\beta}}  \le c_5 \sum_{n = 1}^{\infty} \frac{1}{n^{1 + \delta{\beta}}} < \infty \,.
\end{split}
\end{equation*}
\end{proof}

The next result is a general result for a sum of geometric random variables which will be important in the proof of Proposition~\ref{prop:RangeERW_lower}.

\begin{lemma}\label{lem:geo_sum_bound}
For every $k\geq 1$, let $\{H^{k}_j\}_{j \ge 1}$ be a sequence of independent random variables such that for each $ j \ge 1$, $H^{k}_j\sim {\rm Geo} (1/(k+j)^\beta)$ with $\beta >0$. Then, for every  $\lambda<  1$ there exists a constant $C_\lambda>0$ depending on $\lambda$ only such that for any $m\geq 1$ it holds that 
\begin{equation*}
  P\Big(\sum_{j=1}^m H^{k}_j\leq \lambda \, \frac{m^{1+\beta}}{1+\beta}\Big)\leq \exp{-\frac{C_{\lambda}}{1+\beta} \cdot m}\,.  
\end{equation*}
\end{lemma}
\begin{proof}
The proof relies on Theorem~3.1 in \cite{JANSON20181} and a simple estimate on $E\big[\sum_{j=1}^m H^{k}_j\big]$. Firstly, by Theorem~3.1 in \cite{JANSON20181},  for every $k\geq 1$ and every $\lambda\in (0,1)$, it holds that 
\[
  P\Big(\sum_{j=1}^m H^{k}_j\leq \lambda E\big[\sum_{j=1}^m H^{k}_j\big] \Big)\leq \exp{- \frac{1}{(k +m)^{\beta}} E\big[\sum_{j=1}^m H^{k}_j\big] (\lambda -1 - \log \lambda)}\,.
\]
Taking $C_\lambda:=\lambda -1 - \log \lambda>0$ for $\lambda \in (0,1)$ and using that, 
\[
E\big[\sum_{j=1}^m H^{k}_j\big] \geq \int_0^m (k+x)^{\beta}dx=\frac{(k+m)^{1+\beta}}{1+\beta}\,, 
% \leq \int_0^{m+1} (k+x)^{\beta}dx
% = \frac{(k+m+1)^{1+\beta}}{1+\beta}
\]
the claim easily follows.
\end{proof}

% \rv{We now present a useful auxiliary result. Informally, it shows that if the events $(E_k)$ do not occur too often on average, then their long-run frequency goes to zero almost surely.}

% \begin{lemma}\label{lemm1}
% Let $(E_k)_{k \ge 1}$ be a sequence of events and $S_n = \sum_{k=1}^n \um_{E_k}$. If 
% \begin{equation*}
%    \sum_{k = 1}^{\infty} \frac{\PP[E_k]}{k} < \infty\,, 
% \end{equation*}
% then $S_n/n \to 0$ a.s..
% \end{lemma}
% \begin{proof}
% Fix $\eta > 0$ and set $N_r := 2^r$ for $r \ge 1$. By Markov's inequality,
% \[
% \PP\!\big(S_{N_r} > \eta\, N_r\big) \le \frac{\EE[S_{N_r}]}{\eta\, N_r} = \frac{1}{\eta\, 2^r} \sum_{k=1}^{2^r} \PP(E_k).
% \]
% Summing over $r$ and interchanging the order of summation (justified by Tonelli's theorem and  nonnegativity),
% \[
% \sum_{r=1}^\infty \PP\!\big(S_{N_r} > \eta\, N_r\big) \le \frac{1}{\eta} \sum_{r=1}^\infty 2^{-r} \sum_{k=1}^{2^r} \PP(E_k) = \frac{1}{\eta} \sum_{k=1}^\infty \PP(E_k) \sum_{\substack{r \ge 1 \\ 2^r \ge k}} 2^{-r}.
% \]
% For $2^{m-1} < k \le 2^m$, the inner geometric sum equals $\sum_{r=m}^\infty 2^{-r} = 2^{-m+1} \le 2/k$. Therefore,
% \[
% \sum_{r=1}^\infty \PP\!\big(S_{N_r} > \eta\, N_r\big) \le \frac{2}{\eta} \sum_{k=1}^\infty \frac{\PP(E_k)}{k} < \infty.
% \]
% By the Borel--Cantelli lemma, $S_{N_r} / N_r \to 0$ almost surely. For $n$, if $2^r \le n < 2^{r+1}$, then by monotonicity of $S_n$,
% \[
% 0 \le \frac{S_n}{n} \le \frac{S_{2^{r+1}}}{2^r} = 2 \cdot \frac{S_{2^{r+1}}}{2^{r+1}} \to 0 \quad \text{a.s.}
% \]    
% \end{proof}

% \input{Appendix-MainTheorem-with-Restriction}

\bibliographystyle{abbrv}
\bibliography{referencias}
 
 \end{document}